\documentclass[12pt]{article}
\usepackage{graphicx,psfrag,epsf}
\usepackage{enumerate}
\usepackage{natbib}
\usepackage{url} 
\usepackage{amsmath,amssymb}
\usepackage{dsfont}
\usepackage{epstopdf}
\usepackage[hang]{subfigure}
\usepackage{booktabs}
\usepackage{textcomp}
\usepackage{ulem}
\usepackage{color}
\usepackage{threeparttable}
\usepackage{comment}
\usepackage{float}
\usepackage{indentfirst}
\usepackage{algorithm,algorithmic}
\usepackage{tikz-cd}
\usepackage{tikz}
\usepackage{ntheorem}
\usepackage{newfloat} 
\usepackage{authblk}
\usepackage{enumitem}

\DeclareFloatingEnvironment[fileext=loa, name=example]{example1}
\DeclareFloatingEnvironment[fileext=lob, name=Example]{example2}
\DeclareFloatingEnvironment[fileext=loc, name=Example]{example3}
\DeclareFloatingEnvironment[fileext=lod, name=Example]{example4}

\newcommand{\blind}{1}

\newcounter{regime}
\newcommand{\regime}[1]{\refstepcounter{regime}\textbf{(Regime \theregime}\label{#1})}

\newcounter{mystep}   
\renewcommand{\themystep}{\arabic{mystep}} 

\newcommand{\mystep}[1]{
    \refstepcounter{mystep}
    \textbf{Step \themystep.}\label{#1}
}

\newtheorem{coro}{Corollary}
\newtheorem{theorem}{Theorem}

\newtheorem{lemma}{Lemma}
\newtheorem{definition}{Definition}
\newtheorem{prop}{Proposition}

\newtheorem{condition}{Condition}
\newtheorem{remark}{Remark}

\newtheorem*{proof}{\it Proof.}

\def\vec{\mathrm{vec}}

\def\T{\mathrm{T}}
\def\tr{\mathrm{tr}}

\def\E{\mathbb E}
\def\var{\mathrm{var}}

\def\T{\mathrm T}
\def\var{\mathrm{var}}
\def\cov{\mathrm{cov}}
\def\pr{\mathbb P}
\def\vec{\mathrm{vec}}

\def\diag{\mathrm{diag}}
\def\sumk{\sum_{k=1}^K}
\def\sumh{\sum_{h=1}^H}
\def\sumik{\sum_{i\in\Ik }}
\def\sumjk{\sum_{j\in\Ik }}
\def\nk{n_{[k]}}

\def\sumijk{\sum_{i,j\in \Ik }}

\def\Ik{\mathcal I_{[k] }}
\def\prodk{\prod_{k=1}^K}
\def\sumin{\sum_{i=1}^n}
\def\pik{w_{[k]}}

\def\Rb{\mathcal D_{[b^\dagger]}}
\def\sumRb{\sum_{(i,j)\in \Rb}}
\def\tauk{\tau_{[k]}}

\def\Mktwo{M_{[k]}^2}
\def\Mkthree{M_{[k]}^3}

\newtheorem{example}{Example}

\usepackage{hyperref}
\hypersetup{
    colorlinks=true,
    urlcolor=blue,
    citecolor=blue
}

\begin{document}

\def\spacingset#1{\renewcommand{\baselinestretch}
{#1}\small\normalsize} \spacingset{1}

\if1\blind
{
  \title{\bf Stratified Permutational Berry--Esseen Bounds and Their Applications to Statistics}

\author[1]{Pengfei Tian}
\author[2]{Fan Yang\textsuperscript{*}}
\author[3]{Peng Ding\textsuperscript{*}}

\affil[1]{\small Qiuzhen College, Tsinghua University,  \texttt{\href{mailto:tpf24@mails.tsinghua.edu.cn}{tpf24@mails.tsinghua.edu.cn}}}
\affil[2]{\small Yau Mathematical Sciences Center, Tsinghua University,  \texttt{\href{mailto:yangfan1987@tsinghua.edu.cn}{yangfan1987@tsinghua.edu.cn}}}
\affil[3]{\small Department of
Statistics, University of California, Berkeley, \texttt{\href{mailto:pengdingpku@berkeley.edu}{pengdingpku@berkeley.edu}}}

\footnotetext[1]{Correspondence should be addressed to Fan Yang and Peng Ding}
  
    \date{}
  \maketitle
}\fi
	
\if0\blind
{
	\title{\bf Stratified Permutational Berry--Esseen Bounds and Their Applications to Statistics}
	\maketitle
}\fi

\bigskip

\begin{abstract}
The stratified linear permutation statistic arises in various statistics problems, including stratified and post-stratified survey sampling, stratified and post-stratified experiments, conditional permutation tests, etc. Although we can derive the Berry--Esseen bounds for the stratified linear permutation statistic based on existing bounds for the non-stratified statistics, those bounds are not sharp, and moreover, this strategy does not work in general settings with heterogeneous strata with varying sizes. We first use Stein's method to obtain a unified stratified permutational Berry--Esseen bound that can accommodate heterogeneous strata. We then apply the bound to various statistics problems, leading to stronger theoretical quantifications and thereby facilitating statistical inference in those problems. 
\end{abstract}


\newpage

\spacingset{1.9} 

\section{Introduction to Stratified Permutational Berry--Esseen Bounds}\label{sec:intro}
\subsection{Stratified Linear Permutation Statistic and Overview of the Main Result}
The stratified linear permutation statistic plays a crucial role in various statistical problems. 
Consider $n$ units divided into $K$ strata $\{\Ik\}_{k=1}^K$, with the $k$-th stratum containing $|\Ik|=\nk$ units and $\sum_{k=1}^{K} \nk = n.$ A permutation is a bijection from $\{1,\ldots,n\}$ to itself. Under the stratified structure, we consider the permutation set $\Pi_K$ whose elements are the permutations that preserve the units within strata. Given a fixed $n$ by $n$ matrix $A=[a_{ij}]$ and a random permutation $\pi \sim \textup{Uniform}(\Pi_K)$, i.e., $ \pr(\pi = \pi_0) = \{{n_{[1]}! n_{[2]}! \cdots n_{[K]}!}\}^{-1}$ for any $ \pi_0 \in \Pi_K$, we define the stratified linear permutation statistic as:
\begin{align}\label{eq:matrix permutation}
    W_{A,\pi} = \sum_{i=1}^n a_{i\pi(i)} = \sum_{k=1}^{K} \sum_{i \in \Ik } a_{i\pi(i)}.
\end{align}
The randomness in $W_{A,\pi}$ arises from the randomness in $\pi$. This stratified permutation statistic covers a large number of examples, which will be reviewed in Section \ref{subsec:example}. We will prove the following result on the error in normal approximation of $W_{A, \pi}$:
\begin{theorem}[informal version]\label{thm:informal}
Standardize the stratified linear permutation statistic $W_{A,\pi}$ defined in \eqref{eq:matrix permutation} to have mean $0$ and variance $1$. There exists a universal constant $C$, such that
$$\sup_{t\in \mathbb R}|\pr(W_{A,\pi}\le t)-\Phi(t)|\le C\sumk \sumijk |a_{ij}|^3/\nk.$$
\end{theorem}
We will present the details of the standardization of $W_{A,\pi}$ in Section \ref{subsec:basic result and notation} and present the formal version of Theorem \ref{theorem:beb} in Section \ref{sec:BEB Stein}. In the next subsection, we will review some motivating examples in statistics.

\subsection{Motivating Examples}\label{subsec:example}
 Many classic statistics used in stratified sampling and stratified experiments can be represented in the form of \eqref{eq:matrix permutation}. We review four basic motivating examples in this subsection and will present more examples in Section \ref{sec:permute_test}. Define the index set $\Ik=\{\sum_{k'=1}^{k-1} n_{[k']}+1,\ldots,\sum_{k'=1}^{k}n_{[k']}\}$ and the stratum proportion $\pik=\nk/n$ for all the examples below. We sometimes also re-index the $k$-th stratum as $\Ik=\{{[k]1}, {[k]2},\ldots,{[k]\nk}\}$. Let $\bold 1_{m}$ and $\bold 0_m$ denote, respectively, vectors of all ones and zeros in $\mathbb R^m$. Let $A=\diag\{A_{[k]}\}_{k=1,\ldots,K}$ denote a block-diagonal matrix with $K$ submatrices, which has specific forms in the examples below.

\begin{example}[Stratified sampling]\label{example:sampling}
In stratified sampling \citep{cochran1977,bickel1984asymptotic}, $n$ units are divided into $K$ strata, with $\nk$ units in stratum $k$, $k=1,...,K$. In each stratum $k$, a sample of size $n_{[k]1}\leq n_{[k]}$ is randomly drawn without replacement. Define $n_{[k]0}=\nk- n_{[k]1}$. Given the outcomes $\{Y_i\}_{i=1}^n$, the target parameter is $\gamma=\sumk \pik \bar Y_{[k]}$, where $\bar Y_{[k]}=\nk^{-1}\sumik Y_i$ is the $k$-th stratum mean. Let $Z_i$ denote the inclusion indicator for unit $i$ and $\hat Y_{[k]}=n_{[k]1}^{-1}\sumik Z_i Y_i$ denote the $k$-th stratum sample mean. An unbiased estimator for $\gamma$ is $\hat \gamma=\sumk \pik\hat Y_{[k]}$. We can rewrite $\hat \gamma$ as $W_{A_{\textup{ss}},\pi}$ with
\begin{align*}
    A_{\textup{ss}}=\diag \left\{ w_{[k]}n_{[k]1}^{-1}\left(\begin{matrix}
   Y_{{[k]1}} \bold{1}^\T_{n_{[k]1}} &\bold{0}^\T_{n_{[k]0}}\\
   \vdots &\vdots\\
   Y_{{[k]\nk}} \bold{1}^\T_{n_{[k]1}}  &\bold{0}^\T _{n_{[k]0}}
\end{matrix}\right)\right\}_{k=1,\ldots,K},
        \end{align*}
        where ``$\textup{ss}$'' stands for ``stratified sampling''. 
\end{example}

\begin{example}[Stratified experiment]\label{example:exp}
In a stratified experiment 
\citetext{\citealp{imbens2015causal}, \citealp{ding2024first}}, $n$ units are divided into $K$ strata, with $\nk$ units in stratum $k$, $k=1,...,K$. 
In the $k$-th stratum, we randomly assign $n_{[k]1}$ units to the treatment group and $n_{[k]0}$ units to the control group, with $n_{[k]1}+n_{[k]0}=\nk.$ Let $Z_i$ denote the binary treatment condition for unit $i$, with $z=1$ and $z=0$ representing the treatment condition and the control condition, respectively. Let $Y_i(z)$ denote the potential outcome under treatment condition $z$, $z=0, 1$. The target parameter is the average treatment effect $\tau=\sumk \pik\tauk$, where $\tauk=\nk^{-1}\sumik \{Y_i(1)-Y_i(0)\}$ is the $k$-th stratum-specific average treatment effect. Let $\hat\tau_{[k]}=n_{[k]1}^{-1}\sumik Z_i Y_i-n_{[k]0}^{-1}\sumik(1-Z_i)Y_i$. The stratified estimator is $\hat\tau=\sumk \pik \hat\tau_{[k]}$. We can rewrite $\hat\tau$ as $W_{A_{\textup{sre}}, \pi}$ with
    \begin{align*}
    A_{\textup{sre}}=\diag \left\{  \pik \left(\begin{matrix}
  n_{[k]1}^{-1} Y_{{[k]1}}(1)\bold{1}_{n_{[k]1}}^\T  &-n_{[k]0}^{-1} Y_{{[k]1}}(0)\bold{1}_{n_{[k]0}}^\T \\
   \vdots &\vdots\\
    n_{[k]1}^{-1} Y_{{[k]\nk}}(1)\bold{1}_{n_{[k]1}}^\T  &-n_{[k]0}^{-1} Y_{{[k]\nk}}(0)\bold{1}_{n_{[k]0}}^\T 
        \end{matrix}\right)\right\}_{k=1,\ldots,K},
        \end{align*}
where ``$\textup{sre}$'' stands for ``stratified randomized experiment''. 
\end{example}

\begin{example}[Post-stratified sampling]\label{example:ps exp}
Consider a simple random sample from a population with $n$ units $\{Y_i\}_{i=1}^n$ where the parameter of interest is $\gamma =n^{-1}\sum_{i=1}^n Y_i.$ Post-stratification \citep{cochran1977,holt1979post} stratifies the samples by a predictive discrete covariate $X$ with $K$ levels after the sampling stage, with the goal of improving estimation efficiency. Let $Z_i$ denote the sampling indicator and $n_{[k]1}$ denote the number of sampled units from stratum $k$. Different from stratified sampling in Example \ref{example:sampling}, $n_{[k]1}$ is random in a post-stratified design. Consider the post-stratification estimator $\hat\gamma_{\textup{ps}}=\sumk \pik \hat{Y}_{[k]}$, where the mathematical symbols are the same as those in Example \ref{example:sampling} with $\Ik=\{i:X_i=k\}$. Let $\mathcal D_1=\{(n_{[1]1},\ldots, n_{[K]1}): \prodk n_{[k]1}\neq 0\}$ denote the event that all strata have at least one unit. Conditional on $\bold U_1=(n_{[1]1},\ldots,n_{[K]1})\in\mathcal D_1$, the estimator $\hat\gamma_{\textup{ps}}$ has the same distribution as the one in Example \ref{example:sampling} with the corresponding matrix denoted as $A_{\textup{pss}}=A_{\textup{ss}}(\bold U_1)$, where ``$\textup{pss}$'' stands for ``post-stratified sampling''. Theorem \ref{theorem:beb} is the theoretical foundation for analyzing  $\hat\gamma_{\textup{ps}}$ not only conditionally but also unconditionally.
\end{example}

\begin{example}[Post-stratified experiment]\label{example:ps}
Analogous to Example \ref{example:ps exp}, post-stratification can also be applied to randomized experiments \citep{miratrix2013}. Consider a completely randomized experiment of $n$ units which are post-stratified into $K$ strata based on a $K$-level covariate $X$.
We are interested in the average treatment effect $\tau=n^{-1} \sum_{i=1}^n \{Y_i(1)-Y_i(0)\}$. The post-stratification estimator equals 
$\hat \tau_{\textup{ps}}=\sumk  \pik \hat\tau_{[k]}$, where the mathematical symbols are the same as those in Example \ref{example:exp} with $\Ik=\{i:X_i=k\}$. Although the estimator shares the same form as the one in Example \ref{example:exp}, $n_{[k]1}$ and $n_{[k]0}$ are random in this example. 
Let $\mathcal D=\{(n_{[1]1},n_{[1]0},\ldots, n_{[K]1},n_{[K]0}): \prodk n_{[k]1}n_{[k]0}\neq 0\}$ denote the event that each stratum has at least one treatment unit and one control unit. Conditional on $\bold U=(n_{[1]1},n_{[1]0},\ldots, n_{[K]1},n_{[K]0})\in\mathcal D$, the estimator shares the same distribution as that in Example \ref{example:exp}, with the corresponding matrix $A_{\textup{pse}}=A_{\textup{sre}}(\bold U)$ conditional on $\bold U\in \mathcal D$, where ``$\textup{pse}$'' stands for ``post-stratified experiment''. Theorem \ref{theorem:beb} is the theoretical foundation for analyzing  $\hat\tau_\textup{ps}$ not only conditionally but also unconditionally.
\end{example}

\subsection{Literature Review, Open Questions, and Our Contributions}

The normal approximation to the non-stratified linear permutation statistics ($W_{A,\pi}$ with $K=1$) is a classic topic in probability and statistics. 
\citet{hoeffding1951combinatorial} extended the asymptotic normality of the permutation statistic in the product form of $\sum_{i=1}^n a_i b_{\pi(i)}$ \citep{wald1944statistical,noether1949theorem} to the general matrix form defined in \eqref{eq:matrix permutation}. 
\citet{motoo1956hoeffding}, \citet{hajek1961some} and \citet{fraser1956vector} further established Lindeberg--type central limit theorem (CLT) for the permutation statistic. \citet{von1976remainder} and \citet{ho1978l_p} provided its Berry--Esseen bounds (BEBs), achieving the order $O(n^{-1/2})$ under some boundedness conditions. \citet{bolthausen1984estimate} further relaxed their conditions, while keeping the order $O(n^{-1/2})$.
\cite{goldstein2005berry}, \citet{chen2011normal} and \citet{chen2015error} gave alternative proofs, and \citet{goldstein2007L1bound} studied the $L^1$ BEB based on the Wasserstein distance.

To motivate the theory, we focus on Example \ref{example:exp}, although the same discussion applies in other examples above. Stratified experiments are popular in empirical studies.
Various stratification regimes have been explored in experiments, such as paired experiments \citep{imai2009essential}, stratified factorial designs \citep{liu2022straclt}, finely stratified experiments \citep{fogarty2018}, threshold blocking designs \citep{higgins2016improving}, and stratified survey experiments \citep{tian2026survey}. These designs exhibit variations, with some featuring a large number of small strata, and some containing a small number of large strata. Real-world stratified designs often involve more complex and mixed regimes. We outline three stratification regimes as follows: 

\regime{large strata} The number of strata $K$ is large, and stratum sizes $\nk$'s are small.   

\regime{small strata} The number of strata $K$ is small, and the stratum sizes $\nk$'s are large.

\regime{mixture} 
Hybrid regimes can be mixtures of Regimes \ref{large strata} and \ref{small strata}.

Much of the previous work focused only on Regime \ref{large strata} \citep{imai2008variance,Abadie2008}  or Regime \ref{small strata} \citep{Imai2008,Pashley2017,de2024estimating}, separately. We will obtain a result suitable for all three regimes simultaneously. Furthermore, here we use ``large'' and ``small'' informally since it can be difficult to decide whether $K$ is large or $\nk$'s are large in some applications. Motivated by this difficulty, we will derive finite-sample bounds that can accommodate different values of $K$ and $n_{[k]}$'s.

Previous work has derived normal approximation results under the design-based framework on a case-by-case basis, including applications to survey samplings \citep{bickel1984asymptotic}, experiments \citep{lin2013,Li2016}, regression adjustment \citep{Liu2019}, factorial experiments \citep{liu2022straclt,shi2022berry} and rerandomization \citep{Li2018,wang2022rerandomization}. However, BEBs for stratified design are still missing in the design-based framework.

For post-stratification estimators, the unconditional CLT has not been rigorously established. By conditioning on the stratification variables $\bold U_1$ or $\bold U$, the post-stratification estimators exhibit asymptotic normality.
Here, the BEB offers a critical advantage over CLT: it not only implies the conditional asymptotic normality but also allows for deriving the unconditional asymptotic distribution by combining conditional distributions via the total probability law. Therefore, it is a critical theory tool for post-stratification inference.

Related to our work, \citet[Lemma 4]{dhault2024robust} established a CLT for the stratified permutation statistics in product form, and \cite{tuvaandorj2024combinatorial} established a Lindeberg-type CLT for the stratified permutation statistics in matrix form. 
However, they did not provide BEBs. Our BEB results complement the CLT of \citet{tuvaandorj2024combinatorial} through nonasymptotic bounds.

This paper aims to provide a unified non-asymptotic normal approximation result for the stratified linear permutation statistics. The remaining parts proceed as follows. Section \ref{sec:naive bound} presents BEBs for Regime \ref{large strata} and Regime \ref{small strata} based on the existing BEBs for non-stratified statistics. However, the bounds cannot deal with general regimes. Section \ref{sec:BEB Stein} introduces new general BEBs via Stein's method. The bound is useful for all possible regimes. Under mild conditions, it attains the classic $1/n^{1/2}$ order. Section \ref{sec:proof sketch} presents the sketch of the proof and derives some additional results as byproducts. Section \ref{sec:application} revisits the motivating examples. Section \ref{sec:permute_test} applies our new results to analyze stratified permutation tests. Section \ref{sec:additional resutl} discusses multivariate extensions. Supplementary Material contains additional theoretical tools and proofs.

\subsection{Basic result and notation}\label{subsec:basic result and notation}
As the basis for our discussion, 
we first establish the first two moments of $W_{A,\pi}$, which are also helpful for introducing the notation.

\begin{prop}\label{prop:var}
If $\pi\sim \textup{Uniform}(\Pi_K)$, the expected value of $W_{A,\pi}$ equals 
\begin{align*}
    \E \left[W_{A,\pi}\right] = \sumk n_{[k]} \bar{a}_{[k]\cdot\cdot},
\end{align*}
and the variance of $W_{A,\pi}$ equals
\begin{align*}
    \sigma_A^2 = \var\left(W_{A,\pi}\right) &= \sumk \frac{1}{\nk - 1} \sumijk \left(a_{ij}^2 -  \bar a_{[k] i\cdot} ^2 -  \bar a_{[k] \cdot j} ^2 +  \bar a_{[k] \cdot\cdot} ^2\right) \\
    &= \sumk \frac{1}{\nk - 1} \sumijk \left(a_{ij} -  \bar a_{[k] i\cdot}  -  \bar a_{[k] \cdot j}  +  \bar a_{[k] \cdot\cdot} \right)^2,
\end{align*}
where $ \bar a_{[k] i\cdot}  = {\nk}^{-1}\sumjk a_{ij}$,
 $ \bar a_{[k] \cdot j}  = {\nk}^{-1}\sumik a_{ij}$, and
 $ \bar a_{[k] \cdot\cdot}  = {\nk^{-2}}\sumijk a_{ij}$.
\end{prop}

From Proposition \ref{prop:var}, $W_{A,\pi} - \E [W_{A,\pi}] = \sumk\sumik (a_{i\pi(i)} -  \bar a_{[k] i\cdot}  - \bar a_{[k] \cdot \pi(i)} +  \bar a_{[k] \cdot\cdot} )$. Define $A^0 = [a_{ij}^0]$ with $a_{ij}^0 = a_{ij} - \bar a_{[k]  i\cdot} - \bar a_{[k]  \cdot j} + \bar a_{[k]  \cdot\cdot}$ for $i,j\in \Ik $ to center the statistic as $W_{A,\pi} - \E [W_{A,\pi}]=\sumk\sumik a_{i\pi(i)}^0 = W_{A^0,\pi}$. Centering does not change the variance so that $\sigma^2_{A^0} = \sigma^2_{A}$. If $\sigma^2_{A}>0$, we can define the scaled matrix $A^\textup{s}=[a_{ij}^\textup{s}]$, where ${a}_{ij}^\textup{s} = a_{ij}^0/\sigma_A$. The standardized statistic $W_{{A}^\textup{s},\pi}$ has mean $0$ and variance $1$. 

{\bf Notation.} 
Let $a_n\asymp b_n,a.s.$ denote that $a_n$ is asymptotically equivalent to $b_n$ almost surely, i.e., there exists a $C>0$ such that $a_n/b_n$ lies in the interval $[1/C,C]$ almost surely. Define the Wasserstein distance and Kolmogorov distance between random variables $A$, $B$ as $d_\textup{W}(A,B)=\int^{+\infty}_{-\infty} |\pr(A\le t)-\pr(B\le t)|\textup{d}t$ and $d_\textup{K}(A,B)=\sup_{t\in \mathbb R} |\pr(A\le t)-\pr(B\le t)|$, respectively.

\section{Suboptimal Results Based on Existing BEBs}\label{sec:naive bound}
\citet{berry1941accuracy} derived a BEB for the sum of IID random variables, \citet{esseen1942liapunov} derived a BEB for the sum of independent variables, and \citet{bolthausen1984estimate} derived a BEB for the linear permutation statistic with $K=1$. Both results are formulated in non-stratified settings.
We can immediately develop two BEBs for Regime \ref{large strata} and Regime \ref{small strata} based on those results. They are useful for their targeted regimes but useless for other regimes. 

We further simplify the BEBs in Corollaries \ref{coro:order smallstrata} and \ref{coro:order fewstrata} in some special cases to illustrate their limitations. The limitations motivate us to develop a unified result in Section \ref{sec:BEB Stein}.

\subsection{Leverage the BEB for the Sum of Independent Random Variables}

Define $W_{A,\pi[k]}=\sumik a_{i\pi(i)}$. We have $W_{A^\textup{s},\pi}=\sumk W_{A^\textup{s},\pi[k]}$ where $W_{A^\textup{s},\pi[k]}$'s are $K$ mutually independent random variables with $\E [W_{A^\textup{s},\pi[k]}]=0$ and  $\sumk \var (W_{A^\textup{s},\pi[k]})=1.$ Using the BEB for the sum of independent random variables $W_{A^\textup{s},\pi[k]}$'s \citep[Theorem 3.7]{chen2011normal}, 
we obtain the following bound in Proposition \ref{prop:smallstrata}.
\begin{prop}\label{prop:smallstrata}
If $\pi\sim \textup{Uniform}(\Pi_K)$, there exists a universal constant $C>0$ such that 
    $$\sup_{t\in \mathbb R}\Big| \pr(W_{ A^\textup{s},\pi}\le t)-\Phi (t)\Big| \le C  \sumk \nk \sumijk |a_{ij}^\textup{s}|^3.$$ 
\end{prop}

Proposition \ref{prop:smallstrata} is a direct result of the BEB for the sum of independent variables.  It is more useful in Regime \ref{large strata} with large $K$. The BEB for the sum of $K$ independent random variables asymptotically vanishes at the rate of $1/K^{1/2}$. For stratification regimes with $K$ strata, the convergence order remains $1/K^{1/2}$ under proper conditions described in Corollary \ref{coro:order smallstrata} below.

\begin{coro}\label{coro:order smallstrata}
  Assume that there exists a sequence of nonzero matrices $\{G_h\}_{h=1}^\infty$ where $G_h=[g_{h,ij}]$ is an $n_{h}$ by $n_{h}$ matrix and centered as $G^0_h=[g_{h,ij}^0]$ similar to the definition of $A^0$. Consider a random variable $\Omega$, with $\pr(\Omega = h) = w_h$ with $h = 1, 2, \ldots$ We consider $K$ random variables $\{\Omega_k\}_{k=1}^K$ independent and identically distributed with the same distribution as $ \Omega$. We form matrix $A=\diag\{G_{\Omega_k}\}_{k=1}^K$ and re-index it according to the subscript as $A=\diag\{A_{[k]}\}_{k=1}^K$. If $n_h$'s are uniformly bounded, larger than $1$ and
\begin{align}\label{eq:moment g}
\sum_{h=1}^\infty w_h n_h^{-2} \left(\sum_{1\le i,j\le n_h} |g_{h,ij}^0|^3\right)<\infty,
  \end{align}
 then as $K\rightarrow \infty$,  
  $$\sumk \nk \sumijk |a_{ij}^\textup{s}|^3 \asymp \frac{\sum_{h=1}^\infty w_{h}  n_h^{-2} \sum_{1\le i,j\le n_h} |g_{h,ij}^0|^3}{K^{1/2} (\sum_{h=1}^\infty w_{h}  n_h^{-2} \sum_{1\le i,j\le n_h} |g_{h,ij}^0|^2)^{3/2}} \quad \text{a.s.},$$
  where $\sum_{h=1}^\infty w_{h}  n_{h}^{-2} \sum_{1\le i,j\le n_h} |g_{h,ij}^0|^2$ and $\sum_{h=1}^\infty w_{h}  n_{h}^{-2} \sum_{1\le i,j\le n_h} |g_{h,ij}^0|^3$ are the second and third weighted moments, respectively.
\end{coro}
In Corollary \ref{coro:order smallstrata}, the second and third weighted moments are of constant order if \eqref{eq:moment g} holds.
Proposition \ref{prop:smallstrata} and Corollary \ref{coro:order smallstrata} indicate that under the conditions of Corollary \ref{coro:order smallstrata}, the Kolmogorov distance between the standardized statistic $W_{A^\textup{s},\pi}$ and the standard normal distribution converges to $0$ at the rate of $1/K^{1/2}$. Therefore, when $n\asymp K$ which occurs in Regime \ref{large strata}, this BEB can achieve the classic convergence order $1/ n^{1/2}$. However, Proposition \ref{prop:smallstrata} is not useful in Regimes \ref{small strata} and \ref{mixture}.

\subsection{Leverage the BEB for the Non-stratified Linear Permutation Statistic}

Alternatively, we can first apply the classic unstratified BEB \citep{bolthausen1984estimate} for each stratum-specific permutation statistic, then combine the results to obtain the normal approximation for the sum. 
Define
\begin{align}\label{eq:R_A[K]^2}
R_{A[k]}^2=\var(W_{A,\pi[k]})/\var(W_{A,\pi}),\quad k=1,\ldots,K,
\end{align}
which measures the proportion of variability contributed by stratum $k$ and is bounded between $0$ and $1$. We have the following bound in Proposition \ref{prop:fewstrata}.

\begin{prop}\label{prop:fewstrata}
If $\pi\sim \textup{Uniform}(\Pi_K)$, there exists a universal constant $C>0$ such that
    $$\sup_{t\in \mathbb R}\Big| \pr(W_{ A^\textup{s},\pi}\le t)-\Phi (t)\Big| \le C\left\{\sumk \sumijk |a_{ij}^\textup{s}|^3/(\nk R_{A[k]}^2)\right\}^{1/2}.$$
\end{prop}

Since the Kolmogorov distance does not satisfy the sub-additivity property, we leverage the sub-additivity of the Wasserstein distance to combine the BEB for each stratum \citep{goldstein2007L1bound}.
The order $1/2$ in Proposition \ref{prop:fewstrata} comes from the fact that the Kolmogorov distance is bounded by the square root of the Wasserstein distance \citep{panaretos2019statistical}; see Remark \ref{Kol-Was} later in the paper. Since $\sumk R_{A[k]}^2=1,$ there must exist a stratum $k$ with $R_{A[k]}^2 \le 1/K$. Consequently, Proposition \ref{prop:fewstrata} is useful when $K$ is small as in Regime \ref{small strata} but not useful when $K$ diverges as in Regime \ref{large strata}. 

In Corollary \ref{coro:order fewstrata} below, we discuss the convergence order of the BEB in Proposition \ref{prop:fewstrata}.

\begin{coro}\label{coro:order fewstrata}
Assume that the elements of $A^0$ are independently sampled from a population $Y$ with $\E[|Y|^{6+\varepsilon}]<\infty$ for some $\varepsilon>0$. Consider a homogeneous stratified design where there exist $0<\underline{c}<\overline{c}$ such that for any stratum $k$, $ R_{A[k]}^2 \in  [\underline c/K,\overline c/K]$, $k=1,\ldots,K$. As $\min_{1\le k\le K} n_{[k]}\rightarrow \infty$, 
    we have 
    $$\left\{\sumk \sumijk |a_{ij}^\textup{s}|^3/(\nk R_{A[k]}^2)\right\}^{1/2}\asymp \frac{K^{1/2}(\E[|Y|^3])^{1/2} }{n^{1/4}(\E[|Y|^2])^{3/4}} \quad \text{a.s.}.$$ 
\end{coro} 

In Corollary \ref{coro:order fewstrata}, the BEB in Proposition \ref{prop:fewstrata} is of order $O(K^{1/2}/n^{1/4})$ if each stratum contributes asymptotically the same order to the variance and outcomes have finite $6+\varepsilon$th moments. Therefore, Corollary \ref{coro:order fewstrata} is useful when $K=o( n^{1/2} )$, which occurs in Regime \ref{small strata}. We cannot derive CLTs based on this bound when the growth rate of $K$ exceeds $ n^{1/2}$. Therefore, the bounds are not useful in Regimes \ref{large strata} and \ref{mixture}.  

Both Proposition \ref{prop:smallstrata} and Proposition \ref{prop:fewstrata} are useful for one particular regime, but cannot cover the other and the general Regime \ref{mixture}. We can combine Propositions \ref{prop:smallstrata} and \ref{prop:fewstrata} to derive a $L^1$ BEB using the sub-additivity of the Wasserstein distance (see Section \ref{subsec:combination of proposition 23} in the Supplementary Material). However, it is not easy to use this strategy to obtain the BEB with the optimal convergence rate because the Kolmogorov distance lacks sub-additivity.
In Section \ref{sec:BEB Stein} below, we will derive a BEB that can deal with all possible stratification regimes.

\section{Stratified Permutational BEBs Based on Stein's Method}\label{sec:BEB Stein}

\subsection{Main Theorem}
\newtheorem{theorem1}{Theorem}
Now we present the main theorem of the paper.
\begin{theorem1}\label{theorem:beb}
If $\pi \sim \textup{Uniform}(\Pi_K)$,  there exists a universal constant $C$, such that
$$\sup_{t\in \mathbb R}\Big| \pr(W_{ A^\textup{s},\pi}\le t)-\Phi (t)\Big| \le  C \sum_{k=1}^K\sumijk |a_{ij}^\textup{s}|^3/ \nk.$$
\end{theorem1}

Theorem \ref{theorem:beb} bounds the Kolmogorov distance between the standardized permutation statistic and the standard normal distribution. When the upper bound tends to zero, the stratified linear permutation statistic is asymptotically normal. 

Motivated by \cite{bolthausen1984estimate}, \cite{goldstein2005berry} and \cite{chen2011normal}, we construct zero-bias transformation \citep{goldstein1997stein} to prove the result in Theorem \ref{theorem:beb}. This result holds for all stratification regimes. 

  Corollary \ref{coro:order general-beb} gives the convergence order of the bounds in Theorem \ref{theorem:beb}.

\begin{coro}\label{coro:order general-beb}
    Assume that the elements of $A^0$ are independent and identically distributed samples from $Y$ with $\E [|Y|^{6+\varepsilon}]<\infty$ for some constant $\varepsilon>0$. As $\min_{k=1,\ldots,K} \nk \rightarrow \infty,$ we have
$$\sumk \sumijk |a_{ij}^\textup{s}|^3/\nk \asymp \frac{\E[|Y|^3] }{ n^{1/2} (\E[|Y|^2])^{3/2}} \quad \text{a.s.}.$$
\end{coro}
\begin{table}[H]
\centering
\caption{Convergence orders of the BEBs. We impose additional super-population moment conditions to help interpret the upper bounds.}
\resizebox{0.9\textwidth}{!}{
\begin{tabular}{c|ccc}
\toprule
BEB & Proposition \ref{prop:smallstrata} & Proposition \ref{prop:fewstrata} &  Theorem \ref{theorem:beb} \\
 \midrule
convergence order      & $1/K^{1/2}$     & $K^{1/2}/n^{1/4}$         & $1/n^{1/2}$\\ 
\midrule
Condition    & $3$rd moment finite   &$6+\varepsilon$th moment finite, homogeneity  &$6+\varepsilon$th moment finite\\
\bottomrule
\end{tabular}}
\label{tab:convergence order}
\end{table}
Table \ref{tab:convergence order} summarizes the convergence orders of our BEBs under proper conditions.
From Corollary \ref{coro:order general-beb}, the convergence order of the bound in Theorem \ref{theorem:beb} is $1/n^{1/2}$, which is no worse than that in Proposition \ref{prop:smallstrata} and Proposition \ref{prop:fewstrata}.

\subsection{Corollaries}
In this subsection, we present two corollaries. Corollary \ref{coro:bolthausen1984} shows that Theorem \ref{theorem:beb} includes the result of \cite{bolthausen1984estimate} in the non-stratified setting as a special case. 

\begin{coro}[\cite{bolthausen1984estimate}]\label{coro:bolthausen1984}
   When $K=1$ and $\pi\sim \textup{Uniform}(\Pi_K)$, there exists a universal constant $C>0$, such that
   \begin{align}\label{eq:bolth}
      \sup_{t\in \mathbb R}\Big| \pr(W_{ A^\textup{s},\pi}\le t)-\Phi (t)\Big| \le C \sumin\sum_{j=1}^n |a_{ij}^\textup{s}|^3/ n. 
   \end{align}
\end{coro}

Corollary \ref{coro:apply BEB} below provides a modified version for statistical applications.
The literature on combinatorial CLT presents bounds in the form of row and column demeaned value of matrix $A$, i.e., $a_{ij}^{\textsc{s}}$, because this form is symmetric in row and column. However, $a_{ij}^\textup{s}$ involves both $Y(0)$ and $Y(1)$ in Examples \ref{example:exp} and \ref{example:ps}, which is less interpretable than the columnwise demeaned value of matrix $A$ as $a_{ij}^{\textup{0c}}=a_{ij}-\bar a_{[k]\cdot j},i,j\in \Ik,k=1,\ldots,K$, which only involves $Y(1)$ or $Y(0)$. By the basic inequality between the centered and original moment \citep[Theorem 2.1]{mori2009sharp}, there exists a constant $C$ such that $\sumijk |a_{ij}^0|^3 \le  C \sumijk |a_{ij}^{0\textup{c}}|^3$. Therefore, by Theorem \ref{theorem:beb}, we have the following BEB in the form of $a_{ij}^{0\textup{c}}$.
\begin{coro}
    \label{coro:apply BEB}
If $\pi \sim \textup{Uniform}(\Pi_K)$, there exists a constant $C$ such that $$\sup_{t\in \mathbb R}\Big| \pr(W_{A^\textup{s},\pi}\le t)-\Phi (t)\Big| \le C \sum_{k=1}^K\sumijk |a_{ij}^\textup{0c}|^3/ (\sigma^3_A\nk).$$
\end{coro} 

Although Corollary \ref{coro:apply BEB} is looser than Theorem \ref{theorem:beb}, it can provide more interpretable bounds for statistical applications. We will apply Corollary \ref{coro:apply BEB} to the motivating examples in Section \ref{sec:application}.

\section{Sketch of the Proof and Additional Results}\label{sec:proof sketch} 
Although we relegate the technical details to the Supplementary Material, we present the main steps of our proof. This section not only gives the proving ideas but also introduces byproducts of the proof that are of independent interest. 
\subsection{Review of Stein's Method}
Our proof employs Stein's method \citep{stein1972bound,chen2011normal,chen2021stein} for normal approximation.  \cite{stein1972bound} proved that a random variable $W$ has a standard normal distribution if and only if $\E[ f'(W)]=\E[Wf(W)]$ for all absolutely continuous functions $f$ for which the above expectations exist. 
Let $N\sim\mathcal N(0,1)$. For any measurable function $h(\cdot)$ such that $\E[|h(N)|]<\infty$, define 
\begin{align}\label{eq:f--h}
    f(x)=e^{x^2/2}\int_{-\infty}^x(h(t)-\E [h(N)])e^{-t^2/2}\textup{d} t.
\end{align}
The $f$ in \eqref{eq:f--h} satisfies
\begin{align}\label{eq:identi}
h(x)-\E [h(N)]=f'(x)-xf(x),
\end{align}
for Lebesgue-a.e.\ $x\in\mathbb R$. Moreover, we take a version of $f'$ such that \eqref{eq:identi} holds for all $x\in\mathbb R$; see, for example, Remark~3.2.3 of \citet{nourdin2012normal}.
Choose $x=W_{A^{\textup{s}},\pi}$ in \eqref{eq:identi} and take expectation on both sides to obtain
\begin{align}\label{eq:stein}
   \E [h(W_{A^\textup{s},\pi})]-\E [h(N)]=\E [f'(W_{A^\textup{s},\pi})]-\E[W_{A^\textup{s},\pi}f(W_{A^\textup{s},\pi})].
\end{align}
If we choose  $h(x)=\bold{1}(x\le t)$ with $t\in \mathbb R$, the left-hand side of \eqref{eq:stein} reduces to $\pr(W_{A^\textup{s},\pi}\le t)-\Phi(t)$. Thus, Stein's method bounds the right-hand side of \eqref{eq:stein} to obtain the normal approximation for $W_{A^\textup{s},\pi}$.

\subsection{Zero-Bias Transformation}\label{subsec:pi dagger construct}

Our proof employs the zero-bias transformation introduced by \citet{goldstein1997stein}: 
\begin{definition}[Zero-Bias Transformation]
A random variable $W^*$
is a zero-bias transformation of a random variable $W$
if and only if $\var(W)\E [f'(W^*)] = \E [W f(W)]$ for all absolutely continuous functions $f$ for which these expectations exist.
\end{definition}
Using zero-bias transformation, we can further simplify \eqref{eq:stein}:
\begin{align}\label{eq:zb form}
    \E [h(W_{A^\textup{s},\pi})]-\E [h(N)]=&  \E [f'(W_{A^\textup{s},\pi})]-\E[W_{A^\textup{s},\pi} f(W_{A^\textup{s},\pi})]\notag\\
    =&\E [f'(W_{A^\textup{s},\pi})]-\E[f'(W_{A^\textup{s},\pi}^*)].
\end{align}
This representation is convenient for algebraic manipulation, and we will exploit it in Section~\ref{subsec:algebra cal} to obtain explicit bounds on the right-hand side of \eqref{eq:zb form}.

We construct a zero-bias transformation of $W_{A^\textup{s},\pi}$ as follows. Define $\rho_{i,j}$ as the permutation that swaps $i$ and $j$, leaving all other elements unchanged.
\mbox{}\\
\mystep{step:one} We randomly sample $B^\dagger$ from $\{1,\ldots,K\}$ with $\pr(B^\dagger=k)=R_{A[k]}^2$ and randomly sample $I^\dagger, J^\dagger, P^\dagger, Q^\dagger$ from $\mathcal I_{[B^\dagger]}$ with probability 
$$\pr(I^\dagger = i , J^\dagger =j,P^\dagger =p,Q^\dagger =q \mid B^\dagger =k)=\frac{(a_{ip}^\textup{s}+a_{jq}^\textup{s}-a_{iq}^\textup{s}-a_{jp}^\textup{s})^2}{4\nk^2(\nk-1) R_{A[k]}^2},$$
for any $i,j,p,q\in \mathcal I_{[B^\dagger]}$.
We generate $\pi^\dagger$ based on $I^\dagger=i,J^\dagger=j,P^\dagger=p,Q^\dagger=q$ as follows:
$$\pi^\dagger=\begin{cases}
\pi\rho_{\pi^{-1}(p),j} & \text{if}~~q=\pi(i),p\neq \pi(j),\\
\pi\rho_{\pi^{-1}(q),i} & \text{if}~~q\neq\pi(i),p= \pi(j),\\
\pi\rho_{\pi^{-1}(p),i}\rho_{\pi^{-1}(q),j} & \text{otherwise}.
\end{cases}$$
Based on $I^\dagger,J^\dagger$ and $\pi^\dagger$, we define  $\pi^\ddagger=\pi^\dagger  \rho_{I^\dagger, J^\dagger}$. 
\\
\mystep{step:two} We construct $W_{A^\textup{s},\pi}^*=U W_{A^\textup{s},\pi^\dagger}+(1-U)W_{A^\textup{s},\pi^\ddagger}$, where $U$ follows a uniform distribution on $[0,1]$ and is independent of all the random variables in Step \ref{step:one}. 

\begin{lemma}\label{thm:zerobias}
$W^*_{A,\pi}$ is a zero-bias transformation of $W_{A,\pi}$.
\end{lemma}

The permutation $\pi^\ddagger$ reassigns the sampled indices $\{I^\dagger,J^\dagger\}$ based on $\pi^\dagger$.
Proposition~\ref{prop:permutation property} in the Supplementary Materials shows that $\pi^\ddagger$ maps $\{I^\dagger,J^\dagger\}$ to $\{P^\dagger,Q^\dagger\}$.
On the remaining elements, it acts as a random permutation from 
$\mathcal I_{[B^\dagger]}\setminus\{I^\dagger,J^\dagger\}$ to 
$\mathcal I_{[B^\dagger]}\setminus\{P^\dagger,Q^\dagger\}$.
As a result, the pair $(W_{A,\pi^\dagger}, W_{A,\pi^\ddagger})$ satisfies condition~\eqref{eq:zero-bias proof} in the Supplementary Material, which implies Lemma~\ref{thm:zerobias}.
Further details are given in Section~\ref{sec:construct zero-bias} of the Supplementary Material.

\citet{dhault2024robust} swapped two randomly selected indices $I$, $J$ within a random
stratum $B$ to obtain a coupling of $\pi$, denoted by $\pi''=\pi\rho_{I,J}$.
Using \citet[Corollary 4.1]{chen2011normal}, which bounds the Wasserstein distance between $W_{A^\textup{s},\pi}$ and a normal random variable by moments of $W_{A^\textup{s},\pi}-W_{A^\textup{s},\pi''}$, they derived a CLT for product form permutation statistics. \citet{tuvaandorj2024combinatorial} introduced $W_{A^\textup{s},\pi^\dagger}$ but not $W^*_{A^\textup{s},\pi}$, and used Stein's method to derive a CLT for matrix form permutation statistics. We will show in Section \ref{subsec:proof tech and byproducts} below that their proof strategy can be extended to derive an $L^1$ BEB using zero-bias transformation. Our construction of the zero-bias transformation $W^*_{A^\textup{s},\pi}$ leads to stronger results.  See Section \ref{subsec: product form} below for further comments.

\subsection{Simplification}\label{subsec:simplication}
We first show that we can simplify the discussion of $W_{A^\textup{s},\pi}$ to the permutation statistic with a matrix having zero column and row means within 
strata and elements bounded by $1$.

For any matrix $A$, define $\beta_{A[k]}=\sumijk |a_{ij}^0|^3/\sigma_A^3$.
Let $M_n^\textup{s}$ be the set of matrices with zero column and row means whose corresponding stratified linear permutation statistics have variance $1$:
\begin{align}\label{eq:Mns}
    M_n^\textup{s}=\{A: \bar a_{[k] i\cdot} = \bar a_{[k] \cdot j} =0 \text{ for all } i,j\in\Ik , k=1,\ldots,K; \sigma_A^2=1\}.
\end{align}
Further define a subset of matrices with each element bounded by $1$:  
$$M_n^1=\{A\in M_n^\textup{s}:|a_{ij}| \le 1 \text{ for all } i,j\}.$$

 Consider a matrix $A=[a_{ij}]\in M_n^\textup{s}$. For any $\varepsilon>0$, when $\sumk \beta_{A[k]}/\nk\ge \varepsilon$, the following inequality holds since the right-hand side is larger than or equal to 2: 
 $$\sup_{t\in\mathbb R}|\pr(W_{A,\pi}\le t)-\Phi(t)|\le 2 \varepsilon^{-1} \sumk \beta_{A[k]}/\nk.$$ 
  When $\sumk \beta_{A[k]}/\nk\le \varepsilon$, by Lemma \ref{lemma:scaling} in the Supplementary Material, we can prove that there exists a constant $c$, such that
$$\sup_{t\in\mathbb R}|\pr(W_{A,\pi}\le t)-\Phi(t)|\le \sup_{t\in \mathbb R}|\pr(W_{({A'})^\textup{s},\pi}\le t)-\Phi(t)|+c\sumk \beta_{A[k]  }/\nk,$$
where $A'=[a'_{ij}]$ with $a'_{ij}=a_{ij}^\textup{s}\bold 1(|a_{ij}^\textup{s}|\le 1/2)$.
 We can prove that $(A')^\textup{s}\in M_n^1$ in Lemma \ref{lemma:truncation} in the Supplementary Material.  Therefore, we only need to study the permutation statistic on $(A')^\textup{s}\in M_n^1$, the result of which is summarized in Theorem \ref{theorem:beb mn1}.

\begin{theorem}\label{theorem:beb mn1}
For $A\in M_n^1$ and $\pi \sim \textup{Uniform}(\Pi_K)$, there exists a constant $C$, such that 
$$\sup_{t\in \mathbb R}\Big| \pr(W_{A,\pi}\le t)-\Phi (t)\Big| \le C  \sumk\beta_{A[k]}/\nk .$$
 \end{theorem}

The proof of Theorem \ref{theorem:beb} follows directly from Theorem \ref{theorem:beb mn1}. Because of this, we only consider $A\in M_n^1$ for the remainder of Section \ref{sec:proof sketch}.

\subsection{Algebraic Calculation}\label{subsec:algebra cal}
Denote $h_{t,\alpha}(x)$ as a smoothed indicator function of $(-\infty,t]$ which decays linearly from $1$ to $0$ over the interval $[t,t+\alpha]$ with $h_{t,0}(x)=\bold 1(x\le t)$, $h_{t,0}(x)\le h_{t,\alpha}(x)\le h_{t+\alpha,0}(x)$ and $\E [h_{t,0}(X)]=\pr (X\le t)$. 
For $N\sim \mathcal N(0,1)$, by the triangle inequality, we have
$$|\pr(W_{A,\pi}\le t)-\Phi(t)|\le |\E [h_{t,\alpha}(W_{A,\pi})]-\E [h_{t,\alpha}(N)]|+|\E[|h_{t,\alpha}(N)-h_{t,0}(N)|]|.$$
For the second term, because $h_{t,\alpha}$ differs from $h_{t,0}$ only on $[t,t+\alpha]$ and both are in $[0,1]$, we have $|h_{t,\alpha}(N)-h_{t,0}(N)|\le \boldsymbol 1\{N\in [t,t+\alpha]\}$. Therefore, $\E[|h_{t,\alpha}(N)-h_{t,0}(N)|]\le\pr(N\in[t,t+\alpha]) \le \alpha/\sqrt{2\pi}$.
Therefore, we focus on bounding the first term.
Using $f_{t,\alpha}$ constructed by $h_{t,\alpha}$ in \eqref{eq:f--h} and $W_{A,\pi}^*$ as a zero-bias transformation of $W_{A,\pi}$, we have 
\begin{align*}
    \E [h_{t,\alpha}(W_{A,\pi})]-\E[ h_{t,\alpha}(N)]     \overset{\eqref{eq:zb form}}{=}&~\E[f'_{t,\alpha}(W_{  A,\pi})]-\E [f'_{t,\alpha}(W_{  A,\pi}^*)].
\end{align*}

Using the inequality in \citet[Lemma 2.5]{chen2011normal} for $f_{t,\alpha}$:
$$|f'_{t,\alpha}(w+v)-f'_{t,\alpha}(w)|\le |v|\left(1+|w|+\frac{1}{\alpha} \int_0^1 \bold{1}(w+rv\in [t,t+\alpha] )\textup{d}r\right),$$
we can obtain 
\begin{align}\label{eq:total decomposition}
    \sup_{t\in \mathbb R}|\pr(W_{A,\pi}\le t)-\Phi(t)|\le  \mathcal A_1+\mathcal A_2+\mathcal A_3+\frac{\alpha}{\sqrt{2\pi}},
\end{align}
where
\begin{eqnarray}
\mathcal A_1&=&\E [|W_{A,\pi}^*-W_{A,\pi}|],\notag\\
\mathcal A_2&=&\E [|W_{A,\pi}(W_{A,\pi}^*-W_{A,\pi})|],\notag\\
\mathcal A_3&=&\frac{1}{\alpha} \E \left[|W_{A,\pi}^*-W_{A,\pi}|\int_0^1 \bold{1}\left\{W_{A,\pi}+r(W_{A,\pi}^*-W_{A,\pi})\in [t,t+\alpha]\right\}\textup{d}r \right].\label{eq:A3 keyquan}
\end{eqnarray}
The key is to bound $\mathcal A_1$, $\mathcal A_2$, and $\mathcal A_3$.

For $\mathcal A_1$, it is the $L_1$ norm difference between $W_{A,\pi}$ and its zero-bias transformation. We can prove Proposition \ref{prop:L1} in the Supplementary Material that
\begin{align}\label{eq:bound for A1 in main}
    \mathcal A_1\le 80 \sumk \beta_{A[k]}/\nk.
\end{align}

 For $\mathcal A_2$, define 
 \begin{align}\label{eq:define R S T}
     \mathcal R=\{I^\dagger, J^\dagger, \pi^{-1}(P^\dagger),\pi^{-1}(Q^\dagger) \},\quad S=\sum_{i\not\in \mathcal R}a_{i\pi(i)},\quad T=\sum_{i\in \mathcal R}a_{i\pi(i)}.
 \end{align}
Since $A\in M_n^1$, we have  
$|W_{A,\pi}|=|S+T|\le |S|+\sum_{i\in \mathcal R}|a_{i\pi(i)}|\le |S|+4.$ Define $$\textbf{I}=(B^\dagger,I^\dagger,J^\dagger,\pi^{-1}(P^\dagger),\pi^{-1}(Q^\dagger),\pi(I^\dagger),\pi(J^\dagger),P^\dagger,Q^\dagger).$$ We prove in Section \ref{subsubsec:bound A_1-3} in the Supplementary Material that there exists a universal constant $C$ such that
$$\mathcal A_2\le \E [|W_{A,\pi}^*-W_{A,\pi}|\cdot \E[|S|+4\mid \textbf{I}]]\le C\E [|W_{A,\pi}^*-W_{A,\pi}|]=C\mathcal A_1.$$

For $\mathcal A_3$, we first derive the bound
\begin{align}\label{eq:A3 method2}
\mathcal A_3\le \frac{1}{\alpha}\E!\left[|W_{A,\pi}^*-W_{A,\pi}|\sup_{t\in\mathbb R}\pr!\left(S\in[t,t+\alpha]\mid \mathbf I\right)\right],
\end{align}
where the derivation is given in Section~\ref{subsubsec:bound for A_3 new} of the Supplementary Material.
Moreover, by the concentration inequality \eqref{eq:unified concentration} proved in the Supplementary Material, for some constants $C'$ and $C''$,
$$\pr(S\in [a,b]\mid \textbf{I})\le C'(b-a)+C''\sumk\beta_{A[k]}/\nk.$$

Choosing $\alpha=\sumk\beta_{A[k]}/\nk$ in \eqref{eq:A3 method2}, we have the BEB in Theorem \ref{theorem:beb}.

\subsection{Comments on Proof Techniques and Byproducts}\label{subsec:proof tech and byproducts}
We compare our proof techniques with the existing ones.

First, the proofs of \citet{bolthausen1984estimate}, \citet{chen2011normal}, \citet{chen2015error} and ours all relate to $S$ defined in \eqref{eq:define R S T}. Given $\bold I$, we can view $S$ as a permutation statistic created by deleting the rows $\{I^\dagger,J^\dagger,\pi^{-1}(P^\dagger),\pi^{-1}(Q^\dagger)\}$ and the columns $\{\pi(I^\dagger),\pi(J^\dagger),P^\dagger,Q^\dagger\}$ from the original matrix $A$. In non-stratified settings, \cite{bolthausen1984estimate} and \cite{chen2011normal} bounded the right-hand side of \eqref{eq:A3 keyquan} by an inductive approach. However, it is challenging to apply that approach when we have a multidimensional $(n_{[1]},\ldots,n_{[K]})$ as in the stratified settings. 
In contrast, \cite{chen2015error} used the concentration inequality approach in Stein's method to bound $\pr(S\in [a,b]\mid \textbf{I})$ in the non-stratified settings. We adopt a similar strategy for proving Theorem \ref{theorem:beb} and provide a parallel concentration lemma to their core result for the stratified settings in Section \ref{subsubsec:concentration lemma} of the Supplementary Material.

However, the bound obtained by directly applying the techniques of \citet{chen2015error}, which address only non-stratified settings, is informative only when both $R_{A[k]}^2/\nk$ is small and $\nk\ge6$ for all $k$. These conditions are generally violated when one or several small strata contribute non-negligibly to the overall variance, under which the argument of \citet{chen2015error} no longer applies. Therefore, obtaining a tighter bound requires a more refined use of information across all strata rather than relying solely on per-stratum conditions.

Based on the simplification in Section~\ref{subsec:simplication}, since Theorem \ref{theorem:beb} holds trivially when $\sumk \beta_{A[k]}/\nk\ge \varepsilon,$ we only need to consider the case $\sumk \beta_{A[k]}/\nk\le \varepsilon$, which, by H\"older’s inequality, implies $R_{A[k]}^2/\nk = O(\varepsilon^{2/3})$. This ensures that no single small stratum dominates the total variance, thereby avoiding the limitation inherent in the per-stratum bound.
When $\nk$ is small, however, the difficulty in directly applying the result of \citet{chen2015error} becomes more fundamental. Removing $|\mathcal R|$ rows and columns from stratum $B^\dagger$ can effectively eliminate an entire stratum, so the number of strata contributing to the randomness of $S$ differs from that of $W_{A,\pi}$. To handle this case, we establish a separate concentration inequality tailored to such small-stratum configurations. In this setting, since the $W_{A[k]}$’s are independent, the resulting concentration behavior is analogous to the Berry–Esseen bound for sums of independent random variables \citep[Lemma~3.1]{chen2011normal}.

Second, \cite{bolthausen1984estimate} and \cite{chen2011normal} utilized the truncated matrix $A'$ in their proofs, while \cite{chen2015error} did not. We employ the idea of truncation. Utilizing truncation, we only need to consider matrices in $M_n^1$. It implies $|T|\le 4$ in $\mathcal A_2$ and simplifies the proof of Theorem \ref{theorem:beb}. 

Third, both \cite{chen2011normal} and our proof use the zero-bias transformation, while
\cite{bolthausen1984estimate} and \cite{chen2015error} did not.
Since the normal approximation of $W_{A,\pi}$ in Wasserstein distance can be bounded by the $L_1$ norm difference between $W_{A,\pi}$ and its zero-bias transformation
\citep[Theorem 4.1]{chen2011normal}:
\begin{align}\label{eq:relation between d_W and L1}
    d_{\textup W}(W_{A,\pi},\mathcal N(0,1))\le 2\E [|W_{A,\pi}^*-W_{A,\pi}|].
\end{align}
We can further bound the right-hand side of \eqref{eq:relation between d_W and L1} by the third moment (see \eqref{eq:bound for A1 in main}) and obtain the $L^1$ BEB for stratified linear permutation statistics as a byproduct in Theorem \ref{thm:Wasserstein BEB} below.

\begin{theorem}\label{thm:Wasserstein BEB}
If $\pi \sim \textup{Uniform}(\Pi_K)$,
there exists a constant $C>0$ such that
\begin{align}
    d_{\textup W}\left(W_{A^\textup{s},\pi},\mathcal N(0,1)\right)\le C \sumk \sumijk|a_{ij}^\textup{s}|^3/\nk.
\end{align}
\end{theorem}

\begin{remark}\label{Kol-Was}
   From \citet[Proposition 1.2]{ross2011fundamentals}, we have
   \begin{align}\label{eq:K-W naive bound}
       d_{\textup K}\left(W_{A^\textup{s},\pi},\mathcal N(0,1)\right)\le  \left(\frac{2}{\pi}\right)^{1/4}\sqrt{ d_{\textup W}\left(W_{A^\textup{s},\pi},\mathcal N(0,1)\right)}.
   \end{align}
By \eqref{eq:K-W naive bound}, Theorem \ref{thm:Wasserstein BEB} implies that there exists a constant $C>0$ such that $$d_{\textup K}\left(W_{A^\textup{s},\pi},\mathcal N(0,1)\right)\le C\left(\sumk \sumijk|a_{ij}^\textup{s}|^3/\nk\right)^{1/2},$$
 which is weaker than that in Theorem \ref{theorem:beb} in terms of the convergence rate.
\end{remark}
\begin{remark}
  \citet{dhault2024robust} used Stein's method to obtain a CLT for $W_{A,\pi}$ when $A$ is in the form of outer product of two vectors by proving $d_{\textup W}(W_{A^\textup{s},\pi},\mathcal N(0,1))\rightarrow 0$.
  Using Theorem \ref{thm:Wasserstein BEB}, we can strengthen their results by relaxing their conditions. We will revisit this point in Corollary \ref{coro:IV} below. 
\end{remark}

\section{Applications to the Motivating Examples}\label{sec:application}
We now apply Corollary \ref{coro:apply BEB} to Examples \ref{example:sampling}--\ref{example:ps} and obtain the corresponding BEBs. Define $$M_{[k]}^r(Y) = n_{[k]}^{-1} \sum_{i\in \mathcal{I}_{[k]}} | Y_i - \bar{Y}_{[k]} |^r$$ as the $r$-th finite-population moment of $Y$ within the $k$-th stratum, which will play key roles in our presentation.

\subsection{Example \ref{example:sampling}: Stratified Sampling}\label{subsec:sampling}
For stratified sampling, \cite{bickel1984asymptotic} established a CLT for the unbiased estimator $\hat{\gamma}$ in Example \ref{example:sampling}. \cite{tuvaandorj2024combinatorial} further established a CLT for general stratified linear permutation statistics, which includes the result of \cite{bickel1984asymptotic} as a special case. Here, we focus on $\hat\gamma = \sumk w_{[k]} \hat Y_{[k]}$ by replacing $\{\pik=\nk/n: k=1,\ldots, K\}$ with a general form of weights $\{w_{[k]}: k=1,\ldots, K\}$. Different choices of $w_{[k]}$ correspond to different estimands of interest. For instance, one would choose $w_{[k]}=\nk/n$ if the population of individuals is of interest, and choose $w_{[k]}=1/K$ if the population of strata is of interest. Let $p_{[k]}=n_{[k]1}/\nk$ denote the sampling rate. \cite{bickel1984asymptotic} derived $\var(\hat\gamma)=\sumk \pik^2 (\nk n_{[k]1})^{-1}n_{[k]0}S_{[k]}^2,$ where $S_{[k]}^2=(\nk-1)^{-1}\sumik (Y_i-\bar Y_{[k]})^2$. Applying Corollary \ref{coro:apply BEB}, we obtain the following result on $\hat{\gamma}$.

\begin{coro}[BEB for stratified sampling]\label{coro:sampling}  There exists a constant $C$ such that 
\begin{align}\label{eq:sampling}
       \sup_{t\in \mathbb R}\left|\pr\left(\frac{\hat\gamma-\gamma}{\sigma_{\textup{ss}}}\le t\right)-\Phi(t)\right|\le  \frac C {\sigma^3_{\textup{ss}}}\sumk \left(\frac{w_{[k]}^3}{ n_{[k]1}^2}\right )\Mkthree (Y) ,
   \end{align}
   where $\sigma_{\textup{ss}}^2=\var(\hat{\gamma})$. In particular, if $\pik=\nk/n$ for all $k$, each $p_{[k]}$ is bounded away from $0$ and $1$, and both
$n^{-1}\sumk \nk S_{[k]}^2$ and $n^{-1}\sumk \nk \Mkthree(Y) $ have nonzero finite limits, the upper bound in \eqref{eq:sampling} is of order $1/n^{1/2}$. 
\end{coro}

\subsection{Example \ref{example:exp}: Stratified Experiment}\label{subsec:exp}
\cite{fogarty2018finely,fogarty2018} proved CLTs for two special stratified experiments: paired experiments and finely stratified experiments.
Following the techniques used in \cite{bickel1984asymptotic} for stratified sampling, \cite{Liu2019} provided a CLT for stratified randomized experiments. \cite{liu2022straclt} further proved a multi-treatment level CLT for stratified factorial experiments. 
Similar to Section \ref{subsec:sampling}, we focus on $\hat\tau =\sumk w_{[k]} \hat\tau_{[k]}$ by replacing $\{\pik=\nk/n:k=1,\ldots,K\}$ with a general form of weights $\{w_{[k]}:k=1,\ldots,K\}$. 
Define stratum-specific propensity score as $p_{[k]}=n_{[k]1}/\nk$, stratum-specific variances as $S_{[k]z}^2=(\nk-1)^{-1}\sumik\{Y_i(z)-\bar Y_{[k]}(z)\}^2$ for $z=0,1$, and $S_{[k]\tau}^2=(\nk-1)^{-1}\sumik[\{Y_i(1)-Y_i(0)\}-\{\bar Y_{[k]}(1)-\bar Y_{[k]}(0)\}]^2$. 
We have $$\var(\hat\tau)=\sumk \pik^2 \left(\frac{S_{[k]1}^2}{n_{[k]1}}+\frac{S_{[k]0}^2}{n_{[k]0}}-\frac{S_{[k]\tau}^2}{\nk}\right).$$
Applying Corollary \ref{coro:apply BEB}, we obtain the following result on $\hat{\tau}$.

\begin{coro}[BEB for stratified experiment]\label{coro:strati exp}
 There exists a constant $C$ such that 
    \begin{align}\label{eq:experiment}
        \sup_{t\in \mathbb R}\left|\pr\left(\frac{\hat\tau-\tau}{\sigma_{\textup{sre}}}\le t\right)-\Phi(t)
       \right|
        \le  \frac C {\sigma_{\textup{sre}}^3}\sumk  w_{[k]} ^3\left\{\frac{\Mkthree (Y(1))}{n_{[k]1}^2}+\frac{\Mkthree (Y(0))}{n_{[k]0}^2}\right\},
    \end{align}
 where $\sigma_{\textup{sre}}^2=\var(\hat{\tau})$. In particular, if $w_{[k]}=\nk/n$ for all $k$, each $p_{[k]}$ is bounded away from $0$ and 1, and 
$n\sigma^2_{\textup{sre}}$ and $n^{-1}\sumk \nk \Mkthree(Y(z))$ for $z=0,1 $ have nonzero finite limits, the upper bound in \eqref{eq:experiment} is of order $1/n^{1/2}$. 
    
\end{coro}

\subsection{Example \ref{example:ps exp}: Post-stratification in Survey Sampling}\label{subsec:post sampling}
\cite{fuller1966estimation} proposed
unbiased estimators for post-stratification.
\cite{holt1979post} compared the estimator $\hat\gamma_{\textup{ps}}$ with the sample mean, and showed that neither is the best universally in terms of mean squared errors. \cite{little1993post} studied post-stratification from a model-based perspective. Conditional CLTs are available for stratified permutation statistics, conditional on the realized stratification. In particular, \cite{dhault2024robust} studied stratification based on an auxiliary random variable, allowing for a random number of strata and random strata sizes, and establish a conditional CLT for stratified permutation statistics.  \cite{tuvaandorj2024combinatorial} proved a conditional CLT under a Lindeberg-type condition, which accommodates an arbitrary, and possibly growing number of strata of varying sizes, and applies even when strata are constructed using auxiliary covariates. Both results hold conditional on the stratification. However, rigorous unconditional normal approximation results for the post-stratification estimator $\hat\gamma_{\textup{ps}}$ have not been established.

The post-stratification estimator $\hat\gamma_{\textup{ps}}$ has the same randomness as the stratified sampling estimator conditional on $\bold U_1$, where $\bold U_1\in \mathcal D_1$. We can apply Corollary \ref{coro:sampling} and obtain a BEB for
$$\sup_{t\in \mathbb R}\left|\pr\left\{\frac{\hat\gamma_{\textup{ps}}-\E [\hat\gamma_{\textup{ps}}\mid \bold U_1]}{\sigma_{\textup{ps}}(\bold U_1)}\le t\mid \bold U_1 \right\}-\Phi(t)\right|,$$
where $\sigma_{\textup{ps}}^2(\bold U_1)=\var(\hat\gamma_{\textup{ps}}\mid\bold U_1)$ is the conditional variance of $\hat\gamma_{\textup{ps}}$ given $\bold {U}_1$.
When  $w_{[k]}=\nk/n$, the post-stratification estimator is conditionally unbiased: 
$$\E [\hat \gamma_{\textup {ps}}\mid \bold U_1]=\sumk \pik \nk^{-1} \sum_{X_i=k} Y_i=n^{-1}\sum_{i=1}^n Y_i=\gamma.$$ 
Define $\sigma_{\textup{ps}}^2=\var(\hat\gamma_{\textup{ps}}\mid \mathcal D_1)$ for this subsection. We have that 
 $$\sigma_{\textup{ps}}^2=\frac 1 n \sumk \frac{\nk}{n} \E\left[\frac{n_{[k]0}}{n_{[k]1}}\mid \mathcal D_1\right] S_{[k]}^2,$$
When $n_1/n\rightarrow p\in(0,1)$, we have $\E[n_{[k]1}^{-1}{n_{[k]0}}\mid \mathcal D_1]\rightarrow (1-p)/p$ as $\nk \rightarrow \infty$ by \citet[Lemma 1]{miratrix2013}.
By the triangle inequality, we have that 
\begin{align*}
    \left|\pr\left(\frac{\hat\gamma_{\textup{ps}}-\gamma}{\sigma_{\textup{ps}}}\le t \mid \bold U_1\right)-\Phi(t)\right|\le& \left|\pr\left(\frac{\hat\gamma_{\textup{ps}}-\gamma}{\sigma_{\textup{ps}}(\bold U_1)}\le \frac{\sigma_{\textup{ps}} t}{\sigma_{\textup{ps}}(\bold U_1)} \right)-\Phi\left(\frac{\sigma_{\textup{ps}} t}{\sigma_{\textup{ps}}(\bold U_1)} \right)\right|\\
    +&\left|\Phi\left(\frac{\sigma_{\textup{ps}} t}{\sigma_{\textup{ps}}(\bold U_1)} \right)-\Phi(t)\right|.
\end{align*}

Applying Corollary~\ref{coro:sampling} to the first term and using the law of total probability yields the following BEB conditional on $\mathcal D_1$; moreover, by extending $\hat\gamma_{\textup{ps}}$ to be reasonably defined on $\mathcal D_1^{c}$ (e.g., setting them to $0$), the same inequality can be stated unconditionally, but for brevity and to follow \citet{miratrix2013} we present the bound conditional on $\mathcal D_1$.

\begin{coro}[BEB for post-stratification in survey sampling]\label{coro:post-sampling}
There exists a constant $C$ such that
\begin{align}\label{ineq:post-sampling}
        \sup_{t\in \mathbb R}\left|\pr\left(\frac{\hat\gamma_{\textup{ps}}-\gamma}{\sigma_{\textup{ps}}}\le t\mid \mathcal D_1\right)-\Phi(t)\right|
        \le & C
       \sumk{\left(\frac{\nk}{n}\right)^3 \E[n_{[k]1}^{-2}\sigma_{\textup{ps}}^{-3}(\bold U_1)\mid\mathcal D_1]}\Mkthree(Y)\notag\\
       & ~~~ +\sup_{t\in \mathbb R}\left|\E\left[\Phi\left(\frac{\sigma_{\textup{ps}} t}{\sigma_{\textup{ps}}(\bold U_1)}\right)\mid\mathcal D_1\right]-\Phi(t)\right|.
    \end{align}
In particular, if the number of strata $K$ is fixed, $w_{[k]}=n_{[k]}/n$ is bounded away from $0$ for each $k$, $\max_{1\le k \le K}S_{[k]}^2$ is bounded,  
both $n^{-1}\sumk \nk S_{[k]}^2$ and $n^{-1}\sumk \nk \Mkthree(Y) $ have nonzero finite limits, and $n_1/n\rightarrow p$ with
\begin{align}\label{eq:p sampling}
    p\in \left(0,\liminf_{n\rightarrow \infty} (\min_{1\le k \le K} S_{[k]}/\max_{1\le k \le K} S_{[k]})\right),
\end{align}
then the first term in \eqref{ineq:post-sampling} is $O(1/n^{1/2})$ and the second term in \eqref{ineq:post-sampling} is $O(1/n)$. 
\end{coro}

The first term on the right-hand side of the bound \eqref{ineq:post-sampling} is the same as the right-hand side of the bound \eqref{eq:sampling} in Corollary \ref{coro:sampling} with $n_{[k]1}^{-2}\sigma_{\textup{ss}}^{-3}$ replaced by $\E[n_{[k]1}^{-2}\sigma_{\textup{ps}}^{-3}(\bold U_1 )\mid \mathcal D_1]$. The second term is the Kolmogorov distance between a mixture of normal distributions and a standard normal distribution.
To obtain the orders of the first term and the second term, the condition \eqref{eq:p sampling} imposes constraints on the heterogeneity in $Y$ across strata.
Under \eqref{eq:p sampling}, there exists a constant $\underline c$ such that  $\sigma_{\textup{ps}}^2(\bold U_1)\ge  \underline c \sigma_{\textup{ps}}^2$. In addition, another key component to study is
the variance of $\nk/n_{[k]1}$ under $\mathcal D_1$. 
\citet[Lemma 1]{miratrix2013} studied the expectation of this term. To analyze the asymptotic order of $\var(\nk/n_{[k]1}\mid \mathcal D_1)$, we derive a CLT for ${\nk}/\{{n_{[k]1}+\bold 1(n_{[k]1}=0)}\}$ (See Lemma \ref{lemma:hyper}) by the normal approximation result of the hyper-geometric distribution \citep{lahiri2007normal} and the delta method. After verifying the uniformly integrable conditions, we can obtain $\var({\nk}/{n_{[k]1}} \mid \mathcal D_1)=O(1 /n)$.

\subsection{Example \ref{example:ps}: Post-stratification in Experiments}\label{subsec:post exp}
\citet{mchugh1983post} studied post-stratification in randomized experiments. \cite{miratrix2013} studied the efficiency of the post-stratification estimator $\hat\tau_{\textup{ps}}$ conditional on $\mathcal D$, and found that this estimator is nearly as efficient as a stratified experiment with  $\{n_{[1]1},\ldots,n_{[K]0}\}$ and the difference in their variances is of the order $1/n^{2}$ when $n_{[k]1} /n_{[k]}\rightarrow p$ for some constant $p$. They did not derive a CLT for post-stratification. Define
$S_z^2=n^{-1} \sumk \nk S_{[k]z}^2$, $z=0,1,\tau$ as the weighted variance and $T_{z}^3=n^{-1}\sumk \nk \Mkthree(Y(z))$, $z=0,1$ as the weighted third moment. Define $\sigma_{\textup{ps}}^2=\var(\hat\tau_{\textup{ps}}\mid \mathcal D)$ for this subsection. We have
$$\sigma_{\textup{ps}}^2= \frac{1}{n}\sumk \frac{\nk}{n}\left(\E\left[\frac{\nk}{n_{[k]1}}\mid \mathcal D\right] S_{[k]1}^2+\E
\left[\frac{\nk}{n_{[k]0}}\mid \mathcal D\right] S_{[k]0}^2- S_\tau^2\right).$$

Conditional on $\bold U$, the only remaining randomness in $\hat\tau_{\textup{ps}}$ comes from the within-stratum treatment assignments. Hence, conditional on $\bold U$, $\hat\tau_{\textup{ps}}$ has the same distribution as the stratified difference-in-means estimator in a stratified randomized experiment. Therefore, the normal-approximation argument in Section~\ref{subsec:post sampling} carries over verbatim once the sampling-stage randomness there is replaced by the assignment-stage randomness here. For brevity, we present the bound conditional on $\mathcal D$ similar to Section~\ref{subsec:post sampling}. Corollary \ref{coro:post} below gives a BEB: 

\begin{coro}[BEB for post-stratification in randomized experiment]\label{coro:post}
There exists a constant $C$ such that
\begin{align}\label{ineq:post}
       & \sup_{t\in \mathbb R}\left|\pr\left(\frac{\hat\tau_{\textup{ps}}-\tau}{\sigma_{\textup{ps}}}\le t \mid \mathcal D\right)-\Phi(t)\right|\notag\\
        \le &
        C \sumk\left(\frac{\nk}{n}\right)^3 \left\{\E[n_{[k]1}^{-2}\sigma^{-3}_{\textup{ps}}(\bold U)\mid\mathcal D] \Mkthree(Y(1))+\E [n_{[k]0}^{-2}\sigma^{-3}_{\textup{ps}}(\bold U)\mid\mathcal D]\Mkthree(Y(0))\right\}\notag\\
        ~~~&+\sup_{t\in \mathbb R}\left|\E\left[\Phi\left(\frac{\sigma_{\textup{ps}} t}{\sigma_{\textup{ps}}(\bold U)}\right)\mid\mathcal D\right]-\Phi(t)\right|,
    \end{align}
where $\sigma^2_{\textup{ps}} = \var(\hat\tau_{\textup{ps}}\mid \mathcal D)$ is the conditional variance of $\hat\tau_{\textup{ps}}$ given $\mathcal D$. In particular, if the number of strata $K$ is fixed, $w_{[k]}=n_{[k]}/n$ is bounded away from $0$ for each $k$, $\max_{1\le k \le K}S_{[k]}^2$ is bounded, $S_1^2,S_0^2$ and $T_1^3,T_0^3$ have nonzero finite limits, and $n_1/n\rightarrow p\in(0,1)$ with
\begin{align}\label{eq:p tau}
    \frac{\min_{1\le k \le K} S_{[k]1}}{\max_{1\le k \le K} S_{[k]1}}\frac{S_1^2}{p}+\frac{\min_{1\le k \le K} S_{[k]0}}{\max_{1\le k \le K} S_{[k]0}}\frac{S_0^2}{1-p}>S_\tau^2, 
\end{align}
then the first term in \eqref{ineq:post} is $O(1/n^{1/2})$ and the second term in \eqref{ineq:post} is $O(1/n)$.
\end{coro}

Mimicking the proof of Corollary \ref{coro:post-sampling}, we provide similar regularity conditions on the second and third moments of $Y(1)$ and $Y(0)$. 
Taking into account both the variation of the treatment and control groups, we
need a condition on $p$ as in \eqref{eq:p tau}. 
This condition holds in many scenarios. For example, a sufficient condition for \eqref{eq:p tau} is that $p<\min_{1\le k\le K} S_{[k]1}/\max_{1\le k\le K} S_{[k]1}$ and $1-p<\min_{1\le k\le K} S_{[k]0}/\max_{1\le k\le K} S_{[k]0}$. This is a direct extension of the condition in Corollary \ref{coro:post-sampling}. For another example, consider the case with a constant treatment effect where $Y_i(1)=Y_i(0)+\tau$ for $i=1,\ldots,n$. In this case, $S^2_{\tau} = 0$ and \eqref{eq:p tau} holds unless the minimums of $S_{[k]1}^2$ and $S_{[k]0}^2$ over $k$ both degenerate to $0$.

\section{Application to Stratified Permutation Tests}\label{sec:permute_test}

In Examples \ref{example:exp} and \ref{example:ps} of Section \ref{sec:application}, we focus on the repeated sampling properties of the point estimators for causal effects. The other frequently adopted mode of design-based causal inference is to use the Fisher randomization test to calculate the $p$-value under Fisher's sharp null hypothesis \citetext{\citealp{rosenbaum2002observational}, \citealp{imbens2015causal}, \citealp{ding2024first}}.
We now use our BEBs to study the stratified permutation test statistics under Fisher's sharp null hypothesis.

\subsection{Stratified Permutation Test}
Consider the setting of Example \ref{example:exp}, Fisher's sharp null hypothesis states that $H_0$: $Y_i(1)-Y_i(0)=\tau_0, i=1,\ldots,n$ for a fixed and known constant $\tau_0$. Under $H_0$,   $Y_i(0) =  Y_i- Z_i \tau_0$ is known for every unit. 
Let $\bold R$ denote the score vector generated by the potential outcomes under control $\bold Y(0)$, for example, $\bold R = \bold Y(0)$ or $\bold R$ equals the rank vector of $\bold Y(0)$. Let $\bold Z$ denote the vector of treatment conditions with $n_{[k]1}$ treated and $n_{[k]0}$ control units in the $k$-th stratum, $k=1,\ldots,K$. We can use the following statistic
\citep{rosenbaum2002observational}:
\begin{align}\label{eq:ZR}
    \bold Z^\T \bold R=\sumk\sumik Z_i R_i
\end{align}
to test $H_0$. 
We permute $\bold Z$ with $\pi\sim \textup{Uniform}(\Pi_K)$ and obtain $\bold Z_{\pi}=[Z_{\pi(1)},\ldots,Z_{\pi(n)}]^\T$. The randomness of $\pi$ generates the distribution of $\bold Z_{\pi}^\T\bold R$, which is the basis for calculating the $p$-value.

\subsection{Stratified Permutation Test with an Instrumental Variable}
We can also extend the stratified permutation test to deal with noncompliance in randomized experiments. Let $D_i$ denote the received dose of the treatment. We are interested in testing the effect of $D$ on the outcome $Y$. Since both the received dose and the outcome are results of the treatment assignment, we define $D_i(0)$ and $D_i(1)$ as the potential received dose of the treatment under control assignment and treatment assignment, respectively, and $Y_i(0)$ and $Y_i(1)$ as the potential outcome under control assignment and treatment assignment, respectively. Viewing $Z$ as an instrumental variable, we consider the linear instrumental variable model $Y_i(1)-Y_i(0) = \beta (D_i(1)-D_i(0))$ \citep{rosenbaum2002observational,imbens2005robust}. 
Let $\bold Y$ and $\bold D$ denote the vector of the observed outcomes and the vector of received doses, respectively, for all subjects. 
Under $H_0:\beta=\beta_0$, we have $Y_i - \beta_0 D_i = Y_i(1)-\beta_0 D_i(1) = Y_i(0)-\beta_0 D_i(0)$. 
Let $\bold R$ be the score vector generated by the adjusted outcomes $\bold Y-\beta_0 \bold D$, for example, $\bold R = \bold Y-\beta_0 \bold D$ or $\bold R$ equals the rank vector of $\bold Y-\beta_0 \bold D$. Then we can use a statistic in the same form as \eqref{eq:ZR} to test $H_0$.

\subsection{BEB for the Stratified Permutation Test Statistic}\label{subsec: product form}
Define the stratum-specific means $\bar Z_{[k]}=\nk^{-1}\sumik Z_i$ and $\bar R_{[k]}=\nk^{-1} \sumik R_i$.
The expectation of $\bold Z_{\pi}^\T \bold R$ equals 
$$\mu = \E [\bold Z_{\pi}^\T \bold R]=\sumk\nk \bar Z_{[k]}\bar R_{[k]}.$$
The variance of $\bold Z_{\pi}^\T \bold R$ equals $$\sigma^2=\var(\bold Z_{\pi}\bold R)=\sumk \frac {\nk^2} {\nk-1}\Mktwo(Z)\Mktwo(R).$$ 
With the definitions of $\mu$ and $\sigma^2$ above, we apply Theorem \ref{theorem:beb} to this product form permutation statistic, and obtain the following BEB:
\begin{coro}[BEB for the stratified permutation test]\label{coro:IV}
Assume $\pi\sim \textup{Uniform}(\Pi_K)$. There exists a constant $C$ such that 
    \begin{align*}
    \sup_{t\in \mathbb R}\left|\pr\left(\frac{\bold Z_{\pi}^\T \bold R-\mu}{\sigma}\le t\right)-\Phi(t)\right|
    \le  \frac C {\sigma^3}\sumk  \nk\Mkthree (Z) \Mkthree (R).
    \end{align*}  
\end{coro}

To obtain the CLT in \citet{dhault2024robust} for the product form permutation statistics, we only need $\sigma^2$ to tend to a constant and the third-moment term $\sumk \nk \Mkthree(Z)\Mkthree(R)\rightarrow 0$. With Corollary \ref{coro:IV}, the condition specified in \citet[Lemma 4]{dhault2024robust} is redundant: we do not need to assume that the fourth-moment term in the condition (d) of \citet[Lemma 4]{dhault2024robust} tends to zero. \cite{tuvaandorj2024combinatorial} also relaxed the condition by providing a Lindeberg-type CLT.
The $W_{A, \pi^{\dagger}}$ was introduced in \citet{tuvaandorj2024combinatorial} for establishing CLTs; we adapt this construction in Step \ref{step:one} to obtain BEB. They verified that $\E [W_{A,\pi}]-\E[W_{A,\pi}f(W_{A,\pi})]=\E [W_{A,\pi}]-\E[W_{A,\pi^\dagger}f(W_{A,\pi^\dagger})]\rightarrow 0$ to obtain a CLT under some regularity conditions. \cite{tuvaandorj2024combinatorial} did not introduce the zero-bias transformation in Lemma \ref{thm:zerobias}. We use Stein's method and construct the zero-bias transformation to obtain the BEB in Corollary \ref{coro:IV} that can attain the classic order $1/n^{1/2}$ under regularity conditions.

\section{Extensions}\label{sec:additional resutl}
We have two extensions. 
First, we extend Theorem \ref{theorem:beb} to the multivariate stratified linear permutation statistic. The extension is useful when we are interested in estimating multiple parameters simultaneously. Second, we prove a BEB of the multivariate stratified linear permutation statistic over convex sets as, extending the result of \citet{fang2015rates} in the non-stratified setting. The extension is useful when we are interested in analyzing quadratic form statistics. We give the details below. 
\subsection{BEBs for Linear Combination of Multivariate Permutation Statistics}

In this subsection, we extend Theorem \ref{theorem:beb} to linear combinations of the multivariate stratified linear permutation statistic.

For $H$ matrices $\{G_h\}_{h=1}^H$ and stratified permutation $\pi\sim \textup{Uniform}(\Pi_K)$, define
$$\Gamma=(W_{G_1,\pi},W_{G_2,\pi},\ldots,W_{G_H,\pi})^\T.$$

We can write $\Gamma$ in the form of the trace inner product. Define a permutation matrix 
$$P=[\boldsymbol 1\{i=\pi(j)\}]_{1\le i,j\le n}.$$
 For any $h=1,\ldots,H$, we have $$\tr (G_h P)=\sum_{i=1}^n\sum_{j=1}^n g_{h,ij}\bold 
 1\{j=\pi(i)\}=\sum_{i=1}^n g_{h,i\pi(i)}=W_{G_h,\pi},$$ 
so
$$\Gamma=(\tr(G_1 P),\tr (G_2P),\ldots,\tr(G_H P))^\T.$$
Center $\Gamma$ as follows:
$$\Gamma^0=(\tr(G_1^0 P),\tr(G_2^0 P),\ldots,\tr(G_H^0 P))^\T,$$
where $G_h^0$ is the centered $G_h$ and $\E [G_h^0]=0$. Define $V=\cov(\Gamma)$, 
and standardize $\Gamma$ as $\tilde \Gamma=V^{-1/2}(\Gamma-\E[\Gamma])$ when $V$ is nondegenerate. Denote $\vec(M)$ as vectorizing $M$ by stacking its column vectors.
We can use the following proposition to obtain the explicit form of  $\tilde \Gamma$.

 \begin{prop}\label{prop:multi-scaling}
We have
     $$\tilde \Gamma=V^{-1/2}(\Gamma-\E [\Gamma])=(\tr (\tilde G_1 P),\tr (\tilde G_2 P),\ldots,\tr (\tilde G_H P))^\T,$$
     where 
\begin{align*}\left(\begin{matrix}\vec(\tilde G_1)^\T\\
   \vec(\tilde G_2)^\T\\
    \vdots\\
    \vec(\tilde G_H)^\T
 \end{matrix}\right)=V^{-1/2}\cdot \left(\begin{matrix} \vec (G_1^0)^\T\\
    \vec (G_2^0)^\T\\
    \vdots\\
    \vec  (G_H^0)^\T
 \end{matrix}\right).
 \end{align*}
 \end{prop}

Recall $M_n^\textup{s}$ defined in \eqref{eq:Mns}.
By Proposition \ref{prop:multi-scaling}, we only need to consider $G_h\in  M_n^\textup{s}$. Since for any matrices $\{G_h\}_{h=1}^H$ and permutation $\pi$, we always have $\tilde G_h\in  M_n^\textup{s}, h=1,\ldots,H$.

Let $b \in \mathbb{R}^H$ be a unit vector, i.e., $\|b\|_2 = 1$. Consider a vector $\Gamma$ with $H$ components, where each component $G_h \in M_n^\textup{s}$ for $h = 1, 2, \dots, H$. We aim to study the asymptotic behavior of $b^\T \Gamma$.

Denote $$\langle G_h,G_l\rangle_K=\sumk(\nk-1)^{-1}\sumijk g_{h,ij}g_{l,ij}.$$ 
To ensure $\Gamma$ has been standardized, it is equivalent to impose the following condition:
 
\begin{condition}\label{cond:multibeb}
    $G_1,G_2,\ldots,G_H\in M_n^\textup{s}$ and $\langle G_h,G_l\rangle_K=0,\forall h\neq l $.
\end{condition}

The assumption $G_h\in M_n^\textup{s},\quad h=1,\ldots,H$ ensures that each element of $\Gamma$ is standardized, and $\langle G_h,G_l \rangle_K=0$ for all $h\neq l$, ensures the covariance between distinct components of the vector $\Gamma$ is zero.

\begin{prop}\label{prop:cond_equ}
Under Condition \ref{cond:multibeb}, we have that $\E  [\Gamma] =\boldsymbol 0_H$ and $\var(\Gamma)=\boldsymbol I_H.$
\end{prop}
Now, we present BEB on the linear combination of multivariate stratified linear permutation statistics.

\begin{coro}\label{coro:multibeb}
    For matrix $G_1,\ldots,G_H$ satisfying Condition \ref{cond:multibeb} and for any $\|b\|_2=1, b\in \mathbb R^H$, there exists a universal constant $C$ such that
    $$\sup_{t\in \mathbb R}\left| \pr(b^\T \Gamma\le t)-\Phi(t)\right|\le C\max_{1\le k\le K}\max_{i,j\in\Ik}\left|\sumh b_h g_{h,ij}\right|.$$
\end{coro}

    The right-hand sides of the inequalities in Corollary \ref{coro:multibeb} depend on the maximal element of $\sumh b_h G_h$.

\subsection{Permutational BEB over convex sets}
In this subsection, we extend \citet[Theorem 2.1]{fang2015rates} to stratified linear permutation statistics. Denote $\mathcal C$ as the collection of all Borel convex sets. Let $\xi_H$ denote an $H$-dimensional standard normal distribution. We have the following theorem.

\begin{theorem}\label{thm:convexBEB}
    Assume $|g_{h,ij}|\le B_n$ for $h=1,\ldots,H$ and $i,j\in \Ik, k=1,\ldots,K$. Assume Condition \ref{cond:multibeb}. There exists a universal constant $C>0$, such that 
    \begin{align*}
 & \sup_{V\in \mathcal C}|\pr(\Gamma\in V)-\pr(\xi_H \in V)|\\
 &\le CH^{{13}/{4}}B_n(n B_n^2+K)+CH^{3/4}B_n+CH^{13/8}(n-K)^{1/4}B_n^{3/2}+CH^{11/8}(n-K)^{1/2}B_n^2.
    \end{align*}
    When $B_n=O((n-K)^{-1/2})$, we have
    $$\sup_{V\in \mathcal C}|\pr(\Gamma\in V)-\pr(\xi_H \in V)|= O\left(\frac{H^{13/4}K}{(n-K)^{1/2}}\right).$$
\end{theorem}

When $K=1$, the bound degenerates to the non-stratification case, which is the same as \citet[Theorem S2]{shi2022berry}.

\section*{Acknowledgements}
Pengfei Tian was supported by the Beijing Natural Science Foundation (QY23081).
Peng Ding was supported by the U.S. National Science
Foundation (grants \# 1945136 and 2514234). Fan Yang is also affiliated
with the Yanqi Lake Beijing Institute of Mathematical
Sciences and Applications.

\bibliographystyle{main}
\bibliography{main}

\appendix
\phantomsection 
\setcounter{equation}{0}
\renewcommand{\theequation}{S\arabic{equation}}
\setcounter{table}{0}
\renewcommand{\thetable}{S\arabic{table}}
\setcounter{figure}{0}
\renewcommand{\thefigure}{S\arabic{figure}}
\setcounter{theorem}{0}
\renewcommand{\thetheorem}{S\arabic{theorem}}
\setcounter{lemma}{0}
\renewcommand{\thelemma}{S\arabic{lemma}}
\setcounter{condition}{0}
\renewcommand{\thecondition}{S\arabic{condition}}
\setcounter{remark}{0}
\renewcommand{\theremark}{S\arabic{remark}}
\setcounter{prop}{0}
\makeatletter
\renewcommand{\theprop}{S\arabic{prop}}
\makeatother
\setcounter{coro}{0}
\makeatletter
\renewcommand{\thecoro}{S\arabic{coro}}
\makeatother

\makeatletter
\providecommand{\theHtheorem}{S\arabic{theorem}}
\providecommand{\theHprop}{S\arabic{prop}}
\providecommand{\theHlemma}{S\arabic{lemma}}
\providecommand{\theHcoro}{S\arabic{coro}}
\makeatother

\noindent
\newpage

\renewcommand{\thepage}{S\arabic{page}} 
\setcounter{page}{1}

\begin{center}
\Huge
Supplementary Material
\end{center}

 Section \ref{sec:basic tool}  contains basic tools for Stein's method and technical lemmas for proofs.

 Section \ref{sec:proof-main} provides proofs for the results in the main paper. 
 
 Section \ref{sec:proof-supp} provides proofs for the results in the supplementary material.

\section{Basic Tools}\label{sec:basic tool}
\subsection{Basic Notation and Lemmas}
\subsubsection{Permutation and Stein pair}\label{sec:construct stein pair}
For matrix $A=[a_{ij}]$, recall $\mu_A=\E [W_{A,\pi}]$, $\sigma^2_A=\var(W_{A,\pi})$ and $\beta_{A[k]}=\sumijk |a_{ij}^0|^3/\sigma_A^3$. 
Recall $A^0=[a_{ij}^0=a_{ij}-\bar a_{[k]i\cdot}-\bar a_{[k]\cdot j}+\bar a_{[k]\cdot\cdot}]$, $A^\textup{s}=[a^0/\sigma_A]$, and $A'=[a'_{ij}=a_{ij}^\textup{s}\bold 1(|a_{ij}^\textup{s}|\le 1/2)]$. Define $(S,S')$ as an exchangeable pair when $(S,S')$ and $(S',S)$ have the same distribution.

Recall the definition of Stein pair \citep{chen2011normal}:
\begin{definition}[Stein Pair]
    A pair of random variables $(W,W')$ is $\lambda$-Stein pair if and only if $(W,W')$ and $(W',W)$ have the same distribution and $\E[W'\mid W]=(1-\lambda)W$.
\end{definition}

We randomly sample a stratum index $B$ from the set $\{1,\ldots,K\}$, with the probability of selection of each stratum index proportional to the size of the stratum minus one, i.e.,  
$\pr(B=k)={(\nk-1)}/{(n-K)}.$

Let $\rho_{i,j}$ denote the permutation that swaps elements $i$ and $j$ while leaving other positions fixed.
Within stratum $B$, we randomly sample two indices $I,J$ without replacement, i.e. the conditional probability of selecting indices $I=i, J=j$ given $B=k$ is $$\pr(I=i, J=j\mid B=k)=\frac{\bold {1}(i,j\in\Ik ,i\neq j)}{\nk(\nk-1)}.$$

 Define $\pi''=\pi \rho_{I,J}$ based on the original random permutation $\pi$ and the random transposition $\rho_{I,J}$. Define $W_{A,\pi''}=\sumk\sumik a_{i\pi''(i)}$. Recall $\mu_A=\E [W_{A,\pi}]=\sumk \sumijk a_{ij}/\nk$. We show that $W_{A,\pi}$ and $W_{A,\pi''}$ form a Stein pair in the following proposition:
\begin{prop}\label{prop:stein pair}
For $\pi\in \textup{Uniform}(\Pi_K)$ and $\pi''=\pi \rho_{I,J}$,  $(W_{A,\pi},W_{A,\pi''})$ and $(W_{A,\pi''},W_{A,\pi})$ have the same distribution and 
$$\E[W_{A,\pi''} \mid W_{A,\pi}]=\left(1-\frac{2}{n-K}\right)W_{A,\pi}+\frac{1}{n-K}\mu_A .$$
When $\mu_A=0$, the pair $(W_{A,\pi},W_{A,\pi''})$ forms a $2/(n-K)$-Stein pair.
\end{prop} 
When $K=1$, Proposition \ref{prop:stein pair} reduces to the example in \citet[Example 2.3]{chen2011normal}. 
When $\mu_A=0$, we can obtain $\cov(W_{A,\pi''}, W_{A,\pi})=\{1-2(n-K)^{-1}\}\var(W_{A,\pi})$. Moreover, we can compute the variance of $W_{A,\pi}-W_{A,\pi''}$ under $\mu_A=0$:  
\begin{align}\label{eq:var diff W_A,pi}
    \var(W_{A,\pi}-W_{A,\pi''})=\left\{2-2\left(1-2\cdot\frac{2}{n-K}\right)\right\} \var (W_{A,\pi}) =\frac{4\sigma^2_A}{n-K}.
\end{align}

The formula in \eqref{eq:var diff W_A,pi} shows that when fixing the variance of $W_{A,\pi}$ and the sample size $n$, increasing the number of strata $K$ results in a larger covariance for the difference between the Stein pair $(W_{A,\pi}, W_{A,\pi''})$.

We review the concentration inequality for exchangeable pairs as follows.

\begin{lemma}[Lemma 2.1 of \cite{chen2015error}]\label{lemma:concentration}
    Suppose $(S,S')$ is an exchangeable pair of square integrable random variables and satisfies the following approximate linearity condition 
    $$\E [S'-S\mid S]=-\lambda S+L$$
    for a positive number $\lambda$ and a random variable $L$. Then, for $a<b$, we have 
    \begin{align*}
        \pr(S\in [a,b])\le& \frac{ \E[|S|]+\E[|L|]/\lambda}{ \E [S^2] -\E [|SL|]/\lambda-1/2}\left(\frac{b-a}{2}+\delta\right)\\
        &~~~+\frac{1}{\E [S^2]- \E [|SL|]/\lambda-1/2}\sqrt{\var \left( \E\left[\frac 1 {2\lambda}(S'-S)^2 \bold{1}(|S'-S|\le \delta)\mid S\right] \right)},
    \end{align*}
    where $$\delta=\frac{\E[|S'-S|^3]}{\lambda}$$
    provided that $\E [S^2] -\E [|SL|]/\lambda-1/2>0.$
\end{lemma}

We provide bounds on some terms in Lemma \ref{lemma:concentration} with $S=W_{A,\pi}$ and $S'=W_{A,\pi''}$ in Theorem \ref{thm:bound for delta and B02} below. 

\begin{theorem}\label{thm:bound for delta and B02}
For $\pi\in \textup{Uniform}(\Pi_K)$ and $\pi''=\pi\rho_{I,J}.$ Define $\lambda=2/(n-K)$. 
\begin{enumerate}[label=(\roman*),leftmargin=*, align=left]
\item   We have
$$\frac{\E [|W_{A,\pi''}-W_{A,\pi}|^3]}{\lambda}\le 32 \sumk \sumijk |a_{ij}|^3/\nk.$$
\item We have 
\begin{align*}
 &\var\left(\E \left[ \frac{1}{2\lambda}(W_{A,\pi}-W_{A,\pi''})^2 \bold 1(|W_{A,\pi}-W_{A,\pi''}|\le \delta )\mid W_{A,\pi}\right]\right)\\
\le& 56\times 32 \left(\sumk \sumijk |a_{ij}|^3/\nk\right)^2.
\end{align*}
\end{enumerate}

\end{theorem}

\subsubsection{Zero-bias Transformation Construction}\label{sec:construct zero-bias}

In this section, we review lemmas that construct the zero-bias transformation and explain how the zero-bias transformation helps bound the Wasserstein distance to the standard normal distribution.
We then establish the properties of the permutation $\pi^\ddagger$ used to obtain a zero-bias transformation of $W_{A,\pi}$, based on the Stein pair $(W_{A,\pi}, W_{A,\pi''})$ in Proposition~\ref{prop:stein pair}.

Given a Stein pair $(Y',Y'')$, suppose there exists another pair of random variables $(Y^\dagger, Y^\ddagger)$ satisfying \eqref{eq:y^d y^dd<- y' y''} below.
Then a zero-bias transformation of $Y'$ can be constructed as follows.

\begin{lemma}[Proposition 4.6 in \citet{chen2011normal}]\label{lemma:chen prop4.6}
Let $Y^{\prime}$, $Y^{\prime \prime}$ be a $\lambda$-Stein pair with $\var(Y')=\sigma^2 \in(0, \infty)$. If
\begin{align}\label{eq:y^d y^dd<- y' y''}
\pr\left(Y^\dagger =y^{\prime}, Y^\ddagger =y''\right)=\frac{\left(y'-y''\right)^2}{2 \lambda \sigma^2} \pr\left(Y'=y', Y''=y''\right)
\end{align}
and $U \sim \textup{Uniform}([0,1])$ is independent of $Y^{\dagger}, Y^{\ddagger}$, then $Y^*=U Y^{\dagger}+(1-U) Y^{\ddagger}$ is a zero bias transformation of $Y^{\prime}$.
\end{lemma}

In many applications, a Stein pair $(Y',Y'')$ admits a representation of the form
\[
(Y',Y'')=\psi(\boldsymbol I,\Xi_1,\ldots,\Xi_n)
\]
for some function $\psi$, where $\boldsymbol I$ and $\Xi_1,\ldots,\Xi_n$ are random variables.
For $(W_{A,\pi}, W_{A,\pi''})$, we have $\boldsymbol I=(B,I,J)$ and $\Xi_\alpha=\pi(\alpha)$ for $\alpha=1,\ldots,n$.
The following lemma shows that condition \eqref{eq:y^d y^dd<- y' y''} can be verified by checking the corresponding condition on $(\boldsymbol I,\Xi_1,\ldots,\Xi_n)$.

\begin{lemma}[Lemma 4.4 in \citet{chen2011normal}]\label{lemma:chenlemma4.4}
Let $(Y',Y'')$ be a Stein pair and suppose there exist  an $\mathbb{R}^2$-valued function $\left(y^{\prime}, y^{\prime \prime}\right)=\psi(\mathbf{i}, \xi_\alpha, \alpha =1,\ldots,n)$ and random variables $\mathbf{I}$, $\Xi_\alpha$, $\alpha =1,\ldots,n$ such that
$$
\left(Y^{\prime}, Y^{\prime \prime}\right)=\psi\left(\mathbf{I}, \Xi_\alpha, \alpha =1,\ldots,n\right).
$$
Furthermore, if there exist random variables $\mathbf{I}^{\dagger}, \Xi_\alpha^{\dagger}, \alpha =1,\ldots,n$ such that
$$
\pr\left(\mathbf{I}^\dagger =\mathbf{i}, \Xi^\dagger_\alpha=\xi_\alpha, \alpha =1,\ldots,n\right)=\frac{\left(y^{\prime}-y^{\prime \prime}\right)^2}{\E\left[Y^{\prime}-Y^{\prime \prime}\right]^2} \pr\left(\mathbf{I}=\mathbf{i}, \Xi_\alpha=\xi_\alpha, \alpha =1,\ldots,n \right),
$$
then the pair
$$
\left(Y^{\dagger}, Y^{\ddagger}\right)=\psi\left(\mathbf{I}^{\dagger}, \Xi_\alpha^{\dagger}, \alpha =1,\ldots,n\right)
$$
satisfies \eqref{eq:y^d y^dd<- y' y''}.
\end{lemma}

For $(W_{\pi^\dagger}, W_{\pi^\ddagger})$, we take 
$\boldsymbol I^\dagger=(B^\dagger, I^\dagger, J^\dagger, P^\dagger, Q^\dagger)$
and $\Xi_\alpha^\dagger=\pi^\dagger(\alpha)$ for $\alpha=1,\ldots,n$.
We apply Lemma~\ref{lemma:chenlemma4.4} to prove Lemma~\ref{thm:zerobias}.
We then state the Wasserstein bound for normal approximation based on the zero-bias transformation.

\begin{lemma}[Theorem 4.1 in \cite{chen2011normal}]\label{lemma:W-distance zerobia}
    Let $W$ be a mean-zero, variance $1$ random variable, and let $W^*$ be its zero-bias distribution defined on the same probability space. We have
    $$d_{\textup{W}}(W,\mathcal N(0,1))\le 2\E[|W^*-W|].$$
\end{lemma}

Recall the $\pi^\dagger$ in Section \ref{subsec:pi dagger construct}. It satisfies the following properties:
\begin{prop}\label{prop:permutation property}
Given $B^\dagger=k,I^\dagger=i,J^\dagger=j,P^\dagger=p,Q^\dagger=q$, $\pi^\dagger$ is a permutation that satisfies the following conditions:
\begin{itemize}
    \item $\pi^\dagger(m)=\pi(m)$ for all $m\notin\{i,j,\pi^{-1}(p),\pi^{-1}(q)\}$.
    \item $\{\pi^\dagger(i),\pi^\dagger(j)\}=\{p,q\}$.

\end{itemize}
Furthermore, 
$$\pr(\pi^\dagger(m)=\xi^\dagger_m,m\notin \{i,j\}, m\in \Ik )=\frac 1{(\nk-2)!},$$
for all distinct $\xi_m^\dagger, m\notin\{i,j\},m\in\Ik $ with $\xi_m^\dagger\notin\{p,q\},\xi_m^\dagger\in \Ik $.    
\end{prop}

Proposition \ref{prop:permutation property} extends \citet[Lemma 4.5]{chen2011normal}.  The property of $\pi^\ddagger$ can be used to verify the condition of the construction of a zero-bias transformation in Lemma \ref{lemma:chenlemma4.4}.
\begin{remark}
   Recall the $\pi^\ddagger$ in Section \ref{subsec:pi dagger construct}. Given a realization of $(I^\dagger,J^\dagger,P^\dagger,Q^\dagger)$, denoted as $(i,j,p,q)$, $W_{A,\pi^\dagger}$ contains the term $a_{ip}+a_{jq}$. The transposition $\rho_{I^\dagger, J^\dagger}$ changes $a_{ip}+a_{jq}$ in $W_{A,\pi^\dagger}$ to $a_{jp}+a_{iq}$ in $W_{A,\pi^\ddagger}$, while leaving other terms unchanged. Consequently, the quantity $(a_{ip}+a_{jq}-a_{iq}-a_{jp})^2=(W_{A,\pi^\dagger}-W_{A,\pi^\ddagger})^2$ serves as a measure of the distance between $W_{A,\pi^\dagger}$ and $W_{A,\pi^\ddagger}$ given the realization $(i,j,p,q)$. By considering the design of the sampling process for $(I^\dagger, J^\dagger, P^\dagger, Q^\dagger)$, the sampling probability of $(I^\dagger, J^\dagger, P^\dagger, Q^\dagger)$ within stratum $B^\dagger$ is proportional to the squared Euclidean distance between $W_{A,\pi^\dagger}$ and $W_{A,\pi^\ddagger}$. 
\end{remark}

\subsubsection{Technical Lemmas}\label{sec:A1}
Recall $\beta_{A[k]}=\sumijk |a_{ij}^0|^3/\sigma_A^3$ introduced in Section \ref{subsec:simplication}. We give the following propositions and lemmas which are the stratified analogues of those in \cite{chen2011normal}. 
\begin{prop}[$L_1$ Bound]\label{prop:L1}
For any matrix \(A \in {M}_n^\textup{s} \), we have
\[
\E [|W_{A,\pi}^*-W_{A,\pi}|] \le 80 \sum_{k=1}^K \frac{\beta_{A[k]  }}{\nk}.
\]
\end{prop}

\begin{lemma}[Truncation Lemma]\label{lemma:truncation}
For $n\ge 2$, $A\in  M_n^\textup{s}$, recall $A'=[a_{ij}']$ where $a_{ij}'=a_{ij}^\textup{s}\bold 1(|a_{ij}^\textup{s}|\le 1/2)$, and $\mu_{A'}=\E [W_{A',\pi}]$. There exists $c_1\ge 1$ such that
$$\pr(W_{A,\pi}\neq W_{A',\pi})\le c_1\sumk \frac {\beta_{A[k]  }}{\nk},\quad  |\mu_{A'}|\le c_1\sumk\frac{\beta_{A[k]  }}{\nk}.$$
Additionally, there exist constants $\varepsilon_1$ and $c_2$ such that when $\sumk {\beta_{A[k]}}/{\nk}\le \varepsilon_1$
we have 
$$|\sigma_{A'}^2-1|\le c_2\sumk \frac{\beta_{A[k]  }}{\nk},~~(A')^\textup{s}\in M_n^1~~and~~\beta_{A'[k] }\le c_2\beta_{A[k]  }.$$
\end{lemma}

\begin{lemma}[Scaling Lemma]\label{lemma:scaling}
	There exist constants $\varepsilon_2$ and $c_3$ such that if $A\in M_n^\textup{s}$ for some $n\ge 2$, then whenever $\sumk\beta_{A[k]  }/\nk\le \varepsilon_2$, we have 
	$$\sup_{t\in\mathbb R}|\pr(W_{A,\pi}\le t)-\Phi(t)|\le \sup_{t\in \mathbb R}|\pr(W_{({A'})^\textup{s},\pi}\le t)-\Phi(t)|+c_3\sumk \beta_{A[k]  }/\nk.$$
\end{lemma}

Let $A$ be an $n\times n$ matrix, and let $\mathcal R,\mathcal C \subset \Ik$ be index sets with $|\mathcal R |=|\mathcal C |=l$.
Define $D$ as the $(n-l)\times (n-l)$ submatrix obtained by deleting from $A$ the rows indexed by $\mathcal R$ and the columns indexed by $\mathcal C$.
 
\begin{lemma}[Deleting Lemma]\label{lemma:inductive}
	There exist $\varepsilon_3>0$ and $c_4\ge 1$ such that if $\nk\ge 2$, $l\in\{2,3,4\} $ and $A\in M_n^1$, we have $|\mu_D|\le 8$. Further if $\beta_{A[k]  }/\nk\le \varepsilon_3$, we have $\sigma_D^2\le c_4$. 
\end{lemma}

\begin{lemma}[Lemma 1 in \cite{miratrix2013}]\label{lemma: miratrix lemma1}
Let $W$  be a hypergeometric random variable with parameters $n_1$, $w$, and $n$, representing a sample of size $n_1$ drawn from a population of $n$ items, where $w = np$ are successes and $p$ is the probability of success.
Then $Y=(n_1 / W) \bold {1}(W>0)$ satisfies
$$
-\frac{2}{p}(1-p)^{n_1} \leqslant \mathbb{E}[Y]-\frac{1}{p} \leqslant \frac{4}{p^2} \frac{1}{n_1}-\frac{1}{p} \frac{1}{n_1+1}+\max \left\{\left(\frac{n_1}{2}-\frac{4}{p^2 n_1}\right) \exp \left(-\frac{p^2}{2} n_1\right), 0\right\} .
$$
\end{lemma}

\begin{lemma}\label{lemma:weight S}For positive sequence $\{S_{[k]}\}_{k=1}^K$ and positive weights $\{w_{[k]}\}_{k=1}^K$ with $\sumk w_{[k]}=1$, we have
$$\left(\sumk w_{[k]} S_{[k]}\right)^2\ge \frac{\min_{1\le k\le K}S_{[k]}}{\max_{1\le k\le K}S_{[k]}}\sumk w_{[k]} S_{[k]}^2 .$$    
\end{lemma}

\begin{lemma}\label{lemma:g function}
The function $g(x)=\Phi(x^{-1/2} \sigma t)$ has derivatives
\begin{eqnarray*}
g'(x)&=&-\frac 1 2  x^{-3/2} \phi\left(\frac{\sigma t}{\sqrt{ x}}\right)\sigma t,\\
g''(x)&=&\frac 3 4 \sigma t x^{-5/2}\phi\left(\frac{\sigma t}{\sqrt{x}} \right)-\frac 1 4 \sigma^3 t^3 x^{-7/2} \phi\left(\frac {\sigma t}{\sqrt{x}}\right).
\end{eqnarray*}
\end{lemma}

Recall the definition of $n_{[k]1}$ in Examples \ref{example:ps exp} and \ref{example:ps}. 
Define the variance of hypergeometric distribution variable $\sigma_{\mathrm{hyp}}^2=\var(n_{[k]1})= \nk(n-\nk)n_1(n-n_1)/n^3$.

\begin{lemma}[\citet{lahiri2007normal}]\label{lemma:hyper_beb}
Recall the $n_{[k]1}$ from Examples \ref{example:ps exp} and \ref{example:ps}.
    If $$ \frac{\sqrt{\nk(n-\nk)n_1(n-n_1)/n^3}}{10\max\{\frac{\min(n_1/n,1-n_1/n)+4}{4(1-\min(n_1/n,1-n_1/n))},2\}}>1,$$
then there exist universal constants $C_1,C_2$, such that 
    \begin{align*}
      &\sup_{t\in\mathbb R}\left|\pr\left(\frac{n_{[k]1}-n_1\pik}{\sigma_{\mathrm{hyp}}}\le t\right)-\Phi(t)\right|\\ 
      \le &\frac{C_1}{\sigma_{\mathrm{hyp}}}\frac{1+t^2}{(1-\pik) \bold 1(t\le 0)+\pik \bold 1(t\ge 0)}\exp(-C_2 t^2 \{(1-\pik) \bold 1(t\le 0)+\pik \bold 1(t\ge 0)\}^2).  
    \end{align*}
\end{lemma}

\begin{lemma}\label{lemma:hyper}
Recall $n_{[k]1}$ from Examples \ref{example:ps exp} and \ref{example:ps}.
When $ \pik\rightarrow \pik^*\in(0,1)$ and $n_1/n\rightarrow p\in(0,1)$ as $n\rightarrow \infty$, we have
$${\nk}^{1/2}\left(\frac{\nk}{n_{[k]1}+\bold 1(n_{[k]1}=0)}-\frac 1 p\right)\xrightarrow{d}\mathcal N\left(0,\frac{(1-\pik^*)(1-p)}{p^3}\right).$$ 
  Furthermore,
    $$\var\left(\frac{\nk}{n_{[k]1}} \mid \mathcal D_1\right)=O\left(\frac 1 n\right).$$
\end{lemma}

\begin{lemma}[Rosenthal inequality \citep{rosenthal1970subspaces}]\label{lemma:rosenthal}
    Given independent, mean-zero random variables $\{X_i\}_{i=1}^n$, for any $p\ge 2$ there exists a constant $C_p$ depending only on $p$ such that
    $$\E\left[\left|\sum_{i=1}^n X_i\right|^p\right]\le C_p \left\{\sum_{i=1}^n \E|X_i|^p +\left(\sum_{i=1}^n \E X_i^2\right)^{p/2}\right\}.$$
\end{lemma}
\begin{lemma}\label{lemma:converge to zero with p lar 2}
Let $X_{ij},i,j\in \Ik,k=1,\ldots,K$ be IID, mean-zero random variables with $\E|X_{ij}|^p<\infty$ for some constant $p>2$. As $n\rightarrow\infty$, we have $$\frac 1 n\sumk\sumijk \frac{X_{ij}}{\nk}\xrightarrow{a.s.}0.$$
In particular, the conclusion holds even when the number of strata $K=K(n)$ diverges with $n$.
\end{lemma}

\section{Proofs of the results in the main paper}\label{sec:proof-main}

\subsection{Proof of Proposition~\ref{prop:var}}
\begin{proof}
We start by calculating the expected value of $W_{A,\pi}$:
$$\E \left[\sumk \sumik a_{i\pi(i)}\right]= \sumk \E\left[\sumik a_{i\pi(i)}\right]=\sumk \nk^{-1} \sumijk a_{ij}.$$
Then we can compute the variance of $W_{A,\pi}$:
\begin{align}\label{eq:var W_A}
    \var\left(\sumk \sumik a_{i\pi(i)}\right)=&\var\left(\sumk \sumik a_{i\pi(i)}^0\right)\notag\\
    =&\sumk \left(\sumik \var (a_{i\pi(i)}^0)+\sum_{i,j\in\Ik ,i\neq j}\cov (a_{i\pi(i)}^0,a_{j,\pi(j)}^0)\right).
    \end{align}
For the first term in \eqref{eq:var W_A}, we have:
$$\sumik \var (a_{i\pi(i)}^0)=\sumik\nk^{-1} \sumjk(a_{ij}^0)^2=\nk^{-1}\sumijk (a_{ij}-\bar a_{[k]  i\cdot}-\bar a_{[k]  \cdot j}+\bar a_{[k]  \cdot\cdot})^2.$$
For the second term in \eqref{eq:var W_A}, we have:
\begin{align*}
    \sum_{i,j\in\Ik ,i\neq j}\cov (a_{i\pi(i)}^0,a_{j,\pi(j)}^0)=&\sum_{i,j\in\Ik ,i\neq j}\E [a_{i\pi(i)}^0a_{j,\pi(j)}^0]\\
    =&\sum_{i,j\in\Ik ,i\neq j} \frac{1}{(\nk-1)\nk}\sum_{s,t\in\Ik ,s\neq t}a_{is}^0a_{jt}^0\\
    =&\frac{1}{(\nk-1)\nk}\sum_{i,s\in\Ik }a_{is}^0\sum_{j\in \Ik ,j\neq i}\sum_{t\in\Ik ,t\neq s}a_{jt}^0\\
    =&\frac{1}{(\nk-1)\nk}\sum_{i,s\in\Ik }a_{is}^0\sum_{j\in \Ik ,j\neq i}-a_{js}^0\\
    =&\frac{1}{(\nk-1)\nk}\sum_{i,s\in\Ik }(a_{is}^0)^2\\
    =&\frac{1}{(\nk-1)\nk}\sumijk(a_{ij}-\bar a_{[k]  i\cdot}-\bar a_{[k]  \cdot j}+\bar a_{[k]  \cdot\cdot})^2.
\end{align*}
Therefore, \eqref{eq:var W_A} simplifies to:
$$ \var\left(\sumk \sumik a_{i\pi(i)}\right)
    =\sumk (\nk-1)^{-1}\sumijk (a_{ij}-\bar a_{[k]  i\cdot}-\bar a_{[k]  \cdot j}+\bar a_{[k]  \cdot\cdot})^2.$$
  By simple algebra, we can verify that 
    $$\sumijk (a_{ij}-\bar a_{[k]  i\cdot}-\bar a_{[k]  \cdot j}+\bar a_{[k]  \cdot\cdot})^2=\sumijk(a_{ij}^2-\bar a_{[k]  i\cdot}^2-\bar a_{[k]  \cdot j}^2+\bar a_{[k]  \cdot\cdot}^2).$$
    Then we finish the proof.
\end{proof}

\subsection{Comment on Proposition \ref{prop:smallstrata}}
 Let $B_n=\max_{1\le k\le K}\max_{i,j\in\Ik}|a_{ij}^\textup{s}|$. We have 
\begin{align*}
   \sumk \E[|W_{A^\textup{s},\pi[k]}|^3]\le \sumk \nk B_n\E[|W_{A^\textup{s},\pi[k]}|^2]
= \sumk \nk B_n R_{A[k]}^2
\le 2\sumk  B_n\sumijk (a_{ij}^\textup{s})^2,
\end{align*}
where the last inequality follows from $R_{A[k]}^2=(\nk-1)^{-1} \sumijk (a_{ij}^{\textup{s}})^2.$

From the power mean inequality,
$$\left(\frac 1{\nk^2}\sumijk (a_{ij}^\textup{s})^2\right)^{1/2}\le \left(\frac 1{\nk ^2}\sumijk |a_{ij}^\textup{s}|^3\right)^{1/3},$$
we can obtain
$$\sumijk (a_{ij}^\textup{s})^2\le \nk^{2/3}\left(\sumijk |a_{ij}^\textup{s}|^3\right)^{2/3}.$$
Therefore, we have 
    $$\sup_{t\in \mathbb R}\Big| \pr(W_{ A^\textup{s},\pi}\le t)-\Phi (t)\Big| \le 2 C  B_n \sumk \nk^{2/3}\left(\sumijk |a_{ij}^\textup{s}|^3\right)^{2/3}.$$ 
    
  The above upper bound is similar to the results of \cite{fang2015rates} and \citet[Theorem S2]{shi2022berry}, using the upper bound of all elements in the matrix $A^\textup{s}$.

\subsection{Proof of Corollary~\ref{coro:order smallstrata}}
\begin{proof}
   We have
    \begin{align*}
        \sumk \nk \sumijk |a_{ij}^\textup{s}|^3=\frac{\sumk \nk \sumijk |a_{ij}^\textup{0}|^3}{\sigma_A^3}
    = \frac{K^{-1}\sumk \nk \sumijk |a_{ij}^\textup{0}|^3}{K^{1/2} (K^{-1}\sumk \E[|W_{A^0,\pi[k]}|^2])^{3/2}}.
    \end{align*}
    Because $n_h$'s are uniformly bounded and \eqref{eq:moment g}, we have
    $$\E\left[\nk \sumijk |a_{ij}^\textup{0}|^3\right]=\sum_{h=1}^\infty w_h n_h \left\{\sum_{1\le i,j\le n_j}|g_{h,ij}^0|^3\right\}\asymp \sum_{h=1}^\infty w_h n_h^{-2} \left\{\sum_{1\le i,j\le n_j}|g_{h,ij}^0|^3\right\}.$$ 
   The law of large numbers implies
   \begin{align}\label{eq:order many small strata 1}
       K^{-1}\sumk \nk \sumijk |a_{ij}^\textup{0}|^3 \xrightarrow{a.s.} \sum_{h=1}^\infty w_{h} n_{h}\sum_{1\le i,j\le n_h} |g^0_{h,ij}|^3\asymp\sum_{h=1}^\infty w_{h} n_{h}^{-2}\sum_{1\le i,j\le n_h} |g^0_{h,ij}|^3.
   \end{align}
    By $\nk \sumijk |a_{ij}^\textup{0}|^3\ge \E [|W_{A^0,\pi[k]}|^3]\ge (\E[|W_{A^0,\pi[k]}|^2])^{3/2}$, we can use the law of large numbers to obtain
\begin{align}\label{eq:order many small strata 2}
        \frac 1 K\sumk \E [|W_{A^0,\pi[k]}|^2]\xrightarrow{a.s.} \sum_{h=1}^\infty w_{h} (n_{h}-1)^{-1}\sum_{1\le i,j\le n_h} (g^0_{h,ij})^2\asymp \sum_{h=1}^\infty w_{h} n_{h}^{-2}\sum_{1\le i,j\le n_h} (g^0_{h,ij})^2 .
    \end{align}
    The conclusion holds directly from \eqref{eq:order many small strata 1} and \eqref{eq:order many small strata 2}.
\end{proof}

\subsection{Proof of Proposition~\ref{prop:fewstrata}}
We review the following properties of the Wasserstein distance.
\begin{lemma}[Properties of Wasserstein distance]\label{lemma:Wasserstein}~~

    \textbf{(i) (relation with Kolmogorov distance, \citet[Proposition 1.2, part 2]{ross2011fundamentals})}
    If the random variable $Z$ has Lebesgue density bounded by $C_1$, then for any random variable $W$, we have 
    $$d_{\textup K}(W,Z)\le \sqrt{2C_1 d_{\textup W}(W,Z)}.$$
    
    \textbf{(ii) (sub-additivity and linearity, \cite{panaretos2019statistical})}
For independent $\{X_k\}_{k=1}^K$ and independent $\{Y_k\}_{k=1}^K$, we have
    $$d_{\textup W}\left(\sumk X_k,\sumk Y_k\right)\le \sumk d_{\textup W}( X_k, Y_k).$$
    For any real number $a$, we have 
$$d_{\textup W}(aX_1,aY_1)=|a|d_{\textup W}(X_1,Y_1).$$
    
   \textbf{(iii) ($L^1$ BEB \citet[Theorem 4.8]{chen2011normal})} For $\pi\sim \textup{Uniform}(\Pi_K)$, there exists a universal $C_2$ such that for all $k=1,\ldots,K$, we have
$$d_{\textup W}(W_{A^\textup{s},\pi[k]}/R_{A[k]},\mathcal N(0,1))\le C_2\frac{\sumijk |a_{ij}^\textup{s}|^3}{\nk R_{A[k]}^3}.$$
\end{lemma}

\begin{proof} [Proof of Proposition \ref{prop:fewstrata}]
Assume independent and identically distributed random variables $\{N_k\}_{k=1}^K$ following standard normal distributions.
By definition and Lemma \ref{lemma:Wasserstein}(i), there exists a universal constant $C_1$, such that 
\begin{align}\label{proof:Wass-K-BEB P1}
    \sup_{t\in \mathbb R}|\pr (W_{A^\textup{s},\pi}\le t)-\Phi(t)|=d_{\textup K}( W_{A^\textup{s},\pi}, N_1)
     \le  \left\{2C_1 d_{\textup W}( W_{A^\textup{s},\pi}, N_1)\right\}^{1/2}.
\end{align}
By Lemma \ref{lemma:Wasserstein}(ii),
     \begin{align}\label{proof:Wass-K-BEB P2}
       \left\{2C_1 d_{\textup W}( W_{A^\textup{s},\pi}, N_1)\right\}^{1/2}=&\left(2C_1 d_{\textup W}\left(\sumk  W_{A^\textup{s},\pi[k]}, \sumk R_{A[k]}N_k\right)\right)^{1/2}\notag\\
     \le & \left(2C_1 \sumk d_{\textup W}(  W_{A^\textup{s},\pi[k]}, R_{A[k]}N_k)\right)^{1/2}.
     \end{align}

     Applying Lemma \ref{lemma:Wasserstein}(iii), we have 
 \begin{align}\label{proof:Wass-K-BEB P3}
 \left(2C_1 \sumk R_{A[k]} d_{\textup W}(  W_{A^\textup{s},\pi[k]}/R_{A[k]}, N_k)\right)^{1/2}
     \le & \left(2C_1 \sumk R_{A[k]} C_2 \frac{\sumijk |a_{ij}^\textup{s}|^3}{\nk R_{A[k]}^3}\right)^{1/2}\notag\\
     = & C\left( \sumk  \frac{\sumijk |a_{ij}^\textup{s}|^3}{\nk R_{A[k]}^2}\right)^{1/2},
 \end{align}
 where $C=(2C_1 C_2)^{1/2}$.
Combining \eqref{proof:Wass-K-BEB P1}, \eqref{proof:Wass-K-BEB P2} and \eqref{proof:Wass-K-BEB P3}, we have
 $$\sup_{t\in \mathbb R}\Big| \pr(W_{ A^\textup{s},\pi}\le t)-\Phi (t)\Big| \le C\left\{\sumk \sumijk |a_{ij}^\textup{s}|^3/(\nk R_{A[k]}^2)\right\}^{1/2}.$$
\end{proof}

\subsection{Proof of Corollary~\ref{coro:order fewstrata}}
\begin{proof}
From the condition $R_{A[k]}^2\in [\underline c/K,\overline c/K]$, we have
$$\frac K {\overline c}  \sumk \sumijk|a_{ij}^\textup{s}|^3/\nk \le \sumk \sumijk|a_{ij}^\textup{s}|^3/(\nk R_{A[k]}^2)\le \frac K {\underline c}  \sumk \sumijk|a_{ij}^\textup{s}|^3/\nk.$$
Therefore, we only need to consider 
\begin{align}\label{proof-coro eq}
    \sumk \sumijk|a_{ij}^\textup{s}|^3/\nk=\frac{\frac{1}{n} \sumk \nk \sumijk |a_{ij}^0|^3/\nk^2}{ n^{1/2} (\frac{1}{n} \sumk  \sumijk |a_{ij}^0|^2/(\nk-1))^{3/2}}.
\end{align}
\subsubsection{The numerator of \eqref{proof-coro eq}}
Let $X_{ij}^{(3)}=|a_{ij}^0|^3-\E[|Y|^3]$. We have $\E [X_{ij}^{(3)}]=0$. By the relationship between the centered moment and the original moment, we have $\E[|X_{ij}^{(3)}|^{2+\varepsilon'}]\le 2^{2+\varepsilon'} \E[|Y|^{3(2+\varepsilon')}]<\infty$ for $\varepsilon'=\varepsilon/3.$

For the numerator of \eqref{proof-coro eq}, by Lemma \ref{lemma:converge to zero with p lar 2}, we have
\begin{align}\label{equ:a^3}
        \frac{1}{n} \sumk \nk \sumijk |a_{ij}^0|^3/\nk^2=\E[ |Y|^3]+\frac 1n \sumk\sumijk \frac {X_{ij}^{(3)}} {\nk} \xrightarrow{a.s.}\E[ |Y|^3].
    \end{align}

\subsubsection{The denominator of \eqref{proof-coro eq}}
Let $X_{ij}^{(2)}=(\nk-1)^{-1}\nk(|a_{ij}^0|^2-\E[|Y|^2])$, we have $\E[X_{ij}^{(2)}]=0$ and similarly $\E[|X_{ij}^{(2)}|^{3}]\le 2^3\E[|Y|^6]<\infty.$

    For the denominator of \eqref{proof-coro eq}, 
    $$\frac{1}{n} \sumk  \sumijk |a_{ij}^0|^2/(\nk-1)=\left(\sumk \frac {\nk^2} {n(\nk-1)}\right) \E [|Y|^2]+\frac{1}{n}\sumk\sumijk \frac{X_{ij}^{(2)}}{\nk}.$$
    By Lemma \ref{lemma:converge to zero with p lar 2}, the second term tends to zero almost surely.
 For all $\nk\ge 2,$ because
$$1\le \sumk \frac {\nk^2} {n(\nk-1)} \le 2,$$
we have 
$$\sumk \sumijk|a_{ij}^\textup{s}|^3/\nk \asymp  \frac{\E [|Y|^3]}{n^{1/2} (\E [|Y|^2])^{3/2}}\quad \text{a.s.}.$$
Therefore, we obtain the result.
    
\end{proof}

\subsection{Comment on the Combination of Propositions \ref{prop:smallstrata} and \ref{prop:fewstrata}}\label{subsec:combination of proposition 23}
In this section, we combine Propositions \ref{prop:smallstrata} and \ref{prop:fewstrata} to obtain a $L^1$ BEB using the sub-additivity of the Wasserstein distance. We define strata as ``small'' when their size is below a threshold sequence $c_{n}$, and ``large'' otherwise: 
$$
\mathcal{S}\left(c_{n}\right)=\left\{k: n_{[k]}<c_{n}\right\}, \quad \mathcal{L}\left(c_{n}\right)=\left\{k: n_{[k]} \geq c_{n}\right\}.
$$
Define $W_{S}=\sum_{k\in \mathcal S(c_n)} a_{i,\pi(i)}^\textup s$ with $R_{S}^2=\sum_{k\in \mathcal S(c_n)} R_{A[k]}^2$ and $W_{L}=\sum_{k\in \mathcal L(c_n)} a_{i,\pi(i)}^\textup s$ with $R_{L}^2=\sum_{k\in \mathcal L(c_n)} R_{A[k]}^2$. For IID standard normal random variables $N,N_1,N_2$, we have
\begin{align}\label{eq:d W}
    d_{\textup{W}}(W_{A^\textup{s}},N)=&d_{\textup{W}}(W_S+W_L,R_S N_1+R_L N_2)\notag\\
    \le& R_S d_{\textup{W}}(W_S/R_S,N_1)+R_L d_{\textup{W}}(W_L/R_L,N_2).
\end{align}
By \citet[Corollary 4.2]{chen2011normal} and similar to Proposition \ref{prop:smallstrata}, we can obtain 
\begin{align}\label{eq:d W small}
    d_{\textup{W}}(W_S/R_S,N_2)\le C\sum_{k\in \mathcal S(c_n)}\nk \sumijk |a_{ij}^\textup s|^3/R_S^3.
\end{align}

By the proof of Proposition \ref{prop:fewstrata}, we have 
\begin{align}\label{eq:d W large}
d_{\textup{W}}(W_L/R_L,N_1)\le C\sum_{k\in \mathcal L(c_n)} \sumijk |a_{ij}^\textup s/R_L|^3/\{\nk (R_{A[k]}/R_L)^2\}.
\end{align}
Combining \eqref{eq:d W}, \eqref{eq:d W small} and \eqref{eq:d W large}, we have
\begin{align*}
    d_{\textup{W}}(W_{A^\textup{s}},N)\le& C\sum_{k\in \mathcal S(c_n)}\nk \sumijk |a_{ij}^\textup s|^3/R_S^2 + C\sum_{k\in \mathcal L(c_n)} \sumijk |a_{ij}^\textup s|^3/(\nk R_{A[k]}^2)\\
    \le & C \frac{c_n^2}{R_S^2}\sum_{k\in \mathcal S(c_n)} \sumijk |a_{ij}^\textup s|^3/\nk+ C\sum_{k\in \mathcal L(c_n)} \sumijk |a_{ij}^\textup s|^3/(\nk R_{A[k]}^2).
\end{align*}
To recover a Wasserstein-type bound analogous to Theorem \ref{thm:Wasserstein BEB}, one must assume $R_{A[k]}^2,R_S^2\ge C_1$ and $c_n\le C_2$ for some constants $C_1, C_2>0$. Thus, a direct combination of Propositions \ref{prop:smallstrata} and \ref{prop:fewstrata} only works under such constraints, whereas our approach establishes a uniform bound across all regimes without additional assumptions.

Furthermore, because the sub-additivity does not hold for the Kolmogorov distance, the resulting BEB is suboptimal based on the direct combination of Propositions \ref{prop:smallstrata} and \ref{prop:fewstrata}.

\subsection{Proof of Theorem \ref{theorem:beb}}\label{proof:first bound beb}
\subsubsection{Simplification to Theorem \ref{theorem:beb mn1} and Simplification of Variance Structure}\label{subsub:simplif epsi0}
We can assume $\min_{1\le k\le K}\nk\ge 2$ without changing the statement of Theorem \ref{theorem:beb} because $W_{A,\pi[k]}$ is constant when $\nk\in \{0,1\}$.
Recall that for a matrix $A=[a_{ij}]$, let $A'=[a_{ij}']$, where $a_{ij}'=a_{ij}^\textup{s}\cdot \bold{1}(|a_{ij}^\textup{s}|\le 1/2)$. Recall the definition of $\beta_{A[k]}$, $M_n^\textup{s}$ and $M_n^1$ in Section \ref{sec:proof sketch}.
For any given constant $\varepsilon_0>0$, if $\sum_{k=1}^K{\beta_{A[k]  }}/{\nk}\ge\varepsilon_0$, then 
\begin{align}\label{eq:thm1,first:1}
    \sup_{t\in \mathbb R}|\pr(W_{A^\textup{s},\pi}\le t)-\Phi(t)|\le 2\le \frac{2}{\sqrt{\varepsilon_0}}\left( \sum_{k=1}^K\frac{\beta_{A[k]  }}{\nk}\right)^{1/2}.
\end{align}
Therefore, we focus on the case with $\sum_{k=1}^K {\beta_{A[k]  }}/{\nk}\le \varepsilon_0$.

From Lemma \ref{lemma:scaling}, we have
\begin{align}\label{eq:thm1,first:2}
\sup_{t\in\mathbb R}|\pr(W_{A^\textup{s},\pi}\le t)-\Phi(t)|\le \sup_{t\in \mathbb R}|\pr(W_{(A')^\textup{s},\pi}\le t)-\Phi(t)|+c_3\sumk \frac{\beta_{A[k]  }}{\nk},
\end{align}
and by Lemma \ref{lemma:truncation}, we have
\begin{align}\label{eq:beta A'[k] le beta A[k]}
    \frac{\beta_{(A')^\textup{s}[k]  }}{\nk}=\frac{\beta_{A'[k]  }}{\nk}\le c_2\frac{\beta_{A[k]  }}{\nk},
\end{align}
and also that $(A')^\textup{s}\in M_n^1$. Hence, it suffices to prove that there exists a constant $c'$ such that
$$\sup_{t\in \mathbb R}|\pr(W_{(A')^\textup{s},\pi}\le t)-\Phi(t)|\le c' \sumk \frac{\beta_{(A')^\textup{s}[k]  }}{\nk}.$$

\subsubsection{Zero-Bias Transformation}
Recall $h_{t,\alpha}$ in Section \ref{sec:proof sketch} and
\begin{align}\label{eq:first main theorem}
    \E [h_{t-\alpha,\alpha}(W_{(A')^\textup{s},\pi} )]-\Phi(t) \le \pr(W_{(A')^\textup{s},\pi} \le t)-\Phi(t)\le \E [h_{t,\alpha}(W_{(A')^\textup{s},\pi} )]-\Phi(t).
\end{align}
For $N\sim\mathcal N(0,1)$, because 
$$|\E [h_{t-\alpha,\alpha}(N )]-\Phi(t)|=\E[ h_{t,0}(N)-h_{t-\alpha,\alpha}(N)]\le \pr(N\in [t-\alpha,t])\le \frac{\alpha}{\sqrt{2\pi}},$$
the left-hand side of \eqref{eq:first main theorem} is bounded by
\begin{align}\label{eq: left h t-alph}
   |\E [h_{t-\alpha,\alpha}(W_{(A')^\textup{s},\pi} )]-\Phi(t)|=  &|\E [h_{t-\alpha,\alpha}(W_{(A')^\textup{s},\pi} )]-\E [h_{t-\alpha,\alpha}(N )]+\E [h_{t-\alpha,\alpha}(N )]-\Phi(t)|\notag\\
    \le & |\E [h_{t-\alpha,\alpha}(W_{(A')^\textup{s},\pi} )]-\E [h_{t-\alpha,\alpha}(N )]|+\frac{\alpha}{\sqrt{2\pi}}.
\end{align}
Because 
$$|\E [h_{t,\alpha}(N )]-\Phi(t)|=\E[ h_{t,\alpha}(N)-h_{t,0}(N)]\le \pr(N\in [t,t+\alpha])\le \frac{\alpha}{\sqrt{2\pi}},$$
the right-hand side of \eqref{eq:first main theorem} is bounded by
\begin{align}\label{eq: right h t alph}
    |\E [h_{t,\alpha}(W_{(A')^\textup{s},\pi} )]-\Phi(t)|
    =& |\E [h_{t,\alpha}(W_{(A')^\textup{s},\pi} )]-\E [h_{t,\alpha}(N )]+\E [h_{t,\alpha}(N )]-\Phi(t)|\notag\\
    \le & |\E [h_{t,\alpha}(W_{(A')^\textup{s},\pi} )]-\E [h_{t,\alpha}(N )]|+\frac{\alpha}{\sqrt{2\pi}}.
\end{align}
Combine \eqref{eq: left h t-alph} and \eqref{eq: right h t alph} to obtain
$$\sup_{t\in \mathbb R}|\pr(W_{(A')^\textup{s},\pi}\le t)-\Phi(t)|\le \sup_{t\in \mathbb R}|\E [h_{t,\alpha}(W_{(A')^\textup{s},\pi})]-\E [h_{t,\alpha}(N)]|+\frac{\alpha}{\sqrt{2\pi}}.$$

Recall the definition of $f_{t,\alpha}$ in Section \ref{sec:proof sketch}. From Lemma \ref{thm:zerobias}, we have
\begin{align*}
    \sup_{t\in \mathbb R}|\pr(W_{(A')^\textup{s},\pi}\le t)-\Phi(t)|\le& \sup_{t\in \mathbb R}|\E [f'_{t,\alpha}(W_{(A')^\textup{s},\pi})]-\E [W_{(A')^\textup{s},\pi} f_{t,\alpha}(W_{(A')^\textup{s},\pi})]|+\frac{\alpha}{\sqrt{2\pi}}\\
   \overset{\eqref{eq:zb form}}{=}& \sup_{t\in \mathbb R}|\E [f'_{t,\alpha}(W_{(A')^\textup{s},\pi})]-\E [f'_{t,\alpha}(W_{(A')^\textup{s},\pi}^*)]|+\frac{\alpha}{\sqrt{2\pi}}.
\end{align*}

Recall \eqref{eq:total decomposition}, we only need to bound $\mathcal A_1,\mathcal A_2$ and $\mathcal A_3$.

\subsubsection{Bounds for $\mathcal A_1$ and $\mathcal A_2$}\label{subsubsec:bound A_1-3}

From Proposition \ref{prop:L1}, we bound $\mathcal A_1$ as
$$\mathcal A_1= \E[|W_{(A')^\textup{s},\pi}^*-W_{(A')^\textup{s},\pi}|]\le 80\sumk {\beta_{A'[k]  }}/{\nk}\le 80 c_2 \sumk {\beta_{A[k]  }}/{\nk}.$$

Now, we bound $\mathcal A_2$.
Recall $\mathcal R, S, T$ defined in \eqref{eq:define R S T}. Define 
$$T^\dagger=\sum_{i\in \mathcal R} (a')^\textup{s}_{i\pi^\dagger(i)},\quad T^\ddagger=\sum_{i\in \mathcal R} (a')^\textup{s}_{i\pi^\ddagger(i)}.$$
 From Lemma \ref{thm:zerobias} and Proposition \ref{prop:permutation property}, we have that
\begin{align}\label{eq:diff w* w}
    W_{(A')^\textup{s},\pi}^*-W_{(A')^\textup{s},\pi}=&(UW_{(A')^\textup{s},\pi^\dagger}+(1-U)W_{(A')^\textup{s},\pi^\ddagger})-W_{(A')^\textup{s},\pi}\notag\\
    =& (U(S+T^\dagger)+(1-U)(S+T^\ddagger))-(S+T)\notag\\
    =&UT^\dagger+(1-U)T^\ddagger-T.
\end{align}
Recall $\textbf{I}=(B^\dagger,I^\dagger,J^\dagger,\pi^{-1}(P^\dagger),\pi^{-1}(Q^\dagger),\pi(I^\dagger),\pi(J^\dagger),P^\dagger,Q^\dagger)$. By \eqref{eq:diff w* w}, $W_{(A')^\textup{s},\pi}^*-W_{(A')^\textup{s},\pi}$ is measurable with respect to $\mathcal J=\{\textbf{I},U\}$. Furthermore, because of Lemma \ref{lemma:truncation}, we have $(A')^\textup{s}\in M_n^1$ and $|\mathcal R|\le 4$, and further
$$|W_{(A')^\textup{s},\pi}|=|S+T|\le |S|+\sum_{i\in \mathcal R}|(a')^\textup{s}_{i\pi(i)}|\le |S|+4.$$ Therefore, we have that
\begin{align}\label{eq:A2 bound by 4+var S L1}
	\mathcal A_2=&\E [|W_{(A')^\textup{s},\pi}(W_{(A')^\textup{s},\pi}^*-W_{(A')^\textup{s},\pi})|]\notag\\
=&\E [|W_{(A')^\textup{s},\pi}^*-W_{(A')^\textup{s},\pi} |E(|W_{(A')^\textup{s},\pi}|\mid \mathcal J)]\notag\\
\le & \E [|W_{(A')^\textup{s},\pi}^*-W_{(A')^\textup{s},\pi}|\E (4+|S|\mid \mathbf{I})]\notag\\
\le & \E \left[|W_{(A')^\textup{s},\pi}^*-W_{(A')^\textup{s},\pi}| (4+\sqrt{\E [S^2\mid \mathbf{I}]})\right].
\end{align}
Let $l=|\mathcal R|\in \{2,3,4\}$. Since $S=\sum_{i\not\in \mathcal R}(a')^\textup{s}_{i\pi(i)}$ and given the definition of $\pi$, we have that $S\mid \mathbf{I}$ and $W_{D,\phi}$ have the same distribution,
where $D$ is the $(n-l)\times (n-l)$ array formed by removing from $(A')^\textup{s}$ the rows $\mathcal R=\{I^\dagger,J^\dagger,\pi^{-1}(P^\dagger ) ,\pi^{-1}(Q^\dagger )\}$ and columns $\pi(\mathcal R)=\{\pi(I^\dagger) ,\pi(J^\dagger),P^\dagger,Q^\dagger\}$, and $\phi\sim \textup{Uniform}(\Pi_K)$ over the stratified structure $\{\Ik \setminus\mathcal R\}_{k=1}^K$.

Lemma \ref{lemma:inductive} yields $|\mu_D|\le 8$ and $\sigma_D^2\le c_4$, so that
$$\E [W_{D,\phi}^2]\le |\mu_D|^2+\sigma_D^2\le 64+c_4.$$
By \eqref{eq:A2 bound by 4+var S L1}, we have that
\begin{align}\label{eq:thm1,first:4}
    \mathcal A_2\le (\sqrt{64+c_4}+4) \E [|W_{(A')^\textup{s},\pi}^*-W_{(A')^\textup{s},\pi}|]\le (\sqrt{64+c_4}+4)80 c_2 {\sumk \frac{\beta_{A[k]  }}{\nk}}.
\end{align}

\subsubsection{Simplification of $\mathcal A_3$}\label{subsubsec:bound for A_3 new}
Recall the discussion in Section \ref{subsub:simplif epsi0} to let $\sumk \beta_{A[k]}/\nk\le \varepsilon_0$.
Recall $T^\dagger=\sum_{i\in \mathcal R} (a')^\textup{s}_{i\pi^\dagger(i)}$, $T^\ddagger=\sum_{i\in \mathcal R} (a')^\textup{s}_{i\pi^\ddagger(i)}$ and that $W_{(A')^\textup{s},\pi}^*-W_{(A')^\textup{s},\pi}$ is measurable with respect to $\mathcal J=\{\textbf{I},U\}$. Rewrite \eqref{eq:A3 keyquan} as follows:
\begin{align*}
    \mathcal A_3=&\frac{1}{\alpha} \E \left[|W_{(A')^\textup{s},\pi}^*-W_{(A')^\textup{s},\pi}|\cdot \E\left[\int_0^1 \bold{1}\{W_{(A')^\textup{s},\pi}+r(W_{(A')^\textup{s},\pi}^*-W_{(A')^\textup{s},\pi})\in [t,t+\alpha]\}\textup{d}r\mid \mathcal J \right] \right].
\end{align*}

For any $r\in \mathbb R$, we have 
\begin{align*}
W_{(A')^\textup{s},\pi}+r(W_{(A')^\textup{s},\pi}^*-W_{(A')^\textup{s},\pi})=& r W_{(A')^\textup{s},\pi}^*+(1-r) W_{(A')^\textup{s},\pi}\\
    =& r(S+UT^\dagger+(1-U)T^\ddagger)+(1-r)(S+T)\\
    =& S+g_r,
\end{align*}
where $g_r=rUT^\dagger+r(1-U)T^\ddagger+(1-r)T$.
Therefore, 
\begin{align}\label{eq:S in [t,t+alpha]}
    &\E \left[\int_0^1 \bold{1}\{W_{(A')^\textup{s},\pi}+r(W_{(A')^\textup{s},\pi}^*-W_{(A')^\textup{s},\pi})\in [t,t+\alpha]\}\textup{d}r\mid \mathcal J\right]\notag\\
    =& \int_0^1 \pr (S+g_r\in [t,t+\alpha]\mid \mathcal J)\textup{d}r\notag\\
    \le & \sup_{t\in\mathbb 
 R}\pr (S\in [t-g_r,t+\alpha-g_r]\mid \mathcal J)\notag\\
 =&\sup_{t\in \mathbb R}\pr(S\in[t,t+\alpha]\mid \textbf{I}).
\end{align}
Then \eqref{eq:S in [t,t+alpha]} implies
\begin{align*}
    \mathcal A_3 \le \frac{1}{\alpha} \E \left[|W_{(A')^\textup{s},\pi}^*-W_{(A')^\textup{s},\pi}|  \sup_{t\in \mathbb R}\pr (S\in [t,t+\alpha]\mid \textbf{I}) \right].
\end{align*}

\subsubsection{Upper bound on $\pr(W_{D,\phi}\in [t,t+\alpha ])$}\label{subsubsec:concentration lemma}
The goal of this subsection is to upper bound $\pr(S\in [t,t+\alpha]\mid \textbf{I})$, which will be stated in \eqref{eq:unified concentration}.

Given $\textbf{I}$ and $B^\dagger=b^\dagger$. To study the conditional distribution of $S$ given $\textbf{I}$, recall $\mathcal R$ defined in \eqref{eq:define R S T} and $l=|\mathcal R|$. We define a permutation $\phi$ given $\textbf{I}$, such that in stratum $b^\dagger$, $\phi$ is a uniform random permutation among $\mathcal I_{[b^\dagger]}\setminus\mathcal R$, and in other strata $k\neq b^\dagger$, $\phi$ is a uniform random permutation among $\Ik$. 
For convenience, we use $\Rb=\{(i,j):i\in \mathcal I_{[b^\dagger]}\setminus\mathcal R,j\in \mathcal I_{[b^\dagger]}\setminus\pi(\mathcal R)\}$. The following results are derived conditional on $\textbf{I}$.

When $n_{[b^\dagger]}-l\in \{0,1\}$, $W_{D,\phi[b^\dagger]}$ is constant and essentially $\phi\in \textup{Uniform}(\Pi_{K-1})$. When $n_{[b^\dagger]}-l\ge 2$, we have $\phi\in \textup{Uniform}(\Pi_{K})$. 
To apply Lemma \ref{lemma:concentration}, we define a new random variable $B_D$: 
$$\pr(B_D=k)=\begin{cases}
 \frac{(\nk-1)\boldsymbol{1}(k\neq b^\dagger)}{n-n_{[b^\dagger]}-(K-1)},& n_{[b^\dagger]}-l\in\{0,1\};\\
    \frac{\nk-l\cdot \boldsymbol 1(k=b^\dagger)-1}{n-l-K},& n_{[b^\dagger]}-l\ge 2.
\end{cases}$$
Conditional on $B_D=k$, we uniformly select two different indices $I_D,J_D$ from $\Ik\setminus\mathcal R.$
Let $$W_{D,\phi}'=W_{D,\phi}-(a')^\textup{s}_{I_D \phi(I_D)}-(a')^\textup{s}_{J_D \phi(J_D)}+(a')^\textup{s}_{I_D \phi(J_D)}+(a')^\textup{s}_{J_D \phi(I_D)}.$$

By Proposition \ref{prop:stein pair}, we have 
    $$\E[W_{D,\phi}'-W_{D,\phi}\mid W_{D,\phi}]=-\lambda W_{D,\phi}+L,$$
    where $\lambda=2/(n-l-K)$ and 
\begin{align*}
   L=\begin{cases}0,&n_{[b^\dagger]}-l\in\{0,1\};\\
       \frac{2}{(n_{[b^\dagger]}-l)(n-l-K)}\sum_{i\in \mathcal R,j\in \pi(\mathcal R)}(a')^\textup{s}_{ij},&n_{[b^\dagger]}-l\ge 2.
   \end{cases}
\end{align*}
Here we use the following fact to simplify $L$:
\begin{align}\label{eq:sum Db=sum r pi(r)}
    \sumRb (a')^\textup{s}_{ij}=\sumRb (a')^\textup{s}_{ij}+\sum_{i\in \mathcal R, j\in \mathcal I_{[b^\dagger]}}(a')^\textup{s}_{ij}+\sum_{i\in \mathcal I_{[b^\dagger]}, j\in \pi(\mathcal R)}(a')^\textup{s}_{ij}=\sum_{i\in \mathcal R,j\in \pi(\mathcal R)}(a')^\textup{s}_{ij}.
\end{align}

Given $\textbf{I}$, $L$ is a constant. By Lemma~\ref{lemma:concentration} and following the proof of \citet[Proposition~2.4]{chen2015error}, we proceed in four steps:

1. Bound 
$$\delta =\frac{\E[|W_{D,\phi}'-W_{D,\phi}|^3]}{\lambda},$$
and  $$B_0=\sqrt{\var\left(\E\left[\frac{1}{2\lambda}(W_{D,\phi}'-W_{D,\phi})^2 \bold{1}(|W_{D,\phi}'-W_{D,\phi}|\le \delta)\mid W_{D,\phi}\right]\right)}.$$

2. Bound $|L|/\lambda$.

3. Give a lower bound and an upper bound on $\E [W_{D,\phi}^2]$. 

4. Apply Lemma \ref{lemma:concentration}.

We consider the two cases $n_{[b^\dagger]}-l\in\{0,1\}$ and $n_{[b^\dagger]}-l\ge 2$ separately.
In each case, the above four steps are verified.

We next present a useful fact that will be repeatedly used in the proof.
By choosing a small constant $\varepsilon_0 \le 1/(7200^{3/2}c_2)$, we use \eqref{eq:beta A'[k] le beta A[k]} to obtain that for each $k=1,\ldots,K$,
\begin{align*}
 \frac{R_{A^{\prime}[k]}^{2}}{n_{[k]}}=\frac{1}{n_{[k]}\left(n_{[k]}-1\right)} \sum_{i, j \in \mathcal{I}_{[k]}}|(a')^\textup{s}_{ij}|^{2} \leq \frac{n_{[k]}^{1 / 3}}{n_{[k]}-1}\left(\frac{1}{n_{[k]}} \sum_{i, j \in \mathcal{I}_{[k]}}|(a')^\textup{s}_{ij}|^{3}\right)^{2 / 3}\le 2(c_2\varepsilon_0)^{2/3}. 
\end{align*}
Since $\varepsilon_0\le 1/(7200^{3/2}c_2)$, we have 
$$\max_{1 \le k\le K} R_{A'[k]}^2/\nk \le 1/3600.$$

\subsubsection{When $n_{[b^\dagger]}-l\in \{0, 1\}$}
When $n_{[b^\dagger]}-l=1$, $W_{D,\phi[b^\dagger]}$ is constant. By $\pr(W_{D,\phi}\in [t,t+\alpha])=\pr(\sum_{k\neq b^\dagger}W_{D,\phi[k]}\in [t-W_{D,\phi[b^\dagger]},t+\alpha-W_{D,\phi[b^\dagger]}])$, we can only consider the case with $n_{[b^\dagger]}-l=0$. 

\paragraph{Step 1. Bound $\delta$ and $B_0$}
Using Theorem \ref{thm:bound for delta and B02}(i), we have
$$\delta\le 32 \sum_{k\neq b^\dagger} \sumijk |(a')^\textup{s}_{ij}|^3/\nk \le 32 \sumk \sumijk |(a')^\textup{s}_{ij}|^3/\nk. $$
Using Theorem \ref{thm:bound for delta and B02}(ii), we have
$$B_0\le \sqrt{32\times 56} \sum_{k\neq b^\dagger}\sumijk |(a')^\textup{s}_{ij}|^3/\nk \le 16\sqrt{7}  \sumk \sumijk |(a')^\textup{s}_{ij}|^3/\nk. $$
\paragraph{Step 2. Bound $|L| / \lambda$}
We have $|L|=0$.
\paragraph{Step 3. Give a lower bound and an upper bound on $\E [W_{D,\phi}^2]$}
Since $\E [W_{D,\phi}]=0,$ we have 
\begin{align*}
    \E [W_{D,\phi}^2]=\var(W_{D,\phi})=\sum_{k\neq b^\dagger} R_{A'[k]}^2 =1-R_{A'[b^\dagger]}^2\in[9/10,1],
\end{align*} 
where the last equality holds from $n_{[b^\dagger]}=l \le 4$ and further $R_{A'[b^\dagger]}^2=n_{[b^\dagger]} (R_{A'[b^\dagger]}^2/n_{[b^\dagger]})\le 1/10.$

\paragraph{Step 4. Apply Lemma \ref{lemma:concentration}}
First, for the numerator of the right-hand side of Lemma \ref{lemma:concentration}, we have

\begin{align}\label{eq:thm1 second bounf main numerator nk-l=0}
    &(\E [|W_{D,\phi}|]+\E [|L|]/\lambda)(\frac{\alpha}{2}+\delta)+\sqrt{\var(\E [\frac{1}{2\lambda}(W_{D,\phi}'-W_{D,\phi})^2\bold {1}(|W_{D,\phi}'-W_{D,\phi}|\le \delta\mid W_{D,\phi})])}\notag\\
     \le &(\sqrt{\E [W_{D,\phi}^2]}+ |L|/\lambda)(\frac{\alpha}{2}+\delta)+B_0.
    \end{align}
    
    By $L=0$ under $n_{[b^\dagger]}-l=0$, the right-hand side  of \eqref{eq:thm1 second bounf main numerator nk-l=0} is upper bounded by
    \begin{align}\label{eq:numerator nk=l}
        (\sqrt{\E [W_{D,\phi}^2]}+ |L|/\lambda)(\frac{\alpha}{2}+\delta)+B_0\le \frac{\alpha}{2}+(32+16\sqrt{7})\sumk\sumijk|(a')_{ij}^\textup{s}|^3/\nk.
    \end{align}

Second, for the denominator of the right-hand side of Lemma \ref{lemma:concentration}, by \eqref{eq:bound for L/lambda} and \eqref{eq:upper bounf for W_D,phi}, we have
    \begin{align}\label{eq:demumer nk=l}
        \E [W_{D,\phi}^2]- \E[ |W_{D,\phi} L|]/\lambda -1/2
        \ge \frac{9}{10}-\frac 1 2=\frac 2 5.
    \end{align}

The inequalities \eqref{eq:numerator nk=l} and \eqref{eq:demumer nk=l} imply 
\begin{align}\label{eq:concentration nk=l}
    \pr(W_{D,\phi}\in [t,t+\alpha])\le \frac{5}{4}\alpha +40(2+\sqrt 7)\sumk\sumijk|(a')_{ij}^\textup{s}|^3/\nk.
\end{align}
\subsubsection{When $n_{[b^\dagger]}-l\ge 2$}
\paragraph{Step 1. Bound $\delta$ and $B_0$}
Using Theorem \ref{thm:bound for delta and B02}(i) and $n_{[b^\dagger]}/(n_{[b^\dagger]}-l)\le 3$, we have 
\begin{align}
   \delta\le& 32 \sum_{k\neq b^\dagger}\sumijk |(a')^\textup{s}_{ij}|^3/\nk +32\sumRb  |(a')^\textup{s}_{ij}|^3/(\nk-l) \notag \\
   \le &32 \frac{\nk}{\nk-l}\sumk \sumijk |(a')^\textup{s}_{ij}|^3/\nk \notag \\
   \le  &96 \sumk \sumijk |(a')^\textup{s}_{ij}|^3/\nk
\end{align}

Using Theorem \ref{thm:bound for delta and B02}(ii), similarly, we have
$$B_0\le 3\times \sqrt{32\times 56}  \sumk\sumijk   |(a')^\textup{s}_{ij}|^3/\nk=48\sqrt{7}\sumk\sumijk   |(a')^\textup{s}_{ij}|^3/\nk . $$
\paragraph{Step 2. Bound $|L| / \lambda$} We have 
\begin{align*}
 |L|=&\frac{2}{(n_{[b^\dagger]}-l)(n-l-K)}|\sum_{i\in \mathcal R,j\in \pi(\mathcal R)}(a')^\textup{s}_{ij}|\\
    \le & \frac{2}{(n_{[b^\dagger]}-l)(n-l-K)}\sqrt{l^2 \sum_{i\in \mathcal R,j\in \pi(\mathcal R)}\{(a')^\textup{s}_{ij}\}^2}\\
    \le & \frac{2}{(n_{[b^\dagger]}-l)(n-l-K)}\sqrt{l^2 \sum_{i,j\in \mathcal I_{[b^\dagger]}}\{(a')^\textup{s}_{ij}\}^2}\\
 = & \frac{2}{(n_{[b^\dagger]}-l)(n-l-K)}\sqrt{l^2 (n_{[b^\dagger]}-1)R^2_{A'[b^\dagger]}}\\
 \le &\frac{2l \sqrt{n_{[b^\dagger]}}R_{A'[b^\dagger]}}{(n_{[b^\dagger]}-l)(n-l-K)}.
\end{align*}
Therefore, by $n_{[b^\dagger]}-l\ge 2$, $l\le 4$ and $R_{A'[b^\dagger]}^2/n_{[b^\dagger]}\le 1/3600$, we have
\begin{align}\label{eq:bound for L/lambda}
    \frac{|L|}{\lambda}\le \frac{l \sqrt{n_{[b^\dagger]}}R_{A'[b^\dagger]}}{n_{[b^\dagger]}-l}=\frac{l n_{[b^\dagger]}}{n_{[b^\dagger]}-l}\sqrt{\frac{R_{A'[b^\dagger]}^2}{n_{[b^\dagger]}}}\le \frac 1 5.
\end{align}

\paragraph{Step 3. Give a lower bound and an upper bound on $\E [W_{D,\phi}^2]$}

We have $\E [W_{D,\phi}^2] = (\E [W_{D,\phi}])^2+\var(W_{D,\phi})$, where 
\begin{align}\label{eq:E W_D,phi}
    \E [W_{D,\phi}]=\E \left[\sum_{k\neq b^\dagger} W_{D,\phi[k]}\right]+ \E \left[\sum_{i\in \mathcal I_{[b^\dagger]}\setminus\mathcal R} (a')^\textup{s}_{i\phi(i)}\right]=\E \left[\sum_{i\in \mathcal I_{[b^\dagger]}\setminus\mathcal R} (a')^\textup{s}_{i\phi(i)}\right].
\end{align}
and
\begin{align}\label{eq:var(W_D phi)}
    \var (W_{D,\phi})=  \sum_{k\neq b^\dagger}\var (W_{D,\phi[k]})+\var \left\{\sum_{i\in \mathcal I_{[b^\dagger]}\setminus\mathcal R} (a')^\textup{s}_{i\phi(i)}\right\}.
\end{align}

For the first term of the right-hand side of \eqref{eq:var(W_D phi)}, we have 
$$\var( W_{D,\phi[k]})=\frac{1}{\nk-1} \sumijk \{(a')^\textup{s}_{ij}\}^2=R_{A'[k]}^2.$$

Therefore, by \eqref{eq:E W_D,phi} and \eqref{eq:var(W_D phi)}, we have
\begin{align}\label{eq: W_D,phi decom}
    \E [W_{D,\phi}^2]=&\sum_{k\neq b^\dagger} R_{A'[k]}^2+ \var \left\{\sum_{i\in \mathcal I_{[b^\dagger]}\setminus\mathcal R} (a')^\textup{s}_{i\phi(i)}\right\}+\left[\E \left[\sum_{i\in \mathcal I_{[b^\dagger]}\setminus\mathcal R} (a')^\textup{s}_{i\phi(i)}\right]\right]^2\notag\\
    =& \sum_{k\neq b^\dagger} R_{A'[k]}^2+\E \left[\left[\sum_{i\in \mathcal I_{[b^\dagger]}\setminus\mathcal R} (a')^\textup{s}_{i\phi(i)}\right]^2\right].
\end{align}

For the second term of the right-hand side of \eqref{eq:var(W_D phi)}, we review a lemma to obtain its upper bound.
\begin{lemma}[Equation (2.9) in  \citet{chen2015error}]\label{lemma:chen2015 eq2.9}
We have
    \begin{align}\label{eq:chen2015 eq2.9}
  & \E\left[  \left[\sum_{i\in \mathcal I_{[b^\dagger]}\setminus\mathcal R} (a')^\textup{s}_{i\phi(i)}\right]^2\right] =  \frac{1}{n_{[b^\dagger]}-l-1}\sumRb \{(a')^\textup{s}_{ij}\}^2\notag\\
&~~~~~+\frac{1}{(n_{[b^\dagger]}-l)(n_{[b^\dagger]}-l-1)}\sumRb (a')^\textup{s}_{ij}\left\{\sum_{k \in \mathcal R,t\in \pi(\mathcal R)} (a')^\textup{s}_{kt}+\sum_{k\in \mathcal \pi(\mathcal R)} (a')^\textup{s}_{ik}+\sum_{t\in \mathcal R} (a')^\textup{s}_{tj}\right\}.
\end{align}
\end{lemma}

By Lemma \ref{lemma:chen2015 eq2.9}, we need to bound the four terms of  
the right-hand side of \eqref{eq:chen2015 eq2.9}, respectively. We have
\begin{align*}
    \left|\sumRb (a')^\textup{s}_{ij}\left\{\sum_{k\in \pi(\mathcal R)}(a')^\textup{s}_{ik}\right\}\right|=&\left|\sum_{i\in \mathcal I_{[b^\dagger]}\setminus\mathcal R}\left\{\sum_{k\in \pi(\mathcal R)}(a')^\textup{s}_{ik}\right\}^2\right|\\
    \le & \left|\sum_{i\in \mathcal I_{[b^\dagger]}\setminus\mathcal R} l \sum_{k\in \pi(\mathcal R)}\left\{(a')^\textup{s}_{ik}\right\}^2\right|\le l (n_{[b^\dagger]}-1)R^2_{A'[b^\dagger]}.
\end{align*}
We can similarly obtain
$$\left|\sumRb (a')^\textup{s}_{ij}\left\{\sum_{l \in \mathcal R} (a')^\textup{s}_{tj}\right\}\right|\le l(n_{[b^\dagger]}-1)R^2_{A'[b^\dagger]}.$$
Moreover, by \eqref{eq:sum Db=sum r pi(r)} and the Cauchy--Schwarz inequality, we have
\begin{align*}
  \left|\sumRb (a')^\textup{s}_{ij}\sum_{k \in \mathcal R,t\in \pi(\mathcal R)} (a')^\textup{s}_{kt}\right|=&\left|\left\{\sum_{k \in \mathcal R,t\in \pi(\mathcal R)} (a')^\textup{s}_{kt}\right\}^2\right|\\
  \le& l^2 \sum_{k \in \mathcal R,t\in \pi(\mathcal R)} \{(a')^\textup{s}_{kt}\}^2\le l^2(n_{[b^\dagger]}-1)R^2_{A'[b^\dagger]}.  
\end{align*}
It implies
\begin{align}\label{eq:24()R}
   & \left|\frac{1}{(n_{[b^\dagger]}-l)(n_{[b^\dagger]}-l-1)}\sumRb (a')^\textup{s}_{ij}\left\{\sum_{k \in \mathcal R,t\in \pi(\mathcal R)} (a')^\textup{s}_{kt}+\sum_{k\in \mathcal \pi(\mathcal R)} (a')^\textup{s}_{ik}+\sum_{t\in \mathcal R} (a')^\textup{s}_{tj}\right\}\right|\notag\\
    \le& \frac{ (l^2+l)(n_{[b^\dagger]}-1)R^2_{A'[b^\dagger]}}{(n_{[b^\dagger]}-l)(n_{[b^\dagger]}-l-1)}\notag\\
\le &\frac{24 (n_{[b^\dagger]}-1)R^2_{A'[b^\dagger]}}{(n_{[b^\dagger]}-l)(n_{[b^\dagger]}-l-1)}. 
\end{align}
Therefore, by \eqref{eq: W_D,phi decom}, \eqref{eq:chen2015 eq2.9} and \eqref{eq:24()R}, 
we have
\begin{align}\label{eq:bound 24 nR^2}
   \left |\E [W_{D,\phi}^2]-\left[\sum_{k\neq b^\dagger} R_{A'[k]}^2+\frac{1}{n_{[b^\dagger]}-l-1}\sumRb \{(a')^\textup{s}_{ij}\}^2\right]\right|\le \frac{24 (n_{[b^\dagger]}-1)R^2_{A'[b^\dagger]}}{(n_{[b^\dagger]}-l)(n_{[b^\dagger]}-l-1)}.
\end{align}

Since we assume that $\sumk \sumijk |(a')^\textup{s}_{ij}|^3/\nk\le \varepsilon_0$, we have
\begin{align}\label{eq:bound for sum iR or j piR a_ij^2}
    \sum_{i,j\in \mathcal I_{[b^\dagger]}: i\in \mathcal R\atop or~ j\in \pi(\mathcal R)}\{(a')^\textup{s}_{ij}\}^2\le (8n_{[b^\dagger]})^{1/3}\left\{\sumijk |(a')^\textup{s}_{ij}|^3\right\}^{2/3}\le 2n_{[b^\dagger]}\varepsilon_0^{2/3}.
\end{align}

Furthermore, by \eqref{eq: W_D,phi decom}, \eqref{eq:bound 24 nR^2} and \eqref{eq:bound for sum iR or j piR a_ij^2}, we have
\begin{align}\label{eq:lower bounf for W_D,phi}
    \E [W_{D,\phi}^2]\ge &\sum_{k\neq b^\dagger}  R_{A'[k]}^2+ \frac{1}{n_{[b^\dagger]}-1} \sumRb \{(a')^\textup{s}_{ij}\}^2-\frac{24 (n_{[b^\dagger]}-1)R^2_{A'[b^\dagger]}}{(n_{[b^\dagger]}-l)(n_{[b^\dagger]}-l-1)}\notag \\
    \ge & \sum_{k\neq b^\dagger}  R_{A'[k]}^2+\frac{1}{n_{[b^\dagger]}-1} \sum_{i,j\in \mathcal I_{[b^\dagger]}}\{(a')^\textup{s}_{ij}\}^2- \frac{1}{n_{[b^\dagger]}-1}\sum_{i,j\in \mathcal I_{[b^\dagger]}: i\in \mathcal R\atop or~ j\in \pi(\mathcal R)}\{(a')^\textup{s}_{ij}\}^2\notag\\
    &~~~~~~~ -\frac{24 (n_{[b^\dagger]}-1)R^2_{A'[b^\dagger]}}{(n_{[b^\dagger]}-l)(n_{[b^\dagger]}-l-1)}\notag\\
    \ge & 1- \frac{2n_{[b^\dagger]} \varepsilon_0^{2/3}}{(n_{[b^\dagger]}-1)}-\frac{24 (n_{[b^\dagger]}-1)n_{[b^\dagger]}}{(n_{[b^\dagger]}-l)(n_{[b^\dagger]}-l-1)}\frac{R^2_{A'[b^\dagger]}}{n_{[b^\dagger]}}.\notag\\
    \ge & 1-4\varepsilon_0^{2/3}- 360\times 1/3600\notag\\
    =&9/10-4\varepsilon_0^{2/3}. 
\end{align}
Also, we can obtain the upper bound:
\begin{align}\label{eq:upper bounf for W_D,phi}
    \E [W_{D,\phi}^2] \le& \sum_{k\neq b^\dagger}  R_{A'[k]}^2+ \frac{(n_{[b^\dagger]}-1)R^2_{A'[b^\dagger]}}{n_{[b^\dagger]}-l-1} +\frac{24 (n_{[b^\dagger]}-1)R^2_{A'[b^\dagger]}}{(n_{[b^\dagger]}-l)(n_{[b^\dagger]}-l-1)}\notag\\
    \le &1+\frac{lR^2_{A'[b^\dagger]}}{n_{[b^\dagger]}-l-1}+\frac{24 (n_{[b^\dagger]}-1)R^2_{A'[b^\dagger]}}{(n_{[b^\dagger]}-l)(n_{[b^\dagger]}-l-1)}.
\end{align}
Since $l\le 4$ and $\nk-l\ge 2$, we have that
$$ \E [W_{D,\phi}^2]\le 1+\frac{28(n_{[b^\dagger]}-1)R^2_{A'[b^\dagger]}}{(n_{[b^\dagger]}-l)(n_{[b^\dagger]}-l-1)}=1+\frac{28(n_{[b^\dagger]}-1)n_{[b^\dagger]}}{(n_{[b^\dagger]}-l)(n_{[b^\dagger]}-l-1)}\frac{R^2_{A'[b^\dagger]}}{n_{[b^\dagger]}}\le \frac{5}{4}.$$

\paragraph{Step 5. Apply Lemma \ref{lemma:concentration}}

Here we apply Lemma \ref{lemma:concentration} to bound $\pr(W_{D,\phi}\in [t,t+\alpha])$. 

First, by \eqref{eq:thm1 second bounf main numerator nk-l=0}, the numerator of the right-hand side of Lemma \ref{lemma:concentration} can be bounded by
\begin{align*}
 & (\sqrt{\E [W_{D,\phi}^2]}+ |L|/\lambda)(\frac{\alpha}{2}+\delta)+B_0  \\
  \le& \left(\sqrt{71}+\frac 1 5 \right) \left(\frac{\alpha}{2}+96 \sumk\sumijk   |(a')^\textup{s}_{ij}|^3/\nk\right)+48\sqrt{7}\sumk\sumijk   |(a')^\textup{s}_{ij}|^3/\nk.\\
  = &c_1' \alpha+c_2'\sumk \left(\sumijk |(a')^\textup{s}_{ij}|^3/\nk\right),
\end{align*}
where $c_1'=(\sqrt{71}+1/ 5)/2,$ and $c_2'=96(\sqrt{71}+1/ 5)+48\sqrt 7$. 

Second, for the denominator of the right-hand side of Lemma \ref{lemma:concentration}, by \eqref{eq:bound for L/lambda} and \eqref{eq:upper bounf for W_D,phi}, we have
\begin{align*}
     \E [W_{D,\phi}^2]- \E [|W_{D,\phi} L|]/\lambda -1/2  
\ge  \frac{9}{10}-4\varepsilon_0^{2/3}-\sqrt{\frac 5 4} \frac 1 5-\frac 1 2
\ge \frac{4-\sqrt 5}{10}-4\varepsilon_0^{2/3}.
\end{align*}
To make it positive, we further require $\varepsilon_0<(80)^{-3/2}$. Then we have $$\E [W_{D,\phi}^2]- \E [|W_{D,\phi} L|]/\lambda -1/2\ge 1/10.$$
Therefore,  we have  
\begin{align}\label{eq:concentralition, thm1}
    \pr (W_{D,\phi}\in [t,t+\alpha])\le 10 \left[c_1'\alpha +c_2'\sumk\left\{\sumijk|(a')^\textup{s}_{ij}|^3/\nk\right\}\right].
\end{align}

The inequalities \eqref{eq:concentration nk=l} and \eqref{eq:concentralition, thm1} lead to
\begin{align}\label{eq:unified concentration}
    \pr(W_{D,\phi}\in [t,t+\alpha])\le C_1 \alpha+C_2\sumk\sumijk|(a')^\textup{s}_{ij}|^3/\nk,
\end{align}
for $C_1=\max\{10c_1',5/4\}$ and $C_2=\max\{10c_2',40(2+\sqrt 7)\}$.
\subsubsection{Finish the Proof}
We have 
\begin{align*}
    \mathcal A_3\le& \frac{1}{\alpha}\E \left[|W_{(A')^\textup{s},\pi}^*-W_{(A')^\textup{s},\pi}|\sup_{t\in \mathbb R}\pr (W_{D,\phi}\in [t,t+\alpha])\right]\\
    \le& \frac  1\alpha\left[C_1 \alpha +C_2\sumk\sumijk |(a')^\textup{s}_{ij}|^3/\nk\right] \E [|W_{(A')^\textup{s},\pi}^*-W_{(A')^\textup{s},\pi}|]\\
    \le &\frac{80}{\alpha }\left[C_1 \alpha +C_2\sumk\sumijk |(a')^\textup{s}_{ij}|^3/\nk\right]\sumk\sumijk |(a')^\textup{s}_{ij}|^3/\nk\\
    =&  80\left[C_1 \sumk \sumijk |(a')^\textup{s}_{ij}|^3/\nk+\frac{C_2}{\alpha}\left\{\sumk \sumijk |(a')^\textup{s}_{ij}|^3/\nk \right\}^2\right]. 
\end{align*}
Then in the proof of Theorem \ref{theorem:beb}, we have that for any $\alpha>0$,
\begin{align*}
    & \sup_{t\in \mathbb R}|\pr(W_{(A')^\textup{s},\pi}\le t)-\Phi(t)|\\
    \le &  \mathcal A_1+\mathcal A_2+\mathcal A_3+\frac{\alpha}{\sqrt{2\pi}}\\
    \le & 80\sumk\left\{\sumijk |(a')^\textup{s}_{ij}|^3/\nk\right\}+80(\sqrt{c_6}+4)\sumk\left\{\sumijk |(a')^\textup{s}_{ij}|^3/\nk\right\}\\
    &~~~ +{80}\left[ C_1 \sumk\sumijk |(a')^\textup{s}_{ij}|^3/\nk+\frac{C_2}{\alpha}\left\{\sumk \sumijk |(a')^\textup{s}_{ij}|^3/\nk\right\}^2\right]+\frac{\alpha}{\sqrt{2\pi}}.
\end{align*}
Choose $\alpha=\sumk \sumijk |(a')^\textup{s}_{ij}|^3/\nk$, then
$$\sup_{t\in \mathbb R}|\pr(W_{(A')^\textup{s},\pi}\le t)-\Phi(t)|\le C' \sumk \sumijk |(a')^\textup{s}_{ij}|^3/\nk=C' \sumk \beta_{A'[k]}/\nk.$$
By Lemma \ref{lemma:scaling} and Lemma \ref{lemma:truncation}, for $A\in M_n^\textup{s}$, and $\sumk \beta_{A[k]}/\nk\le \varepsilon_0$, we have 
\begin{align*}
    \sup_{t\in \mathbb R}|\pr(W_{A,\pi}\le t)-\Phi(t)|\le& \sup_{t\in \mathbb R}|\pr(W_{(A')^{\textup{s}},\pi}\le t)-\Phi(t)|+c_3\sumk\frac{\beta_{A[k]}}{\nk}\\
    \le& \left(c_2C'+c_3\right)\sumk \frac{\beta_{A[k]}}{\nk},
\end{align*}
which leads to Theorem \ref{theorem:beb}.

\subsection{Proof of Corollary \ref{coro:order general-beb}}
We only need to consider 
\begin{align}\label{eq:coro order}
    \sumk \sumijk|a_{ij}^\textup{s}|^3/\nk=\frac{\frac{1}{n} \sumk \nk \sumijk |a_{ij}^0|^3/\nk^2}{ n^{1/2} (\frac{1}{n} \sumk  \sumijk |a_{ij}^0|^2/(\nk-1))^{3/2}}.
\end{align}
It is the same as \eqref{proof-coro eq} in Corollary \ref{coro:order fewstrata} and does not need to be repeated.

\subsection{Proof of Lemma \ref{thm:zerobias}}
\begin{proof}

Based on Lemmas \ref{lemma:chen prop4.6} and \ref{lemma:chenlemma4.4}, we only need to verify
\begin{align}\label{eq:zero-bias proof}
&\pr(B^\dagger=k,I^\dagger=i,J^\dagger=j,\pi^\dagger(\alpha)=\xi_\alpha,\alpha=1,\ldots,n)\notag\\
=&\frac{(W_{A,\pi}-W_{A,\pi''})^2}{\E[[W_{A,\pi}-W_{A,\pi''}]^2]}\cdot\pr(B=k,I=i,J=j,\pi(\alpha)=\xi_\alpha,\alpha=1,\ldots,n).
\end{align}
We assume $i,j\in\Ik $ and $\xi_\alpha\in \Ik $ for $\alpha\in\Ik $, $k=1,\ldots,K$. Define
$$
b(i,j,p,q)=a_{ip}+a_{jq}-a_{iq}-a_{jp},
$$
so that $b(i,j,p,q)=-b(i,j,q,p)=-b(j,i,p,q)$.
\paragraph{Simplify the left-hand side of \eqref{eq:zero-bias proof}}
We have
\begin{align*}
&\pr(B^\dagger=k,I^\dagger=i,J^\dagger=j,\pi^\dagger(\alpha)=\xi_\alpha,\alpha=1,\ldots,n)\\
=&\pr(B^\dagger=k)\cdot\pr(I^\dagger=i,J^\dagger=j,\pi^\dagger(i)=\xi_i,\pi^\dagger(j)=\xi_j\mid B^\dagger=k)\\
&\cdot \pr(\pi^\dagger(\alpha_1)=\xi_{\alpha_1},\alpha_1\in \Ik, \alpha_1\notin\{i,j\})\cdot \pr(\pi^\dagger(\alpha_2)=\xi_{\alpha_2},\alpha_2\in \mathcal I_{[k']}, k'\neq k).
\end{align*}
The definition of $B^\dagger$, $I^\dagger,J^\dagger,\pi^\dagger$ gives the first two terms, and Proposition \ref{prop:permutation property} gives the last two terms. Therefore, the left-hand side of \eqref{eq:zero-bias proof} equals 
\begin{align*}
\frac{\sigma_{A[k]}^2}{\sigma^2_A}\cdot\frac{b^2(i,j,\xi_i,\xi_j)}{4\nk^2(\nk-1)\sigma_{A[k]}^2}\cdot\frac{1}{(\nk-2)!}\cdot\prod_{k'\neq k}\frac{1}{n_{[k']}!}=\frac{b^2(i,j,\xi_i,\xi_j)}{4\nk\sigma^2_A}\prod_{k'=1}^K\frac{1}{n_{[k']}!}.
\end{align*}
\paragraph{Simplify the right-hand side of \eqref{eq:zero-bias proof}}
By definition, we have
$$
W_{A,\pi}-W_{A,\pi''}=a_{i,\pi(i)}+a_{j,\pi(j)}-a_{i,\pi(j)}-a_{j,\pi(i)}=b(i,j,\pi(i),\pi(j)).
$$
The result in \eqref{eq:var diff W_A,pi} implies
$$\E[[W_{A,\pi}-W_{A,\pi''}]^2]=\frac{4\sigma_A^2}{n-K}.$$
Furthermore, combining the definition of $B$ and $\pi$, we have
\begin{align*}
    &\frac{(W_{A,\pi}-W_{A,\pi''})^2}{\E[[W_{A,\pi}-W_{A,\pi''}]^2]}\cdot\pr(B=k,I=i,J=j,\pi(\alpha)=\xi_\alpha,\alpha=1,\ldots,n)\\
    =& \frac{b^2(i,j,\xi_i,\xi_j)}{4\sigma^2_A/(n-K)}\cdot \pr(B=k)\cdot \pr(I=i,J=j\mid B=k)\cdot\pr(\pi(\alpha)=\xi_\alpha,\alpha=1,\ldots,n)\\
    =& \frac{b^2(i,j,\xi_i,\xi_j)}{4\sigma^2_A/(n-K)}\cdot \frac{\nk-1}{n-K}\cdot \frac{1}{\nk(\nk-1)}\cdot \prod_{k'=1}^K\frac{1}{n_{[k']}!}\\
    =&\frac{b^2(i,j,\xi_i,\xi_j)}{4\nk\sigma_A^2}\prod_{k'=1}^K\frac{1}{n_{[k']}!}.
\end{align*}

Therefore, we prove \eqref{eq:zero-bias proof}.
\end{proof}

\subsection{Proof of Corollary \ref{coro:sampling}}
\begin{proof}
We have
  \begin{align*}
    A^{\textup{0c}}=\diag \left\{ \frac{w_{[k]}}{n_{[k]1}}\left(\begin{matrix}
   (Y_{{[k]1}}-\bar Y_{[k]}) \bold{1}^\T_{n_{[k]1}} &\bold{0}^\T_{n_{[k]0}}\\
   \vdots &\vdots\\
   (Y_{{[k]\nk}}-\bar Y_{[k]}) \bold{1}^\T_{n_{[k]1}}  &\bold{0}^\T _{n_{[k]0}}
        \end{matrix}\right)\right\}_{k=1,\ldots,K},
        \end{align*}
        Then $A^\textup{s}=A^0/\sigma$. Applying Lemma \ref{coro:apply BEB}, we have that
        \begin{align}\label{eq:beb sampling order proof}
            \sumk \frac 1 {\nk\sigma^3}\sumijk |a_{ij}^\textup{0c}|^3
            =&\sumk \frac 1 {\nk\sigma^3} \frac{w_{[k]}^3}{n_{[k]1}^3}\sumik n_{[k]1} |Y_i-\bar Y_{[k]}|^3\notag\\
            =& \sumk \frac{w_{[k]}^3}{\nk n_{[k]1}^2 \sigma^3}\sumik |Y_i-\bar Y_{[k]}|^3.\notag\\
            =&\sumk \frac{w_{[k]}^3}{ n_{[k]1}^2 \sigma^3}\Mkthree(Y).
        \end{align}
        Now we quantify the order of \eqref{eq:beb sampling order proof} under $w_{[k]}=\nk/n$, $k=1,\ldots,K$.
        First, recall $$\sigma^2=\frac 1 n \sumk w_{[k]}\frac{n_{[k]0}}{n_{[k]1}} S_{[k]}^2=\frac 1 n \sumk w_{[k]}\frac{1-p_{[k]}}{p_{[k]}} S_{[k]}^2.$$
When $p_{[k]}$ are bounded away from $0$ and $1$ and $\sumk  w_{[k]} S_{[k]}^2$ have nonzero finite limits, $\sigma^2\asymp 1/n.$
Second, when  $p_{[k]}$ are bounded away from $0$ and $1$ and $\sumk  w_{[k]} \Mkthree(Y)$ have nonzero finite limits, we have $$\sumk \frac{w_{[k]}^3}{ n_{[k]1}^2}\Mkthree(Y)=\frac 1 {n^2}\sumk \frac{1}{ p_{[k]}^2 }w_{[k]}\Mkthree(Y)\asymp \frac 1 {n^2}.$$
Therefore, the order of \eqref{eq:beb sampling order proof} is $1/n^{1/2}$.
\end{proof}

\subsection{Proof of Corollary \ref{coro:strati exp}}
\begin{proof}
We have 
  \begin{align*}
    A^{\textup{0c}}=\diag \left\{ w_{[k]}\left(\begin{matrix}
  n_{[k]1}^{-1} (Y_{{[k]1}}(1)-\bar Y_{[k]}(1))\bold{1}_{n_{[k]1}}^\T  &n_{[k]0}^{-1} (Y_{{[k]1}}(0)-\bar Y_{[k]}(0))\bold{1}_{n_{[k]0}}^\T \\
   \vdots &\vdots\\
    n_{[k]1}^{-1}( Y_{{[k]\nk}}(1)-\bar Y_{[k]}(1))\bold{1}_{n_{[k]1}}^\T  &n_{[k]0}^{-1} (Y_{{[k]\nk}}(0)-\bar Y_{[k]}(0))\bold{1}_{n_{[k]0}}^\T 
        \end{matrix}\right)\right\}_{k=1,\ldots,K}.
        \end{align*}
          Therefore, applying Lemma \ref{coro:apply BEB}, we have that
        \begin{align}\label{eq:exp proof}
            &\sumk \frac 1 {\nk\sigma^3}\sumijk |a_{ij}^\textup{0c}|^3\notag\\
            =&\sumk \frac{w_{[k]}^3}{\nk\sigma^3}\sumik \left\{\frac{|Y_i(1)-\bar Y_{[k]}(1)|^3}{n_{[k]1}^3} n_{[k]1}+\frac{|Y_i(0)-\bar Y_{[k]}(0)|^3}{n_{[k]0}^3} n_{[k]0}\right\}\notag\\
            =& \sumk \frac{w_{[k]}^3}{\nk\sigma^3}\sumik 
            \left\{\frac{|Y_i(1)-\bar Y_{[k]}(1)|^3}{n_{[k]1}^2} +\frac{|Y_i(0)-\bar Y_{[k]}(0)|^3}{n_{[k]0}^2}\right\}\notag\\
            =& \sumk \frac{w_{[k]}^3}{\sigma^3} 
            \left\{\frac{\Mkthree(Y(1))}{n_{[k]1}^2} +\frac{\Mkthree(Y(0))}{n_{[k]0}^2}\right\}.
            \end{align}
Under Assumption of Corollary \ref{coro:strati exp}, $\sigma^2\asymp 1/n$. Furthermore, for $z=0,1$, we have
$$\sumk \frac{w_{[k]}^3}{n_{[k]z}^2} 
            {\Mkthree(Y(z))}=\frac 1 {n^2} \sumk \frac{1}{\{z p_{[k]}+(1-z)(1-p_{[k]})\}^2} 
            {w_{[k]}\Mkthree(Y(z))} \asymp \frac{1}{n^2}.$$
Therefore, the order of \eqref{eq:exp proof} is $1/n^{1/2}$. 
\end{proof}

\subsection{Proof of Corollary \ref{coro:post-sampling}}

\subsubsection{Proof of \eqref{ineq:post-sampling}}
We have
\begin{align}\label{eq:coro post-sampling}
        &\sup_{t\in \mathbb R}\left|\pr(\frac{\hat \gamma_{\textup{ps}}-\gamma}{\sigma}\le t \mid\mathcal D_1)-\Phi(t)\right|\notag\\
        \le &\sup_{t\in \mathbb R}\left|\sum_{ \bold U_1\in \mathcal D_1} \pr(\bold U_1\mid \mathcal D_1)\{\pr(\frac{\hat \gamma_{\textup{ps}}-\gamma}{\sigma}\le t \mid \bold U_1)-\Phi(\frac{\sigma t}{\sigma_{\textup{ps}}(\bold U_1)} )+\Phi(\frac{\sigma t}{\sigma_{\textup{ps}}(\bold U_1)} )-\Phi(t)\}\right|\notag\\
        \le  &\sum_{ \bold U_1\in \mathcal D_1}\pr(\bold U_1\mid \mathcal D_1) \sup_{t\in \mathbb R}| \pr(\frac{\hat \gamma_{\textup{ps}}-\gamma}{\sigma}\le t \mid \bold U_1)-\Phi(\frac{\sigma t}{\sigma_{\textup{ps}}(\bold U_1)} )| \notag\\
        &~~~~~~~~~~+ \sup_{t\in \mathbb R}|\sum_{ \bold U_1\in \mathcal D_1}\pr(\bold U_1\mid \mathcal D_1)\{\Phi(\frac{\sigma t}{\sigma_{\textup{ps}}(\bold U_1)} )-\Phi(t)\}|\notag\\
        =& \sum_{ \bold U_1\in \mathcal D_1} \pr(\bold U_1\mid \mathcal D_1) \sup_{t\in \mathbb R}| \pr(\frac{\hat \gamma_{\textup{ps}}-\gamma}{\sigma_{\textup{ps}}(\bold U_1)}\le t \mid \bold U_1)-\Phi(t )| + \sup_{t\in \mathbb R}|\E[\Phi(\frac{\sigma t}{\sigma_{\textup{ps}}(\bold U_1)})\mid\mathcal D_1]-\Phi(t)|.
    \end{align}
   By Corollary \ref{coro:sampling}, given $\bold U_1$, we have  
    \begin{align*}
        \sup_{t\in \mathbb R}\left|\pr(\frac{\hat\gamma_{\textup{ps}}-\gamma}{\sigma_{\textup{ps}}(\bold U_1)}\le t \mid \bold U_1)-\Phi(t)\right|
        \le  C \left(\sumk\frac{\pik^3}{\sigma^3(\bold U_1)} \frac{\Mkthree(Y)}{n_{[k]1}^2}\right).
    \end{align*}
    Consequently, we can bound the first term in the right-hand side of \eqref{eq:coro post-sampling} as 
    \begin{align}\label{eq:post sampling first term}
       & \sum_{ \bold U_1\in \mathcal D_1}\pr(\bold U_1\mid \mathcal D_1) \sup_{t\in \mathbb R}\left| \pr(\frac{\hat \gamma_{\textup{ps}}-\gamma}{\sigma_{\textup{ps}}(\bold U_1)}\le t \mid \bold U_1)-\Phi(t )\right|\notag\\
        \le & \E \left[C \sumk\frac{\pik^3}{\sigma^3(\bold U_1)} \frac{\Mkthree(Y)}{n_{[k]1}^2}\mid\mathcal D_1\right]\notag\\
        = &C \left[\sumk{\pik^3 \E[n_{[k]1}^{-2}\sigma^{-3}(\bold U_1)\mid\mathcal D_1]}\Mkthree(Y)\right].
    \end{align}
Therefore, \eqref{eq:coro post-sampling} and \eqref{eq:post sampling first term} imply \eqref{ineq:post-sampling} in Corollary \ref{coro:post-sampling}.

 \subsubsection{Order of the first term in \eqref{ineq:post-sampling}}

Because of
$$\sigma_{\textup{ps}}^2(\bold U_1)=\frac 1 n \sumk w_{[k]} \frac {n_{[k]0}}{n_{[k]1}} S_{[k]}^2 = \sumk \frac{w_{[k]}^2 S_{[k]}^2}{n_{[k]1}}-\frac 1 n \sumk w_{[k]} S_{[k]}^2,$$
and Lemma \ref{lemma:weight S}, we have 
\begin{align*}
\sigma_{\textup{ps}}^2(\bold U_1)\ge \frac{(\sumk w_{[k]} S_{[k]})^2}{n_1}-\frac{\sumk w_{[k]} S_{[k]}^2 }{n}\ge \frac{(\frac{\min_{1\le k\le K} S_{[k]}}{\max_{1\le k\le K} S_{[k]}} -p)(\sumk w_{[k]}S_{[k]}^2)}{p n}.
\end{align*}

Therefore, $\sigma_{\textup{ps}}^2(\bold U_1)\ge \underline C n^{-1}$ for some constant $\underline C$. 

We bound the first term in \eqref{ineq:post-sampling}:
\begin{align}\label{eq:first in post sampling simplify1}
    \sumk{\pik^3 \E[n_{[k]1}^{-2}\sigma^{-3}(\bold U_1)\mid\mathcal D_1]}\Mkthree(Y)\le & \underline C^{-3/2}n^{3/2} \sumk{\pik^3 \E[n_{[k]1}^{-2}\mid\mathcal D_1]}\Mkthree(Y)\notag\\
    = &  \underline C^{-3/2}n^{-1/2} \sumk{\pik \E[\frac{\nk^2}{n_{[k]1}^2}\mid\mathcal D_1]}\Mkthree(Y).
\end{align}

For the right-hand side of \eqref{eq:first in post sampling simplify1}, we have
$$\E\left[\frac{\nk^2}{n_{[k]1}^2}\mid\mathcal D_1\right]=\var\left(\frac{\nk}{n_{[k]1}}\mid\mathcal D_1\right)+\E\left[\frac{\nk}{n_{[k]1}}\mid\mathcal D_1\right]^2,$$
where the first term on the right-hand side is $O(1/n)$ from 
Lemma \ref{lemma:hyper} and the second term tends to $1/p^2$ from Lemma \ref{lemma: miratrix lemma1}. Therefore, the first term on the right-hand side of \eqref{ineq:post-sampling} is $O(1/n^{1/2})$.

\subsubsection{Order of the second term in \eqref{ineq:post-sampling}}
Denote $g(x)=\Phi(x^{-1/2} \sigma t)$. By the Taylor expansion, there exists $x_0$ between $x$ and $\sigma^2$ such that, 
$$\Phi\left(\frac{\sigma t}{\sqrt x}\right)-\Phi(t)=g(x)-g(\sigma^2)=g'(\sigma)(x-\sigma^2)+g''(x_0)(x-\sigma^2)^2.$$

Choose $x=\sigma^2(\bold U_1)$, so that $x_0=x_0(\bold U_1)$ depends on $\bold U_1$.
By $\E[\sigma_{\textup{ps}}^2(\bold U_1)\mid \mathcal D_1]=\sigma^2$, we have 
\begin{align*}
    |\E[\Phi(\frac{\sigma}{\sigma(\bold U_1)}\mid \mathcal D_1)-\Phi(t)]|=&|\E[g''(x_0(U_1)) \{\sigma^2(\bold U_1)-\sigma^2\}^2\mid \mathcal D_1]|\\
    \le & \E[|g''(x_0(\bold U_1))\{\sigma^2(\bold U_1)-\sigma^2\}^2|\mid \mathcal D_1].
\end{align*}

By Lemma \ref{lemma:g function}, we have 
\begin{align*}
    |g''(x_0)|\le& \frac 3 4 \sigma |t| x_0^{-5/2}\phi\left(\frac{\sigma t}{\sqrt{x_0}} \right)+\frac 1 4 \sigma^3 |t|^3 x_0^{-7/2} \phi(\frac {\sigma t}{\sqrt{x_0}}).
\end{align*}

Because $$\lim_{n\rightarrow \infty}n \sigma_{\textup{ps}}^2=\lim_{n\rightarrow \infty}\sumk \E[\frac{n_{[k]0}}{n_{[k]1}}\mid \mathcal D_1]w_{[k]}S_{[k]}^2= \frac{1-p}{p} \lim_{n\rightarrow \infty}\sumk w_{[k]}S_{[k]}^2$$
is a positive constant, where the last equation holds from Lemma \ref{lemma: miratrix lemma1}. Therefore, $\sigma_{\textup{ps}}^2$ is of the order $1/n$ and $\sigma^2(\bold U_1)\ge \underline c \sigma^2$ for some constant $\underline c$.

Because $x_0(\bold U_1)$ is between $\sigma^2$ and $\sigma^2(\bold U_1)$,  there exists a constant $B_1$, such that 
$$|g''(x_0(\bold U_1))|\le \sigma^{-4}\left(\frac{3}{4}\underline c^{-5/2} |t|+\frac{1}{4} \underline c^{-7/2} |t|^3 \right)\phi(\frac{|t|}{\underline  c^{1/2}} )\le B_1 \sigma^{-4},$$
where the last inequality holds because for any constant 
$C'>0$, we have $(|t|+|t|^3)\exp (-C't^2)\rightarrow 0$ as $|t|\rightarrow \infty$.
Therefore, 
$$\left|\E[\Phi(\frac{\sigma t}{\sigma_{\textup{ps}}(\bold U_1)})\mid\mathcal D_1]-\Phi(t)\right|\le \frac{B_1 \var(\sigma_{\textup{ps}}^2(\bold U_1)\mid \mathcal D_1)}{\sigma^4}.$$

Since $\sigma^{-4}=O(n^2)$, we only need to consider 
$$\var(\sigma^2(\bold U_1)\mid \mathcal D_1)= \var\left(\frac 1 {n} \sumk w_{[k]} \frac{n_{[k]0}}{n_{[k]1}}S_{[k]}^2\mid \mathcal D_1\right)=\frac 1 {n^2}\var( \sumk w_{[k]} \frac{\nk}{n_{[k]1}}S_{[k]}^2\mid \mathcal D_1).$$
The Cauchy--Schwarz inequality  implies that
\begin{align*}
  \var(\sigma^2(\bold U_1)\mid \mathcal D_1)\le &\frac {K}{n^2}\sumk w_{[k]}^2 \var(\frac{\nk}{n_{[k]1}}\mid \mathcal D_1) S_{[k]}^4\\
  \le& \frac {K}{n^2} \max_{1\le k\le K} S_{[k]}^4 \max_{1\le k\le K}  \var\left(\frac{\nk}{n_{[k]1}}\mid \mathcal D_1\right) =O\left(\frac 1 {n^3}\right).  
\end{align*}
Therefore, the second term in the right-hand side of \eqref{ineq:post-sampling} is $O(1/n)$.

\subsection{Proof of Corollary \ref{coro:post}}
\subsubsection{Proof of \eqref{ineq:post}}
We have 
\begin{align*}
        &\sup_{t\in \mathbb R}\left|\pr(\frac{\hat \tau_{\textup{ps}}-\tau}{\sigma}\le t \mid\mathcal D)-\Phi(t)\right|\\
        \le &\sup_{t\in \mathbb R}\left|\sum_{ \bold U\in \mathcal D} \pr(\bold U\mid \mathcal D)\{\pr(\frac{\hat \tau_{\textup{ps}}-\tau}{\sigma}\le t \mid \bold U)-\Phi(\frac{\sigma t}{\sigma_{\textup{ps}}(\bold U)} )+\Phi(\frac{\sigma t}{\sigma_{\textup{ps}}(\bold U)} )-\Phi(t)\}\right|\\
        \le  &\sum_{ \bold U\in \mathcal D}\pr(\bold U\mid \mathcal D) \sup_{t\in \mathbb R}\left| \pr(\frac{\hat \tau_{\textup{ps}}-\tau}{\sigma}\le t \mid  \bold U)-\Phi(\frac{\sigma t}{\sigma_{\textup{ps}}(\bold U)} )| + \sup_{t\in \mathbb R}|\sum_{ \bold U\in \mathcal D}\pr(\bold U\mid \mathcal D)\Phi(\frac{\sigma t}{\sigma_{\textup{ps}}(\bold U)} )-\Phi(t)\}\right|\\
        =& \sum_{ \bold U\in \mathcal D} \pr(\bold U\mid \mathcal D) \sup_{t\in \mathbb R}\left| \pr(\frac{\hat \tau_{\textup{ps}}-\tau}{\sigma_{\textup{ps}}(\bold U)}\le t \mid \bold U)-\Phi(t )\right| + \sup_{t\in \mathbb R}\left|\E[\Phi(\frac{\sigma t}{\sigma_{\textup{ps}}(\bold U)})\mid \mathcal D]-\Phi(t)\right|.
    \end{align*}
    Then, by Corollary \ref{coro:strati exp}, we have 
\begin{align*}
        \sup_{t\in \mathbb R}\left|\pr\left(\frac{\hat\tau_{\textup{ps}}-\tau}{\sigma_{\textup{ps}}(\bold U)}\le t\mid \bold U\right)-\Phi(t)
       \right|
        \le   \frac C {\sigma_{\textup{ps}}^3(\bold U)}\sumk  w_{[k]} ^3\left\{\frac{\Mkthree (Y(1))}{n_{[k]1}^2}+\frac{\Mkthree (Y(0))}{n_{[k]0}^2}\right\},
    \end{align*}
    Consequently, 
    \begin{align*}
       & \sum_{ \bold U\in \mathcal D}\pr(\bold U\mid \mathcal D) \sup_{t\in \mathbb R}\left| \pr(\frac{\hat \tau_{\textup{ps}}-\tau}{\sigma_{\textup{ps}}(\bold U)}\le t \mid \bold U)-\Phi(t)\right|\\
        \le & \E \left[   \frac C {\sigma_{\textup{ps}}^3(\bold U)}\sumk  w_{[k]} ^3\left\{\frac{\Mkthree (Y(1))}{n_{[k]1}^2}+\frac{\Mkthree (Y(0))}{n_{[k]0}^2}\right\} \mid \mathcal D\right]\\
       \le & C \E\left[\sumk\frac{\pik^3}{\sigma^3_{\textup{ps}}(\bold U)} 
       \left(\frac{\Mkthree(Y(1))}{n_{[k]1}^2}+\frac{\Mkthree(Y(0))}{n_{[k]0}^2}\right)\mid \mathcal D\right ].
    \end{align*}
\subsubsection{Order of the first term in \eqref{ineq:post}}
Because of
$$\sigma_{\textup{ps}}^2(\bold U)=\frac 1 n \sumk w_{[k]}( \frac {S_{[k]1}^2}{p_{[k]}} + \frac {S_{[k]0}^2}{1-p_{[k]}} -S_{[k]\tau}^2) = \sumk \frac{w_{[k]}^2 S_{[k]1}^2}{n_{[k]1}}+\sumk \frac{w_{[k]}^2 S_{[k]0}^2}{n_{[k]0}}-\frac 1 n \sumk w_{[k]} S_{[k]\tau}^2$$
and Lemma \ref{lemma:weight S}, we have that
\begin{align*}
\sigma_{\textup{ps}}^2(\bold U)\ge &\frac{(\sumk w_{[k]} S_{[k]1})^2}{n_1}+\frac{(\sumk w_{[k]} S_{[k]0})^2}{n_0}-\frac{\sumk w_{[k]} S_{[k]\tau}^2 }{n}\\
\ge& \frac{1}{n}\left(\frac{\min_{1\le k\le K} S_{[k]1}}{\max_{1\le k\le K} S_{[k]1}}\frac{S_1^2}{p}+\frac{\min_{1\le k\le K} S_{[k]0}}{\max_{1\le k\le K} S_{[k]0}}\frac{S_0^2}{1-p}-S_\tau^2\right).
\end{align*}

Therefore, $\sigma_{\textup{ps}}^2(\bold U)\ge \underline C n^{-1}$ for some constant $\underline C$ and further $\sigma_{\textup{ps}}^2\ge \underline C n^{-1}$. 
Similar to the proof of Corollary \ref{coro:post-sampling}, we can obtain that the first term in the right-hand side of \eqref{ineq:post} is $O(1/n)$.

\subsubsection{Order of the second term in \eqref{ineq:post}}
Similar to the proof of 
Corollary \ref{coro:post-sampling}, we have 
$$\E[\nk/n_{[k]1}\mid \mathcal D]\rightarrow 1/p,\quad \E[\nk/n_{[k]0}\mid \mathcal D]\rightarrow 1/(1-p).$$

Therefore, $$\limsup_n n \sigma_{\textup{ps}}^2\le  \limsup_n \sumk w_{[k]} (\E[p_{[k]}^{-1}\mid \mathcal D]S_{[k]1}^2+\E[(1-p_{[k]})^{-1}\mid \mathcal D]S_{[k]0}^2)= \lim_n \left(\frac{S_1^2}{p}+\frac{S_0^2}{1-p}\right),$$
which implies that $\sigma_{\textup{ps}}^2$ is of the order $1/n$. Thus, $\sigma^2(\bold U)\ge \underline c \sigma^2$ for some constant $\underline c$.

Similar to the proof of Corollary \ref{coro:post-sampling}, we have that
$$\left|\E[\Phi(\frac{\sigma t}{\sigma_{\textup{ps}}(\bold U)})\mid\mathcal D]-\Phi(t)\right|\le \frac{B_1 \var(\sigma_{\textup{ps}}^2(\bold U)\mid \mathcal D)}{\sigma^4},$$
for some constant $B_1$.
By $\sigma^{-4}=O(n^2)$ and the Cauchy--Schwarz inequality, we only need to consider 
\begin{align*}
    \var(\sigma^2(\bold U)\mid \mathcal D)\le & K \sumk \var\left\{ w_{[k]}^2 \left(\frac{S_{[k]1}^2}{n_{[k]1}}+\frac{S_{[k]0}^2}{n_{[k]0}}\right)\mid \mathcal D\right\}\\
    \le & 2 K \sumk \left\{\var\left( w_{[k]}^2 \frac{S_{[k]1}^2}{n_{[k]1}}\mid \mathcal D\right)+\var\left( w_{[k]} ^2\frac{S_{[k]0}^2}{n_{[k]0}}\mid \mathcal D\right)\right\}.
\end{align*}

We have 
\begin{align*}
    \sumk \var( w_{[k]}^2 \frac{S_{[k]1}^2}{n_{[k]1}}\mid \mathcal D)\le\frac{1}{n^2} \max_{1\le k\le K}S_{[k]1}^4 \max_{1\le k\le K}  \var(\frac{\nk}{n_{[k]1}}\mid \mathcal D_1)=O\left(\frac 1 {n^3}\right).
\end{align*}
Similarly,
$$\sumk \var( w_{[k]}^2 \frac{S_{[k]0}^2}{n_{[k]0}}\mid \mathcal D)=O\left(\frac 1 {n^3}\right).$$
Therefore, the second term in the right-hand side of \eqref{ineq:post} is $O(1/n)$.

\subsection{Proof of Corollary \ref{coro:IV}}
\begin{proof}
We can choose 
    \begin{align*}
    A^0=\diag \left\{ \left(\begin{matrix}
  (R_{{[k]1}}-\bar R_{[k]})(Z_{{[k]1}}-\bar Z_{[k]}) & \cdots&(R_{{[k]1}}-\bar R_{[k]})(Z_{{[k]\nk}}-\bar Z_{[k]}) \\
   \vdots& &\vdots\\
(R_{{[k]\nk}}-\bar R_{[k]})(Z_{{[k]1}}-\bar Z_{[k]})&\cdots&(R_{{[k]\nk}}-\bar R_{[k]})(Z_{{[k]\nk}}-\bar Z_{[k]}) 
    \end{matrix}\right)\right\}_{k=1,\ldots,K},
        \end{align*}
   Then $A^\textup{s}=A^0/\sigma$, and
    \begin{align*}
            \sumk \frac 1 \nk\sumijk |a_{ij}^\textup{s}|^3=&\sumk \frac{1}{\nk\sigma^3}
        \left(\sumik |R_i-\bar R_{[k]}|^3\right)\left(\sumik |Z_i-\bar Z_{[k]}|^3\right)\\
        =& \frac 1 {\sigma^3}\sumk {\nk }\Mkthree(Z)\Mkthree(R).
            \end{align*}
        
\end{proof}

\subsection{Proof of Proposition \ref{prop:multi-scaling}}
\begin{proof}
We have 
    \begin{align*}
        V^{-1/2}(\Gamma-\E [\Gamma])
        =V^{-1/2}(\tr(G_1^0 P),\ldots,\tr(G_H^0 P))^\T
        =&V^{-1/2}\left(\begin{matrix}
            \vec(G_1^0)^\T\\
            \ldots\\
            \vec(G_H^0)^\T
        \end{matrix}\right)\cdot\vec (P) 
    \end{align*}
 By the definition of $\tilde G_h,h=1,\ldots,H$, it further equals   
    \begin{align*}
        V^{-1/2}(\Gamma-\E [\Gamma])= \left(\begin{matrix}
            \vec(\tilde G_1)^\T\\
            \ldots\\
            \vec(\tilde G_H)^\T
        \end{matrix}\right)\cdot\vec (P)
        =(\tr(\tilde G_1 P),\ldots,\tr(\tilde G_H P))^\T.
    \end{align*}

\end{proof}

\subsection{Proof of Proposition \ref{prop:cond_equ}}
\begin{proof}
We have $\E [\Gamma_h ]=\sumk \nk \bar g_{h,[k]\cdot \cdot},$
  where $\bar g_{h,[k]\cdot \cdot}=\nk^{-2} \sumijk g_{h,ij}$. From $G_h\in  M_n^\textup{s}$, we have $\bar g_{h,[k]\cdot j}=\nk^{-1}\sumik g_{h,ij}=0$ and  $\bar g_{h,[k] i \cdot }=\nk^{-1}\sumjk g_{h,ij}=0$. Therefore, $\E [\Gamma]=0$.
  We review a lemma to compute $\cov(\sumik g_{h,i\pi(i)},\sumik g_{l,i\pi(i)})$.
  \begin{lemma}[Lemma S1(ii) in \citet{shi2022berry}]\label{lemma:shi2022berry lemmaS1 ii}
      For $\pi\sim \textup{Uniform}(\Pi_K)$ and $G_h,G_l\in M_n^\textup{s}$, we have    $$\cov\left(\sumik g_{h,i\pi(i)},\sumik g_{l,i\pi(i)}\right)=\frac{1}{\nk -1} \sumijk g_{h, ij}g_{l,ij}.$$
  \end{lemma}
  By Lemma \ref{lemma:shi2022berry lemmaS1 ii} and $\bar g_{h,[k]\cdot j}=\bar g_{h,[k] i \cdot }=0$, we have 
    \begin{align*}
        \cov\left(\tr( G_h P), \tr( G_l P)\right) 
     =&  \sumk\cov\left( \sumik g_{h,i \pi(i)} , \sumik g_{l,i \pi(i)} \right)\\
       =&  \sumk \frac{1}{\nk -1} \sumijk g_{h, ij}g_{l,ij}\\
       =& \langle G_h, G_l\rangle_K.
    \end{align*}
Under Condition \ref{cond:multibeb}, we have $\cov (\Gamma)=I_H$.
\end{proof} 

\subsection{Proof of Corollary \ref{coro:multibeb}}
\begin{proof}
We have $$b^\T \Gamma =\sumh b_h \tr(G_h P)=\tr \left((\sumh b_h G_h)P \right).$$
Define $G^\dagger=\sumh b_h G_h$. From $G_h\in  M_n^\textup{s}$, we have that for $k=1,\ldots, K,$
$$\sumik g_{ij}^\dagger=\sumik\sumh b_h g_{h, ij}=0,\quad \sumjk g_{ij}^\dagger=\sumjk\sumh b_h g_{h, ij}=0.$$
Then 
\begin{align*}
    \var (W_{G^\dagger,\pi})=& \sumk \frac{1}{\nk-1}\sumijk (g_{ij}^\dagger)^2
    = \langle G^\dagger,G^\dagger\rangle_K
    = \sum_{1\le h,l\le K} b_h b_l \langle G_h, G_l \rangle_K 
    =\sumh  b_h^2 =1.
\end{align*}
That is, $G^\dagger\in  M_n^\textup{s}$. By Theorem \ref{theorem:beb},
there exists a universal constant $C$ such that  
$$\sup_{t\in \mathbb R} |\pr (b^\T \Gamma\le t)-\Phi (t)|\le C \sumk\sumijk \frac{|g_{ij}^\dagger|^3}{\nk}.$$

Because $G^\dagger\in M_n^\textup{s}$, we have $$\sumk\frac{1}{\nk-1}\sumijk (g_{ij}^\dagger)^2=1.$$
Therefore, we have
$$\sumk\sumijk \frac{|g_{ij}^\dagger|^3}{\nk}\le \sumk\sumijk \frac{|g_{ij}^\dagger|^2}{\nk-1}\max_{1\le k\le K}\max_{i,j\in \Ik} |g_{ij}^\dagger|=\max_{1\le k\le K}\max_{i,j\in \Ik} |g_{ij}^\dagger|.$$
\end{proof}

\subsection{Proof of Theorem \ref{thm:convexBEB}}
First, we review Stein coupling.
\begin{definition}(Stein coupling \citep{chen2010stein})
    A triple of square integrable $H$-dimensional random vectors $(\Gamma,\Gamma',X)$ is called an $H$-dimensional Stein coupling if $$\E[X^\T f(\Gamma')-X^\T  f(\Gamma)]=\E [X^\T f(\Gamma)]$$
    for all $f:\mathbb R^H\rightarrow \mathbb R^H$ provided that the expectations exist.
\end{definition}

Recall that we choose a random index $B$ with $\pr(B=k)=(\nk-1)/(n-K)$, swap two randomly selected
 indexes $I,J$ within $\mathcal I_{[B]}$ and obtain $\pi''$. Define $\Gamma'= (W_{G_1,\pi''}, W_{G_2,\pi''},\ldots, W_{G_H,\pi''})$. From Proposition \ref{prop:stein pair}, $\E[\Gamma'-\Gamma\mid \Gamma]=-(n-K) \Gamma/2$. The following lemma connects the Stein pair with the Stein coupling.
\begin{lemma}[Remark 2.3 of \cite{fang2015rates}]\label{lemma:connect pair coupling} If $(\Gamma,\Gamma')$ is an exchangeable pair and $\E[\Gamma'-\Gamma\mid \Gamma]=-\Lambda \Gamma$ for some invertible $\Lambda$, then $(\Gamma,\Gamma',\Lambda^{-1}(\Gamma'-\Gamma)/2)$ is a Stein coupling.  
\end{lemma}
From Lemma \ref{lemma:connect pair coupling}, $(\Gamma,\Gamma',-{(n-K)}(\Gamma-\Gamma')/4)$ is a Stein coupling.
Define $D=\Gamma-\Gamma'$ and denote $D=(D_1,\ldots,D_H)^\T$ and $X=-(n-K)D/4$. Lemma \ref{lemma:compute} follows from the result in the proof of Theorem S2 of \citet{shi2022berry}, which holds for $k=1,\ldots,K$. 

\begin{lemma}\label{lemma:compute}
   (i) There exists a universal constant $C$, such that for any $k=1,\ldots, K$ and $h,l,m\in\{1,\ldots,H\}$, 
\begin{align*}
    \var\{\E[D_h^2\mid\Gamma,B=k]\}&\le \frac{CB_n^4}{\nk},\\
    \var\{\E[D_h D_l\mid \Gamma,B=k]\}&\le \frac{CB_n^4}{\nk},\\
    \var\{\E[D_h D_l D_m\mid \Gamma,B=k]\}&\le \frac{CB_n^6}{\nk}.
\end{align*}

(ii) \citep[Lemma 8]{chatterjee2007multivariate} For any $k=1,\ldots,K$ and $h=1,\ldots,H$,
\begin{align*}
   \E[(\Gamma_h-\Gamma_h')^2\mid \pi,B=k]=&\frac{2(\nk+1)}{\nk(\nk-1)}\sumik g_{h,i\pi(i)}^2+\frac{2}{\nk}+\frac{2}{\nk(\nk-1)}\Gamma_{h[k]}^2\\
    &+\frac{2}{\nk(\nk-1)}\sum_{i,j\in\Ik,i\neq j} g_{h,i\pi(j)}g_{h,j\pi(i)}.
\end{align*}

\end{lemma}
We now prove Theorem \ref{thm:convexBEB}.

\begin{proof}[Proof of Theorem \ref{thm:convexBEB}]
Recall $D=\Gamma-\Gamma'$. We review a lemma for Stein's coupling and normal approximation as follows.
\begin{lemma}[Theorem 2.1 of \citet{fang2015rates}]\label{lemma:fang2015 thm2.1}
Let $\left(\Gamma, \Gamma^{\prime}, G\right)$ be a $H$-dimensional Stein coupling. Assume $\operatorname{Cov}(\Gamma)=I_H$. Let $\xi_H$ be an $H$-dimensional standard Normal random vector. With $D=\Gamma^{\prime}-\Gamma$, suppose that there are positive constants $\alpha$ and $\beta$ such that $\|X\|_2 \leq \alpha$ and $\|D\|_2 \leq \beta$. Then there exists a universal constant $C$, such that
$$
\begin{aligned}
& \sup _{V \in \mathcal{C}}\left|\mathbb{P}\{\Gamma \in V\}-\mathbb{P}\left\{\xi_H \in V\right\}\right| \\
& \leq C\left(H^{7 / 4} \alpha \mathbb{E}\|D\|_2^2+H^{1 / 4} \beta+H^{7 / 8} \alpha^{1 / 2} \mathcal B_1^{1 / 2}+H^{3 / 8} \mathcal B_2+H^{1 / 8} \mathcal B_3^{1 / 2}\right)
\end{aligned}
$$
where
\begin{eqnarray*}
\mathcal B_1^2&=&\operatorname{Var}\left\{\mathbb{E}\left(\|D\|_2^2 \mid \Gamma\right)\right\}, \\
\mathcal B_2^2&=&\sum_{h=1}^H \sum_{l=1}^H \operatorname{Var}\left\{\mathbb{E}\left(X_h D_l \mid \Gamma\right)\right\}, \\
\mathcal B_3^2&=&\sum_{h=1}^H \sum_{l=1}^H \sum_{m=1}^H \operatorname{Var}\left\{\mathbb{E}\left(X_h D_l D_m \mid \Gamma\right)\right\}.
\end{eqnarray*}

\end{lemma}

Now, we will upper bound $\|X\|_2,\|D\|_2, \mathcal B_i,i=1,2,3$, respectively.
\paragraph{Step 1. Bound $\|X\|_2,\|D\|_2$}
Since $D_h=g_{h,I \pi(i)}+g_{h,J\pi(j)}-g_{h,I\pi(j)}-g_{h,J\pi(i)}$, we can obtain that $|D_h|\le 4 B_n$, then $|X_h|=(n-K)|D_h|/4\le (n-K)B_n$. Then $\|D\|_2\le 4\sqrt{ H} B_n$ and $\|X\|_2\le (n-K)\sqrt{H}B_n$.

\paragraph{Step 2. Bound $\E[\|D\|_2^2]$}
We first upper bound $\E[\|D\|_2^2\mid \pi]$.
 We have
\begin{align}\label{eq:thm fang Step2}
    &\E[\|D\|_2^2\mid \pi]=\sumh \E[D_h^2\mid \pi]=\sumh \E[(\Gamma_h-\Gamma'_h)^2\mid \pi].
    \end{align}
    By Lemma \ref{lemma:compute}(ii), \eqref{eq:thm fang Step2} equals
    \begin{align*}        
&\sumh \sumk\pr(B=k)\E[(\Gamma_h-\Gamma'_h)^2\mid \pi,B=k]\\
    =& \sumh \Big[\sumk \frac{\nk-1}{n-K}\Big\{\frac{2(\nk+1)}{\nk(\nk-1)}\sumik g_{h,i\pi(i)}^2+\frac 2 \nk +\frac 2 {\nk(\nk-1)} \Gamma_{h[k]}^2\\
    &~~+\frac{2}{\nk(\nk-1)}\sum_{i,j\in \Ik,i\neq j}g_{h,i\pi(j)} g_{h,j\pi(i)} \Big\}\Big]\\
    \le & \frac{1}{n-K}\sumh \left[\sumk 2(\nk+1) B_n^2+\sumk \frac {2(\nk-1)} \nk +\sumk \frac 2 {\nk} (\nk B_n)^2 +\sumk 2(\nk-1)B_n^2\right].
    \end{align*}
    Therefore, we have
    \begin{align*}
   \E[\|D\|_2^2\mid \pi] \le & \frac{1}{n-K}\sumh \left[ 2(n+K) B_n^2+2K +2 n  B_n^2 +2 (n-K)  B_n^2\right]\\
    =&\frac {2H}{n-K}(3nB_n^2+K).
\end{align*}
Furthermore, by the law of iterated expectation, we have 
$$\E [\|D\|_2^2]=\E [\|D\|_2^2\mid \pi]\le \frac {2H}{n-K}(3nB_n^2+K).$$

\paragraph{Step 3. Bound $\mathcal  B_1^2$} 
By the Cauchy-Schwarz inequality, we have
\begin{align}\label{eq:thm fang Step3}
   \mathcal  B_1^2=\var\{\E [\|D\|_2^2\mid \Gamma]\}= \var\{\E[D_h^2\mid \Gamma]\}
    \le  H\sumh \var\{\E[D_h^2\mid \Gamma]\}.
    \end{align}
   Using the law of total probability and Lemma \ref{lemma:compute}(i), the right-hand side of \eqref{eq:thm fang Step3} equals
    \begin{align*}
     & H\sumh \var\left\{\sumk \frac{\nk-1}{n-K}\E[D_h^2\mid \Gamma,B=k]\right\}\\
    = & H\sumh \sumk (\frac{\nk-1}{n-K})^2\var\{\E[D_h^2\mid \Gamma,B=k]\}\\
     \le & H^2 \sumk (\frac{\nk-1}{n-K})^2 C\frac{B_n^4}{\nk}.
    \end{align*}
   Combining these results, we have 
    \begin{align*}
  \mathcal  B_1^2  \le  H^2\sumk \frac{\nk-1}{(n-K)^2} CB_n^4= \frac{CH^2B_n^4}{n-K}.
\end{align*}
\paragraph{Step 4. Bound $\mathcal  B_2^2$} 
Recall $X=-(n-K)D/4$ in the remark of Lemma \ref{lemma:connect pair coupling}. We have
\begin{align}\label{eq:thm fang Step4}
   \mathcal  B_2^2=&\sumh\sum_{l=1}^H \var\{\E [X_h D_l\mid \Gamma]\}=\left(\frac{n-K}{4}\right)^2\sumh\sum_{l=1}^H\var\{\E [D_h D_l\mid \Gamma]\}.
   \end{align}
Using the law of total probability and Lemma \ref{lemma:compute}(i), \eqref{eq:thm fang Step4} equals
\begin{align*}
    &\left(\frac{n-K}{4}\right)^2\sumh\sum_{l=1}^H\var\left\{\sumk \frac{\nk-1}{n-K}\E [D_h D_l\mid \Gamma,B=k]\right\}\\
    =&\left(\frac{n-K}{4}\right)^2\sumh\sum_{l=1}^H\sumk\left(\frac{\nk-1}{n-K}\right)^2\var\{ \E [D_h D_l\mid \Gamma,B=k]\}\\
    \le &\frac{1}{16}H^2\sumk (\nk-1)^2\frac{CB_n^4}{\nk}\\
    \le &CH^2(n-K)B_n^4.
\end{align*}
\paragraph{Step 5. Bound $\mathcal  B_3^2$} 
We have
\begin{align}\label{eq:thm fang Step5}
   \mathcal  B_3^2=\sumh \sum_{l=1}^H\sum_{m=1}^H \var(\E [X_h D_l D_m\mid \Gamma])
   = \sumh \sum_{l=1}^H\sum_{m=1}^H \left(\frac{n-K}{4}\right)^2\var\{\E [D_h D_l D_m\mid \Gamma]\}.
   \end{align}
   Using the law of total probability and Lemma \ref{lemma:compute} (i), \eqref{eq:thm fang Step5} equals
   \begin{align*}
   &\sumh \sum_{l=1}^H\sum_{m=1}^H \left(\frac{n-K}{4}\right)^2\var\left\{\sumk \frac{\nk-1}{n-K}\E[D_h D_l D_m\mid \Gamma,B=k]\right\}\\
   = & \sumh \sum_{l=1}^H\sum_{m=1}^H \left(\frac{n-K}{4}\right)^2\sumk\left(\frac{\nk-1}{n-K}\right)^2\var\{ \E[D_h D_l D_m\mid \Gamma,B=k]\}\\
   \le  & H^3 \frac{1}{16}\sumk (\nk-1)^2\frac{CB_n^6}{\nk}\\
   \le & CH^3(n-K)B_n^6.
\end{align*}

Therefore, we have
$$\|X\|_2\le \alpha C(n-K)\sqrt{H} B_n,\quad \|D\|_2\le C\sqrt{H}B_n,\quad \E[\|D\|_2^2]\le \frac{CH}{n-K}(nB_n^2+K),$$
$$\mathcal B_1\le CHB_n^2(n-K)^{-1/2},\quad \mathcal B_2\le CHB_n^2(n-K)^{1/2},\quad\mathcal B_3\le CH^{3/2}(n-K)^{1/2}B_n^3.$$

Applying Lemma \ref{lemma:fang2015 thm2.1}, we have  
\begin{align*}
   & \sup_{V\in \mathcal C}|\pr(\Gamma\in V)-\pr(\xi_H \in V)|\\
    \le& c(H^{7/4}\alpha \E[\|D\|_2^2]+H^{1/4}\beta+H^{7/8}\alpha^{1/2}\mathcal B_1^{1/2}+H^{3/8}\mathcal B_2+H^{1/8}B_3^{1/2})\\
    \le& C H^{13/4}B_n(nB_n^2+K)+CH^{3/4}B_n+CH^{13/8}(n-K)^{1/4}B_n^{3/2}+CH^{11/8}(n-K)^{1/2}B_n^2.
\end{align*}

\end{proof}

\section{Proof of the Results in Supplementary Material}\label{sec:proof-supp}

\subsection{Proof of Proposition~\ref{prop:stein pair}}

\begin{proof}
We start by examining the difference between $W_{A,\pi''}$ and $W_{A,\pi}$:
\begin{align*}
W_{A,\pi''}-W_{A,\pi} &= (a_{I\pi(j)}+a_{J\pi(i)})-(a_{i\pi(i)}+a_{J\pi(j)}).
\end{align*}

We calculate the conditional expectation as follows:
\begin{align*}
\E [W_{A,\pi''}-W_{A,\pi}\mid \pi] &= \sumk \pr(B=k)\E[W_{A,\pi''}-W_{A,\pi}\mid \pi,B=k]\\
&= \sumk\pr(B=k) \cdot 2\left(-\frac{1}{\nk}\sumik a_{i,\pi(i)}+\frac{1}{\nk(\nk-1)}\sum_{i,j\in\Ik ,i\neq j}a_{i,\pi(j)} \right).
\end{align*}
Because 
$$\sum_{i,j\in\Ik ,i\neq j}a_{i,\pi(j)}=\sumijk a_{ij}-\sumik a_{i,\pi(i)},$$
we have
\begin{align*}
    \E [W_{A,\pi''}-W_{A,\pi}\mid \pi]=&\sumk\frac{\nk-1}{n-K} \cdot 2\left(-\frac{1}{\nk-1}\sumik a_{i,\pi(i)}+\frac{1}{\nk(\nk-1)}\sumijk a_{ij} \right)\\
    =&-\frac{2}{n-K}W_{A,\pi}+\frac{2}{n-K}\mu_A.
\end{align*}

By the definition of random  transposition, $(W_{A,\pi},W_{A,\pi''})$ and $(W_{A,\pi''},W_{A,\pi})$ have the same distribution. Therefore, we conclude that $(W_{A,\pi},W_{A,\pi''})$ forms a $2/(n-K)$-Stein pair under $\mu_A=0$.
\end{proof}

\subsection{Proof of Theorem \ref{thm:bound for delta and B02}}

\subsubsection{Proof of Theorem \ref{thm:bound for delta and B02}(i)}
We have 
\begin{align*}
        &\E[|W_{A,\pi''}-W_{A,\pi}|^3]\\
    =&\sumk \pr(B=k)\E[|W_{A,\pi''}-W_{A,\pi}|^3\mid B=k]\\
    = & \sumk \frac{\nk-1}{n-K}\frac{1}{\nk(\nk-1)}\sumijk \E[|a_{i\pi(i)}+a_{j\pi(j)}-a_{i\pi(j)}-a_{j\pi(i)}|^3].
\end{align*}

From 
\begin{align*}
   &|a_{i\pi(j)}+a_{j\pi(i)}-a_{i\pi(i)}-a_{j\pi(j)}|^3\\
   \le& 16(|a_{i\pi(j)}|^3+|a_{j\pi(i)}|^3+|a_{i\pi(i)}|^3+|a_{j\pi(j)}|^3), 
\end{align*}
we have
\begin{align*}
    &\E[|W_{A,\pi''}-W_{A,\pi}|^3]\\
    \le &\sumk\frac{16}{\nk(n-K)} \sumijk \E\left[|a_{i\pi(i)}|^3+|a_{j\pi(j)}|^3+|a_{i\pi(j)}|^3+|a_{j\pi(i)}|^3\right]\\
    = & \sumk\frac{64}{\nk(n-K)} \sumijk |a_{ij}|^3\\
     \le & \frac{64}{n-K}  \sumk {\sumijk |a_{ij}|^3}/{\nk}.
\end{align*}
Therefore, 
\begin{align*}
    \frac{\E[|W_{A,\pi''}-W_{A,\pi}|^3]}{\lambda}\le32 \sumk {\sumijk |a_{ij}|^3}/{\nk}.
\end{align*}

\subsubsection{Proof of Theorem \ref{thm:bound for delta and B02}(ii)}
Define $$\delta=\E[|W_{A,\pi''}-W_{A,\pi}|^3]/\lambda,\quad B_0^2=\var\left(\E [ \frac{1}{2\lambda}(W_{A,\pi}'-W_{A,\pi})^2 \bold 1(|W_{A,\pi}'-W_{A,\pi}|\le \delta )\mid W_{A,\pi}]\right)$$ and
\begin{align*}
    \alpha_{ij}^\pi=\{a_{i\pi(i)}+a_{j\pi(j)}-a_{i\pi(j)}-a_{j\pi(i)}\}^2\bold {1}\{|a_{i\pi(i)}+a_{j\pi(j)}-a_{i\pi(j)}-a_{j\pi(i)}|\le \delta\}.
\end{align*}
Therefore, 
\begin{align*}
    \frac{1}{2\lambda}\E[(W_{A,\pi}'-W_{A,\pi})^2\bold {1}(|W_{A,\pi}'
-W_{A,\pi}|\le \delta)\mid \pi]
= \frac{n-K}{4}\E[ \alpha_{IJ}^\pi\mid \pi].
\end{align*}

The law of total probability implies
\begin{align*}
 \E[\alpha_{IJ}^\pi\mid \pi]
= \sumk\pr (B=k)\E[\alpha_{IJ}^\pi\mid \pi,B=k].
\end{align*}

Therefore, 
\begin{align}\label{eq:2lambda}
  &\frac{1}{2\lambda}\E[(W_{A,\pi}'-W_{A,\pi})^2\bold {1}(|W_{A,\pi}'
-W_{A,\pi}|\le \delta)\mid \pi]\notag\\
=&\frac{n-K}{4}\Big(\sumk \frac{\nk-1}{n-K}\frac{1}{\nk (\nk-1) } \sumijk   \alpha^\pi_{ij} \bold 1(i\neq j) \Big)\notag\\
=& \sumk \frac{1}{4\nk}\sumijk \alpha_{ij}^\pi \bold 1(i\neq j).
\end{align}
By \eqref{eq:2lambda}, we have
\begin{align*}
    B_0^2=&\var\left(\E \left[ \frac{1}{2\lambda}(W_{A,\pi}'-W_{A,\pi})^2 \bold 1(|W_{A,\pi}'-W_{A,\pi}|\le \delta )\mid W_{A,\pi}\right]\right)\\
    \le& \var\left(\E \left[ \frac{1}{2\lambda}(W_{A,\pi}'-W_{A,\pi})^2 \bold 1(|W_{A,\pi}'-W_{A,\pi}|\le \delta )\mid \pi\right]\right)\\
    =&\sumk \frac{1}{16\nk^2}\var\left\{\sumijk \alpha_{ij}^\pi \bold 1(i\neq j)\right\}\\
    =&\sumk\frac{1}{16\nk^2}\Big\{2 \sum_{i<j\in \Ik} \var(\alpha_{ij}^\pi)+ \sum_{i,i',j,j'\in \Ik,\atop i\neq j,i'\neq j',|\{i,i',j,j'\}|=3}\cov(\alpha_{ij}^\pi,\alpha_{i'j'}^\pi) \\
    &~~~~~~~+ \sum_{i,i',j,j'\in \Ik,\atop i\neq j,i'\neq j',|\{i,i',j,j'\}|=4}\cov(\alpha_{ij}^\pi,\alpha_{i'j'}^\pi) \Big\}\\
    =& \sumk (\mathcal T_{1[k]}+\mathcal T_{2[k]}+\mathcal T_{3[k]}),
\end{align*}
where 
\begin{align*}
\mathcal T_{1[k]} =& \frac{2}{16\nk^2}   \sum_{i<j\in \Ik} \var (\alpha_{ij}^\pi),\\
\mathcal T_{2[k]} =& \frac{1}{16\nk^2}   \sum_{i,i',j,j'\in \Ik,\atop i\neq j,i'\neq j',|\{i,i',j,j'\}|=3}\cov(\alpha_{ij}^\pi,\alpha_{i'j'}^\pi),\\
\mathcal T_{3[k]} =& \frac{1}{16\nk^2}   
  \sum_{i,i',j,j'\in \Ik,\atop i\neq j,i'\neq j',|\{i,i',j,j'\}|=4}\cov(\alpha_{ij}^\pi,\alpha_{i'j'}^\pi),
\end{align*}

Then 
\begin{align}\label{eq:T1}
    |\mathcal T_{1[k]}|\le & \frac{1}{8\nk^2}\sumijk \delta \E[|a_{i\pi(i)}+a_{j\pi(j)}-a_{i\pi(j)}-a_{j\pi(i)}|^3]\notag\\
    \le & \frac{8\delta}{\nk^2}\sumijk  |a_{ij}|^3.
\end{align}

We use $\cov(X,Y)\le (\var(X)+\var(Y))/2$ to show that under the restriction that $i\neq j$, $i'\neq j'$, and $|\{i,j,i'j'\}|=3$, we have 
\begin{align}\label{eq:T2}
    |\mathcal T_{2[k]}|\le & \frac{\delta}{16\nk^2} \sum_{i,i',j,j'\in \Ik,\atop i\neq j,i'\neq j',|\{i,i',j,j'\}|=3}\E[|a_{i\pi(i)}+a_{j\pi(j)}-a_{i\pi(j)}-a_{j\pi(i)}|^3]\notag\\
    \le & \frac{\delta}{\nk^2} \sum_{i,i',j,j'\in \Ik,\atop i\neq j,i'\neq j',|\{i,i',j,j'\}|=3}(\E[|a_{i\pi(i)}|^3]+\E[ |a_{j\pi(j)}|^3]+\E [|a_{i\pi(j)}|^3]+\E [|a_{j\pi(i)}|^3])\notag\\
    =&  \frac{ 8 \delta \{\nk-2\}}{\nk^2}\left[\{\nk-1\}\sum_{i\in \Ik} \E [|a_{i\pi(i)}|^3]+\sum_{i\neq j\in \Ik} \E [|a_{i\pi(j)}|^3] \right]\notag\\
    \le & \frac{ 16\delta }{\nk}\sumijk |a_{ij}|^3.
\end{align}

For $|\{i,j,i',j'\}|=4$, we obtain the following bound based on the computation in \citet[(2.21)]{chen2015error}:
\begin{align}\label{eq:T3}
    |\mathcal T_{3[k]}|\le & \frac{\delta}{2\nk^3}\sum_{i,j,k,t\in \Ik} |a_{ik}+a_{jt}-a_{it}-a_{jk}|^3\notag\\
    \le & \frac{32 \delta }{\nk}\sumijk |a_{ij}|^3.
\end{align}
Based on \eqref{eq:T1}, \eqref{eq:T2} and \eqref{eq:T3}, we can bound $B_0^2$ as
\begin{align*}
    B_0^2 \le & \sumk\left(\frac{8\delta}{\nk^2}+\frac{ 16\delta }{\nk}+ \frac{32 \delta }{\nk} \right)\sumijk |a_{ij}|^3\notag\\
    \le & \sumk\frac{ 56\delta }{\nk} \sumijk |a_{ij}|^3\\
    \le & 32\times 56 \left(\sumk\sumijk |a_{ij}|^3/\nk\right)^2,
\end{align*}
where the last inequality holds from Theorem \ref{thm:bound for delta and B02}(i).

\subsection{Proof of Proposition \ref{prop:L1}}
\begin{proof}
We begin by expressing the difference between \(W_{A,\pi}^*\) and \(W_{A,\pi}\) as follows:
\begin{align*}
|W_{A,\pi}^*-W_{A,\pi}| &= UW_{A,\pi^\dagger} +(1-U)W_{A,\pi^\ddagger}-W_{A,\pi}\\
&= U\sum_{k=1}^K\sumik a_{i,\pi^\dagger(i)}+(1-U)\sumk \sumik a_{i,\pi^\ddagger(i)}-\sumk \sumik a_{i,\pi(i)}\\
&= \sum_{i\in\mathcal{R}}\left(U a_{i,\pi^\dagger (i)}+(1-U)a_{i,\pi^\ddagger(i)}-a_{i,\pi(i)}\right),
\end{align*}
where $\mathcal{R}=\{I^\dagger,J^\dagger, \pi^{-1}(P^\dagger ),\pi^{-1}(Q^\dagger )\}\subset \mathcal I_{[B^\dagger]}.$
We review an upper bound on the $L_1$ distance between non-stratified linear permutation statistics and its zero-bias transformation as follows.
\begin{lemma}[Lemma 4.7 of \citet{chen2011normal}]\label{lemma:chen2011 lemma4.7}
For $\pi\sim \textup{Uniform}(\Pi_K)$ with $K=1$, we have 
\[
\E \left[|W_{A,\pi}^*-W_{A,\pi}|\right]\le \frac{\sum_{1\le i,j\le n}|a_{ij}|^3}{(n-1)\sigma_{A }^2}\left(8+\frac{28}{(n-1)}+\frac{4}{(n-1)^2}\right).
\]
\end{lemma}

We apply Lemma \ref{lemma:chen2011 lemma4.7} within each stratum and obtain the following inequality:
\[
\E \left[|W_{A,\pi}^*-W_{A,\pi}|\,\bigg|\,B^\dagger=k\right] \le \frac{\sumijk|a_{ij}|^3}{(\nk-1)\sigma_{A[k]  }^2}\left(8+\frac{28}{(\nk-1)}+\frac{4}{(\nk-1)^2}\right).
\]

Now, we can compute the unconditional expectation by incorporating the conditional expectation:
\begin{align*}
\E [|W_{A,\pi}^*-W_{A,\pi}|] &\le \sumk \frac{\sigma_{A[k]  }^2}{\sigma^2_A}\frac{\sumijk|a_{ij}|^3}{(\nk-1)\sigma_{A[k]  }^2}\left(8+\frac{28}{(\nk-1)}+\frac{4}{(\nk-1)^2}\right)\\
&= \sumk \frac{\sumijk|a_{ij}|^3}{(\nk-1)\sigma_A^2}\left(8+\frac{28}{(\nk-1)}+\frac{4}{(\nk-1)^2}\right)\\
&\le 80 \sum_{k=1}^K \frac{\beta_{A[k]  }}{\nk}.
\end{align*}

This completes the proof.

\end{proof}

\subsection{Proof of Lemma \ref{lemma:truncation}}
\subsubsection{Upper bound on $\pr(W_{A,\pi}\neq W_{A',\pi})$}
    Let $\Lambda_{[k]  }=\{(i,j)\in  \Ik ^2:|a_{ij}|>1/2\}$, $\Lambda_{[k]  i}=\{j\in\Ik :(i,j)\in \Lambda_{[k]  }\}$, and $\Lambda =\cup_{k=1}^K \Lambda_{[k]  }$. We have 
$$|\Lambda_{[k]  }|= \sum_{i,j\in \Ik }\bold {1}(|a_{ij}|>1/2)\le 8\sumijk|a_{ij}|^3=8\beta_{A[k]  }.$$

Then we have 
$$\{W_{A,\pi}\neq W_{A',\pi}\}\subset \cup_{i=1}^n \{(i,\pi(i))\in \Lambda\},$$
which implies
\begin{align*}
    \pr(W_{A,\pi}\neq W_{A',\pi})\le & \E\left[\sumk\sum_{i\in\Ik }\bold{1}[(i,\pi(i))\in \Lambda]\right]=\sumk \sum_{i\in \Ik }|\Lambda_{[k]  i}|/\nk\\
    =&\sumk |\Lambda_{[k]  }|/\nk\le 8\sumk\frac{\beta_{A[k]  }}{\nk}.
\end{align*}
\subsubsection{Upper bound on $|\mu_{A'}|$}\label{subsubsec: bound for mu_A'}
From H\"older's inequality, for all $r\in(0,3]$, we have
\begin{align}\label{eq:holder r}
    \sum_{i,j\in{\Lambda_{[k]  } }}|a_{ij}|^r\le |\Lambda_{[k]  }|^{1-r/3}\left(\sumijk|a_{ij}|^3\right)^{r/3}\le c_1\beta_{A[k]  }.
\end{align}
Similarly, because 
$$|\Lambda_{[k]  i}|=\sum_{j\in \Ik }\bold {1}(|a_{ij}|>1/2)\le 8\sum_{j\in\Ik } |a_{ij}|^3,$$
we have 
$$\left|\sum_{j\in \Lambda_{[k]  i}}a_{ij}\right|\le |\Lambda_{[k]  i}|^{2/3}(\sum_{j\in \Lambda_{[k]  i}}|a_{ij}|^3)^{1/3}\le 4\sum_{j\in \Ik }|a_{ij}|^3\le c_1\beta_{A[k]  }.$$
Then since $\sumijk a_{ij}=0$, we have
$$|\mu_{A'}|=\left|\sumk\frac 1 {\nk}\sumijk a_{ij}'\right|=\left|\sumk\frac 1 {\nk}\sum_{(i,j)\in \Lambda_{[k]  }}a_{ij}\right|\le c_1\sumk\frac{\beta_{A[k]  }}{\nk}.$$
\subsubsection{Upper bound on $|\sigma_{A'}^2-1|$}
Recall $\sumk \sigma_{A[k]}^2=1$ and $\sigma_{A[k]}^2=(\nk-1)^{-1}\sumijk(a_{ij}^2-\bar a_{[k]i
\cdot}^2-\bar a_{[k]
\cdot j}^2+\bar a_{[k]\cdot
\cdot}^2)$. We have
\begin{align}\label{eq:truncation sigmaA' 1}
    |\sigma_{A'}^2-1|\le& \sumk|\sigma_{A'[k] }^2-\sigma_{A[k]  }^2|\notag\\
    =&\sumk \frac{1}{\nk-1}|\sumijk (a'_{ij})^2-\sumijk (\bar a'_{[k]i\cdot})^2-\sumijk (\bar a'_{[k]\cdot j})^2+\sumijk (\bar a'_{[k]\cdot\cdot})^2+\sumijk a_{ij}^2|\notag\\
    \le& \sumk \frac{1}{\nk-1}\left(\sum_{i,j\in\Lambda_{[k]  }}a_{ij}^2+\sumijk \bar a'^2_{[k]  i\cdot}+\sumijk \bar a'^2_{[k]  \cdot j}+\sumijk \bar a'^2_{[k]  \cdot\cdot}\right).
\end{align} 
By \eqref{eq:holder r} with $r=2$, we bound the first term of the right-hand side of \eqref{eq:truncation sigmaA' 1}  as
$$\frac{1}{\nk-1}\sum_{i,j\in \Lambda_{[k]  }} a_{ij}^2\le \frac{2c_1\beta_{A[k]  }}{\nk}.$$

For the second term of the right-hand side of \eqref{eq:truncation sigmaA' 1}, we have $$|\bar a'_{[k] i\cdot}|=|\frac{1}{\nk}\sum_{j\in\Ik }a'_{ij}|=|\frac{1}{\nk}\sum_{j\in\Lambda_{[k]  i}}a_{ij}| \le \frac{4}{\nk}\sum_{j\in\Ik }|a_{ij}|^3\le\frac{4\beta_{A[k]  }}{\nk}.$$

Therefore, we have
$$\frac{1}{\nk-1}\sumijk \bar a'^2_{[k]  i\cdot}\le \frac{4\beta_{A[k]  }}{\nk-1}\sum_{i\in\Ik }|\bar a'_{[k] i\cdot}|\le \frac{16\beta_{A[k]  }}{\nk(\nk-1)}\sumijk|a_{ij}|^3\le 32\beta_{A[k]  }^2/\nk^2,$$
with the same bound holding when $i$ and $j$ are interchanged. In addition, by Section \ref{subsubsec: bound for mu_A'}, we have $$|\bar a'_{[k] \cdot\cdot}|=|\mu_{A'[k]}|/\nk\le c_1\beta_{A[k]  }/\nk^2,$$
and so $$\frac{1}{\nk-1}\sumijk \bar a'^2_{[k]  \cdot\cdot}\le \frac{\nk^2}{\nk-1}\frac{c_1^2\beta_{A[k]  }^2}{\nk^4}\le 2c_1^2\beta_{A[k]  }^2/\nk^3.$$
Hence
$$|\sigma_{A'[k] }^2-\sigma_{A[k]  }^2|\le \frac{\beta_{A[k]  }}{\nk}(2c_1+64\beta_{A[k]  }/\nk+2c_1^2\beta_{A[k]  }/\nk^2).$$

For any $\varepsilon_1<1$, take $c_2=2c_1+64+2c_1^2$. If we further require
$\varepsilon_1 \in(0,1/(3c_2))$, then when $\sumk \beta_{A[k]  }/\nk\le \varepsilon_1$, we have
\begin{align}\label{eq:sigma A' 1 bound 1/3}
    |\sigma^2_{A'}-1|\le \sumk |\sigma_{A'[k] }^2-\sigma_{A[k]  }^2|\le \frac 1 3 \sumk  \frac{\beta_{A[k]  }}{\nk} \le \frac 1 3.
\end{align}

\subsubsection{Proof of $(A')^\textup{s}\in M_n^1$ and $\beta_{A'[k] }\le c_2\beta_{A[k]  }$}
By \eqref{eq:sigma A' 1 bound 1/3}, we have $\sigma_{A'}^2>2/3$, which implies $\sigma_{A'}>2/3$. When $\sumk \beta_{A[k]  }/\nk\le \varepsilon_1$, the elements of $({A'})^\textup{s}$ satisfy 
$$|a'_{ij}-\bar a'_{[k] i\cdot}-\bar a'_{[k] \cdot j}+\bar a'_{[k] \cdot\cdot}|/\sigma_{A'}\le \frac 3 4+\frac 3 2 (|\bar a'_{[k]i\cdot}|+|\bar a'_{[k]\cdot j}|+|\bar a'_{[k]\cdot\cdot}|),$$
 hence considering the bound of $|\bar a'_{[k] i\cdot}|$, $|\bar a'_{[k] \cdot j}|$ and $|\bar a'_{[k] \cdot\cdot}|$, we have that there exists $\varepsilon_1$ small enough such that the elements of $({A'})^\textup{s}$ are all bounded by $1$, hence $({A'})^\textup{s}\in M_n^1$.
 Lastly,  we have $$\beta_{A'[k] }=\frac{\sumijk |a'_{ij}|^3}{\sigma_{A'}^3}\le c_2\beta_{A[k]  }.$$

\subsection{Proof of Lemma \ref{lemma:scaling}}
	From Lemma \ref{lemma:truncation}, we have 
	\begin{align*}
		&\sup_{t\in \mathbb R}|\pr(W_{A,\pi}\le t)-\Phi(t)|\\
		\le & \sup_{t\in \mathbb R}|\pr(W_{A',\pi}\le t)-\Phi(t)|+\sup_{t\in \mathbb R}|\pr(W_{A',\pi}\le t)-\pr(W_{A,\pi}\le t)|\\
		\le & \sup_{t\in \mathbb R}|\pr(W_{A',\pi}\le t)-\Phi(t)|+c_1\sumk\beta_{A[k]  }/\nk \\
		\le & \sup_{t\in \mathbb R}|\pr(W_{A',\pi}\le t)-\Phi(\frac{t-\mu_{A'}}{\sigma_{A'}})|+\sup_{t\in \mathbb R}|\Phi(\frac{t-\mu_{A'}}{\sigma_{A'}})-\Phi(t)|+c_1\sumk\beta_{A[k]  }/\nk \\
		\le & \sup_{t\in \mathbb R}|\pr(W_{({A'})^\textup{s}}\le t)-\Phi(t)|+\sup_{t\in \mathbb R}|\Phi(\frac{t-\mu_{A'}}{\sigma_{A'}})-\Phi(t)|+c_1\sumk\beta_{A[k]  }/\nk.
	\end{align*}
	Hence we only need to prove that there exists some $c_{2}$ such that $$\sup_{t\in \mathbb R}|\Phi(\frac{t-\mu_{A'}}{\sigma_{A'}})-\Phi(t)|\le c_{2}\sumk\beta_{A[k]  }/\nk.$$
	
	From the proof of Lemma \ref{lemma:truncation}, we have $|\sigma^2_{A'}-1|\le 1/3$, which implies $\sigma_{A'}\in[2/3,4/3]$. Then we consider the following two cases to prove the lemma.
    \subsubsection{When $|t|\ge c_1\sumk\beta_{A[k]  }/\nk$}
We can show that $|t\exp(-at^2/2)|\le 1/\sqrt{a}$ for $a>0.$
		Hence \begin{align*}
			 \left|t\exp\big(-\frac{9}{32}(t-\mu_{A'})^2\big)\right|\le &\left|(t-\mu_{A'})\exp\big(-\frac{9}{32}(t-\mu_{A'})^2\big)\right|+|\mu_{A'}|\\
		\le & \frac 4 3+|\mu_{A'}| \le  \frac 4 3(1+|\mu_{A'}|).
		\end{align*}
		Since $\sigma_{A'}\ge 2/3$, from Lemma \ref{lemma:truncation}, we have		$$|\sigma_{A'}-1|=\frac{\sigma_{A'}^2-1}{\sigma_{A'}+1}\le c_2\sumk \beta_{A[k]  }/\nk.$$		
Because $|\mu_{A'}|\le c_1\sumk\beta_{A[k]  }/\nk$, we know that $t$ and $(t-\mu_{A'})/\sigma_{A'}$ are on the same side of the origin.
		Under the same logic of \citet[page 179]{chen2011normal}, we obtain
		$$\left|\Phi(\frac{t-\mu_{A'}}{\sigma_{A'}})-\Phi(t)\right|\le \frac{2}{\sqrt{2\pi}}|\sigma_{A'}-1|(1+|\mu_{A'}|)+\frac 3 4 |\mu_{A'}|.$$
	From Lemma \ref{lemma:truncation}, $|\mu_{A'}|\le c_1\sumk \beta_{A[k]  }/\nk\le c_1\varepsilon_2$ and further
		\begin{align*}
			& \frac{2}{\sqrt{2\pi}}|\sigma_{A'}-1|(1+|\mu_{A'}|)+\frac 3 4 |\mu_{A'}|\\
			\le & \frac{2c_2}{\sqrt{2\pi}}\sumk \beta_{A[k]  }/\nk (1+c_1\varepsilon_2)+\frac {3c_1}{4}\sumk \beta_{A[k]  }/\nk\\
			=&\left\{\frac{2c_2}{\sqrt{2\pi}}(1+c_1\varepsilon_2)+\frac {3c_1}{4}\right\}\sumk \beta_{A[k]  }/\nk.
		\end{align*}
		\subsubsection{When $|t|< c_1\sumk\beta_{A[k]  }/\nk$}
We have that $(t-\mu_{A'})/\sigma_{A'}$ lies in the interval $[3(t-\mu_{A'})/2,3(t-\mu_{A'})/4]$. It implies that
 $$|(t-\mu_{A'})/\sigma_{A'}|\le \frac{3(|t|+|\mu_{A'}|)}{2},$$
then we obtain
 \begin{align*}
\left|\Phi(\frac{t-\mu_{A'}}{\sigma_{A'}})-\Phi(t)\right|
	        	\le  \frac{1}{\sqrt{2\pi}}\left|\frac{t-\mu_{A'}}{\sigma_{A'}}-t\right|
	        	\le  \frac{1}{\sqrt{2\pi}}\left(3|t|+2|\mu_{A'}|\right)
	        	\le  \frac{5c_1}{\sqrt{2\pi}}\sumk \beta_{A[k]  }/\nk .
	        \end{align*}

\subsection{Proof of Lemma \ref{lemma:inductive}}
\begin{proof}
Let $\mathcal C=\pi(\mathcal R)$ and $m=\nk-l$. When $m=0$, we have $\mu_{D[k]}=0$ for $k=1,\ldots, K$ and thus $\mu_D=0$. When $m\ge 1$,
    $$|\bar d_{[k] i\cdot}|=\frac 1 m\left|\sum_{j\in\Ik \setminus \mathcal C}d_{ij}\right|=\frac 1 m \left|\sum_{j\in\Ik \setminus \mathcal C}a_{ij}\right|= \frac 1 m \left|\sum_{j\in\mathcal C}a_{ij}\right|\le \frac {|\mathcal C|} m \le \frac 4 m$$
with the same bound holding when the roles of $i$ and $j$ are interchanged. Similarly,  as $\bar a_{ [k]\cdot\cdot}=0,$
\begin{align*}
    |\bar d_{[k]\cdot\cdot}|=&\frac{1}{m^2}|\sum_{i\in \Ik \setminus \mathcal R\atop j\in \Ik \setminus \mathcal C} d_{ij}|=\frac{1}{m^2}|\sum_{i\in  \mathcal R\atop j\in  \mathcal C} a_{ij}|\le \frac{|\mathcal R|+|\mathcal C|}{m}\le \frac 8 m.
\end{align*}

Then $\mu_D=\sum_{k'=1}^K\mu_{D[k']}=\mu_{D[k] }=m \bar d_{[k]\cdot\cdot}\le 8$.
When $m=0,1$, we have $$\sigma_D^2=\sum_{k'=1}^K\sigma_{D[k']}^2=\sum_{k'=1}^K\sigma_{A[k']}^2=\sigma_A^2=1.$$ When $m\ge 2$, we have 
\begin{align*}
    \sigma_D^2 =&\sigma_{D[k] }^2+\sum_{k'\neq k} \sigma_{D[k']}^2\\
    =&\sigma_{D[k] }^2-\frac{1}{m-1}\sumijk a_{ij}^2+\sum_{k'\neq k} \sigma_{A[k']}^2+\frac{1}{m-1}\sumijk a_{ij}^2\\
    \le &\frac{1}{m-1}\left(\Big|\sum_{i\in \Ik \setminus \mathcal R\atop j\in \Ik \setminus \mathcal C}a_{ij}^2-\sumijk a_{ij}^2 \Big|+m\sum_{i\in \Ik \setminus \mathcal R } \bar d_{[k] i\cdot}^2+m\sum_{j\in \Ik \setminus \mathcal C } \bar d_{[k] \cdot j}^2+m^2 \bar d_{[k] \cdot\cdot}^2 \right)+\frac{\nk-1}{m-1}\\
    \le & \frac{1}{m-1}\left(\sum_{i\in \mathcal R,j\in \mathcal C} a_{ij}^2 +96\right)+\frac{\nk-1}{m-1}\\
    \le & \frac{1}{m-1}\left(8\varepsilon_3^{2/3}\nk +96\right)+\frac{\nk}{m-1},
\end{align*}
where the last equation holds due to H\"older's inequality that 
$$\sumjk a_{ij}^2 \le \nk^{1/3} \left(\sumjk |a_{ij}|^3\right)^{2/3}\le \nk^{1/3}\beta_{A[k]  }^{2/3}.$$
The same inequality holds when the roles of $i$ and $j$ are interchanged, and so, when $\beta_{A[k]  }/\nk\le \varepsilon_3,$
$$\sum_{i\in\mathcal R,j\in \mathcal C} a_{ij}^2\le 
 \sum_{i\in \mathcal R} \sumjk a_{ij}^2+\sum_{j\in \mathcal C}\sumik a_{ij}^2\le 2 l\cdot n^{1/3}\beta_{A[k]  }^{2/3}\le 8\varepsilon_3^{2/3}\nk.$$

Since $m\ge 2$ and $m=\nk-l,l\in\{2,3,4\}$, we have $\nk/(m-1)=1+(l+1)/(m-1)\le 6$.
Therefore, $$\sigma_D^2\le 48\varepsilon_3^{2/3}+102.$$
    
\end{proof}

\subsection{Proof of Lemma \ref{lemma:weight S}}
\begin{proof}
We have
    \begin{align*}
        \max_{1\le k\le K} S_{[k]}\left(\sumk w_{[k]}S_{[k]}\right)\ge \sumk w_{[k]}S_{[k]}^2,\quad \left(\sumk w_{[k]}S_{[k]}\right)\ge \min_{1\le k\le K} S_{[k]},
    \end{align*}
    which lead to Lemma \ref{lemma:weight S}.
\end{proof}

\subsection{Proof of Lemma \ref{lemma:hyper}}
We prove a CLT for $n_{[k]1}/\{n_{[k]1}+\bold 1 (n_{[k]1}=0)\}$ and verify the uniformly integrable to obtain the order of $\var(\nk/n_{[k]1}\mid \mathcal D_1)$.

\subsubsection{CLT for $n_{[k]}/\{n_{[k]1}+\bold 1 (n_{[k]1=0})\}$}
By Lemma \ref{lemma:hyper_beb}, if $ \pik\rightarrow \pik^*\in(0,1)$ and $n_1/n\rightarrow p\in(0,1)$ as $n\rightarrow \infty$, 
    then $$ \frac{\sqrt{\nk(n-\nk)n_1(n-n_1)/n^3}}{10\max\left\{\frac{\min(n_1/n,1-n_1/n)+4}{4(1-\min(n_1/n,1-n_1/n))},2\right\}}\rightarrow\infty,$$
and $$\frac{1}{\nk}\sigma^2_{\mathrm{hyp}}= (n-\nk)n_1(n-n_1)/n^3 \rightarrow (1-\pik^*)p(1-p)\in(0,1) $$
and for the normal approximation result, because the exponential function converges faster than the polynomial function, then we have 
$$\sup_{t\in\mathbb R}\left|\pr\left(\frac{n_{[k]1}-n_1\pik}{\sigma_{\mathrm{hyp}}}\le t\right)-\Phi(t)\right|=O\left(\frac{1}{\sigma_{\mathrm{hyp}}}\right)=O\left(\frac{1}{\sqrt{\nk}}\right).$$
The Kolmogorov distance tends to zero as $n\rightarrow \infty.$
Therefore, 
\begin{align}\label{eq:clt for nk inver hyp}
    \sqrt{\nk}\left(\frac{n_{[k]1}}{\nk}-p\right)\xrightarrow{d}\mathcal N(0,(1-\pik^*)p(1-p)).
\end{align}
Furthermore, 
$$\pr\left\{\frac{\bold 1(n_{[k]1}=0)}{\sqrt{\nk }}\neq 0\right\}=\frac{n-n_1}{n}\frac{n-n_1-1}{n}\cdots\frac{n-n_1-(\nk-1)}{n-(\nk-1)}\le (\frac{n-n_1}{n})^{\nk}\rightarrow 0.$$
Therefore, $\bold 1(n_{[k]1}=0)/{\sqrt{\nk }}\xrightarrow{\pr} 0.$ Combine \eqref{eq:clt for nk inver hyp}, we have
$$\sqrt{\nk}\left\{\frac{n_{[k]1}+\bold 1(n_{[k]1}=0)}{\nk}-p\right\} \xrightarrow{d}\mathcal N(0,(1-\pik^*)p(1-p)).$$
By delta method with $g(x) = 1/x$, we have 
\begin{align}\label{eq:clt for nk/nk1}
    \sqrt{\nk}\left\{\frac{\nk}{n_{[k]1}+\bold 1(n_{[k]1}=0)}-\frac 1 p\right\}\xrightarrow{d}\mathcal N\left\{0,\frac{(1-\pik^*)(1-p)}{p^3}\right\}.
\end{align}

\subsubsection{Uniformly Integrable Property}
We will show that
\begin{align}\label{eq:UI}
  \lim_{b\rightarrow \infty} \sup_{(n_{[k]1})_{k=1}^K}  \E\left[n \left\{\frac{\nk}{n_{[k]1}+\bold 1(n_{[k]1}=0)}-\frac{1}{p}\right\}^2 \bold 1\left\{\left|\frac{\nk}{n_{[k]1}+\bold 1(n_{[k]1}=0)}-\frac{1}{p}\right|>b\right\}\right]=0.
\end{align}
When $b> p^{-1}$, we only need to consider $n_{[k]}(n_{[k]1}+\bold 1(n_{[k]1}=0))^{-1}> p^{-1}+b$.
Then we have that 
\begin{align*}
    &\E\left[n \left\{\frac{\nk}{n_{[k]1}+\bold 1(n_{[k]1}=0)}-\frac{1}{p}\right\}^2 \bold 1\left\{\left|\frac{\nk}{n_{[k]1}+\bold 1(n_{[k]1}=0)}-\frac{1}{p}\right|>b\right\}\right]\\
    \le &\E\left[n b^2 \bold 1\left\{\left|\frac{\nk}{n_{[k]1}+\bold 1(n_{[k]1}=0)}-\frac{1}{p}\right|>b\right\}\right]\\
    \le & n b^2 \pr\left\{\frac{\nk}{n_{[k]1}+\bold 1(n_{[k]1}=0)}>b\right\}.
\end{align*}
Because $$\frac{\nk}{n_{[k]1}+\bold 1(n_{[k]1}=0)}\le \nk,$$
for fixed $b$, we only consider $\nk>b$ to make the probability non-zero.

Because 
\begin{align*}
    \pr\left(n_{[k]1}< \frac \nk b \right)=& \pr\left(\text{there exist }[(1-\frac 1 b)\nk]  \text{ control units in the $k$-th stratum}\right)\\
    \le & \binom{\nk}{[(1-\frac 1 b)\nk]}\frac{n_0}{n}\frac{n_0-1}{n-1}\cdots\frac{n_0-[(1-\frac 1 b)\nk]+1}{n-[(1-\frac 1 b)\nk]+1}\\
    \le & \left(\frac{\nk}{[(1-\frac 1 b)\nk]}\right)^{[(1-\frac 1 b)\nk]} (\frac{n_0}{n})^{[(1-\frac 1 b)\nk]}\\
    =& \left(\frac{\nk}{[(1-\frac 1 b)\nk]}\frac{n_0}{n}\right)^{[(1-\frac 1 b)\nk]}.
\end{align*}
As $\nk\rightarrow \infty$, we have that $\nk/[(1-b^{-1})\nk] \rightarrow (1-b^{-1})^{-1}<1.$ When $b\gg n/n_1,$ we have $$\frac{\nk}{[(1-\frac 1 b)\nk]}\frac{n_0}{n}<1.$$

The left-hand side of  \eqref{eq:UI} equals
\begin{align*}
   &\lim_{b\rightarrow \infty}\sup_{\nk:\nk>b} nb^2 \left(\frac{\nk}{[(1-\frac 1 b)\nk]}\frac{n_0}{n}\right)^{[(1-\frac 1 b)\nk]}\\
= &\lim_{b\rightarrow \infty}\sup_{\nk:\nk>b} \pik^{-1} \nk b^2 \left(\frac{1}{1-\frac 1 b}\frac{n_0}{n}\right)^{(1-\frac 1 b)\nk}\\
= &\lim_{b\rightarrow \infty}\sup_{\nk:\nk>b} \pik^{-1} \nk b^2 (1-p)^{\nk}\\
=&0.
\end{align*}

Because
\begin{align*}
  \var\left(\frac{\nk}{n_{[k]1}} \mid \mathcal D_1\right) =\var\left\{\frac{\nk}{n_{[k]1}+\bold 1(n_{[k]1}=0)} -\frac{1}{p}\mid \mathcal D_1\right\}
  \le  \var\left\{\frac{\nk}{n_{[k]1}+\bold 1(n_{[k]1}=0)} -\frac{1}{p}\right\},
\end{align*}
by the uniform integrability properties \eqref{eq:UI} and convergence in distribution \eqref{eq:clt for nk/nk1}, we can obtain that $$\var\left(\frac{\nk}{n_{[k]1}} \mid \mathcal D_1\right)=O\left(\frac{1}{\nk}\right)=O\left(\frac 1 n\right).$$

\subsection{Proof of Lemma \ref{lemma:converge to zero with p lar 2}}

For any positive integer $n$, by Lemma \ref{lemma:rosenthal}, we have
\begin{align*}
  \E\left[\left|\frac 1 n\sumk\sumijk  \frac {X_{ij}} {\nk} \right|^p\right]\le & C_p\left\{\sumk\sumijk  \frac{\E[|X_{ij}|^p]}{\nk^p}+ \left(\sumk\sumijk \frac{\E [X_{ij}^2]}{\nk^2}\right)^{p/2}\right\}\\
  \le & C_p \left\{ K \E[|X_{ij}|^p] + K^{p/2} (\E[X_{ij}^2])^{p/2}\right\}= O(n^{p/2}),  
\end{align*}
where the last inequality holds because $\nk^{p-2}\ge 1$.
 
Then for any positive number $\phi$ and $p>2$, we have that 
$$\sum_{n=1}^\infty \pr\left(\left|\frac 1 n\sumk\sumijk  \frac {X_{ij}} {\nk} \right|\ge \phi\right)\le \phi^{-p} \sum_{n=1}^\infty \frac{ \E\left[\left|\frac 1 n\sumk\sumijk  \frac {X_{ij}} {\nk} \right|^p\right]}{n^{p}}=O\left(\sum_{n=1}^\infty \frac{1}{n^{p/2}}\right)<\infty.$$
    Then the Borel--Cantelli Lemma implies $n^{-1}\sumk\sumijk X_{ij}/\nk \xrightarrow{a.s.} 0.$

\end{document}